\documentclass[12pt,reqno]{amsart}

\usepackage{amsmath}
\usepackage{amssymb}
\usepackage{amsthm}
\usepackage{xcolor}
\usepackage{graphicx}
\usepackage{url}
\usepackage{multirow}

\numberwithin{equation}{section}

\newtheorem{theorem}{Theorem}[section]
\newtheorem{cor}[theorem]{Corollary}
\newtheorem{lemma}[theorem]{Lemma}

\newtheorem{Conjecture}[theorem]{Conjecture}

\theoremstyle{definition}

\theoremstyle{remark}
\newtheorem{rem}[theorem]{Remark}

\theoremstyle{remark}

\renewcommand{\Re}{\operatorname{Re}}
\renewcommand{\Im}{\operatorname{Im}}
\newcommand{\LHS}{\operatorname{LHS}}
\DeclareMathOperator{\erf}{erf}

\newcommand{\R}{\mathbb{R}}
\newcommand{\Z}{\mathbb{Z}}
\newcommand{\C}{\mathbb{C}}

\def\al{\alpha}

\def\la{\lambda}

\def\Ga{\Gamma}

\def\Ct{C_0}
\def\dd{\delta}

\begin{document}

\title{An asymptotic approach to Borwein-type sign pattern theorems}
\author{Chen Wang \and Christian Krattenthaler}
\address{Fakult\"at f\"ur Mathematik, Universit\"at Wien,
Oskar-Morgenstern-Platz~1, A-1090 Vienna, Austria.}
\email{chen.wang@univie.ac.at}
\urladdr{http://www.mat.univie.ac.at/\~{}kratt}

\thanks{This work is supported by the Austrian Science Fund (FWF) grant SFB F50 (F5009-N15).}

\begin{abstract}
The celebrated (First) Borwein Conjecture predicts that for all positive integers~$n$ 
the sign pattern of the coefficients of the ``Borwein polynomial''
$$(1-q)(1-q^2)(1-q^4)(1-q^5)
\cdots(1-q^{3n-2})(1-q^{3n-1})$$ 
is $+--+--\cdots$. It was proved by the first
author in [{\it Adv.\ Math.} {\bf 394} (2022), Paper No.~108028]. 
In the present paper, we extract the essentials from the former paper
and enhance them to a conceptual approach for the proof of ``Borwein-like''
sign pattern statements. In particular, we provide a new proof of the original
(First) Borwein Conjecture, a proof of the Second Borwein Conjecture (predicting
that the sign pattern of the square of the ``Borwein polynomial'' is also 
$+--+--\cdots$), and a partial proof of a ``cubic'' Borwein Conjecture due to the
first author (predicting the same sign pattern for the cube of the ``Borwein
polynomial''). Many further applications are discussed.
\end{abstract}

\maketitle

\section{Introduction}
It was in 1993 at a workshop at Cornell University, when what became known
as {\it the Borwein Conjecture} was born. (One of the authors was an intrigued
witness of this event.) George Andrews delivered a two-part lecture on
{\it``AXIOM and the Borwein Conjecture''}, in which he --- first of all ---
stated three conjectures that had been communicated to him by Peter Borwein 
(the first of which became known as ``the Borwein
Conjecture''), and then reported the lines of attack that he had tried, all of
which had failed to give a proof, stressing (quoting from~\cite{MR1395410},
which contains Andrews' findings in printed form)
that {\it``this is the sort of intriguing simply stated problem that
devotees of the theory of partitions love.''} Indeed, the statement of the
first conjecture, dubbed the ``First Borwein Conjecture'' in~\cite{MR1395410}, 
is the following.

\begin{Conjecture}[\sc P. Borwein]\label{cjBorwein}
For all positive integers $n$, the sign pattern of the coefficients
in the expansion of the polynomial $P_n(q)$ defined by
\begin{equation}\label{eqPolynomial}
P_n(q):=(1-q)(1-q^2)(1-q^4)(1-q^5)\cdots(1-q^{3n-2})(1-q^{3n-1})
\end{equation}
is $+--+--+--\cdots$, with a coefficient $0$ being considered as both $+$
and~$-$.
\end{Conjecture}

The {\it Second Borwein Conjecture} from \cite{MR1395410} 
predicts the same sign behaviour of the
coefficients for the square of the ``Borwein polynomial''.

\begin{Conjecture}[\sc P. Borwein]\label{cjBorwein2}
For all positive integers $n$, the sign pattern of the coefficients
in the expansion of the polynomial $P_n^2(q)$, where $P_n(q)$ is defined by
\eqref{eqPolynomial},
is $+--+--+--\cdots$, with the same convention concerning zero coefficients.
\end{Conjecture}

The {\it Third Borwein Conjecture} from \cite{MR1395410} 
is an assertion on the sign behaviour of the
coefficients of a polynomial similar to $P_n(q)$, where however the involved
modulus is~$5$ instead of~$3$. We shall return to it at the end of
this paper, see Conjecture~\ref{cjBorwein3} in Section~\ref{seDiscuss}.

Interestingly, 
the first author observed recently that a cubic version of the conjecture also appears 
to hold, which both Borwein and Andrews missed.

\begin{Conjecture}[\sc C. Wang]\label{cjBorwein4}
For all positive integers $n$, the sign pattern of the coefficients
in the expansion of the polynomial $P_n^3(q)$, where $P_n(q)$ is defined by
\eqref{eqPolynomial},
is $+--+--+--\cdots$, with the same convention concerning zero coefficients
as before.
\end{Conjecture}

These deceivingly simple conjectures intrigued many researchers after
Andrews had introduced them to a larger audience --- in particular the
first one, Conjecture~\ref{cjBorwein}. Various approaches were tried
--- combinatorial, or using $q$-series techniques
(cf.\ e.g.\ \cite{MR1395410,MR4119403,MR2140441,MR1392489,MR1660081,SchlZhou,MR1874535,MR2009544,MR2220659}) ---, 
variations and generalisations were proposed (see 
\cite{MR4039550,MR1392489,MR1660081,SchlZhou}) --- most notably
Bressoud's conjecture in~\cite{MR1392489} --- sometimes leading to proofs
of related results. However, none of these attempts came anything close to
progress concerning the original First Borwein Conjecture, Conjecture~\ref{cjBorwein}.
It took almost 30 years until the first author succeeded in proving this
conjecture in~\cite{WANG2022108028}, using analytic means.

Starting point of the proof in \cite{WANG2022108028} was explicit sum representations
of the polynomials $A_n(q),B_n(q),C_n(q)$ in the decomposition
of $P_n(q)$ given by
\begin{equation} \label{eq:ABC} 
P_n(q)=A_n(q^3)-qB_n(q^3)-q^2C_n(q^3),
\end{equation}
due to Andrews \cite{MR1395410}. It should be noted that the First Borwein Conjecture,
Conjecture~\ref{cjBorwein}, is equivalent to the statement that all coefficients
of the polynomials $A_n(q),B_n(q),C_n(q)$ are non-negative.
These coefficients were written in~\cite{WANG2022108028} 
in terms of the obvious Cauchy integrals.
Subsequent saddle point approximations showed that for $n>7000$
the coefficient of $q^m$
in $A_n(q),B_n(q),C_n(q)$ is positive in the range $n< m< n^2-n$.
The proof could then be completed by appealing to
another result of Andrews~\cite{MR1395410} which gives non-negativity of the
coefficients of $q^m$ in $A_n(q),B_n(q),C_n(q)$ for $m\le n$ and $m\ge n^2-n$
``for free'', and by performing a
computer check of the conjecture for $n\le 7000$.

At this point, it must be mentioned that formulae analogous to Andrews' formulae
for the decomposition polynomials $A_n(q),B_n(q),C_n(q)$ are not available for
the analogous decompositions of $P_n^2(q)$ or $P_n^3(q)$, or for the corresponding
decomposition of the polynomial $S_n(q)$ in the Third Borwein Conjecture
(Conjecture~\ref{cjBorwein3}), and that it is unlikely that such formulae exist. 

Thus, the article~\cite{WANG2022108028} left open
the question whether it was just an isolated instance that this approach
succeeded to prove the First Borwein Conjecture, or whether similar ideas
could also lead to proofs of the Second and Third Borwein Conjecture, or of
the new Conjecture~\ref{cjBorwein4}. Admittedly, since the proof
in~\cite{WANG2022108028} relied on Andrews' sum representations
for the decomposition polynomials $A_n(q),B_n(q),C_n(q)$ in an
essential way, at the time it
did not seem very realistic to expect that, with these ideas, one could
go beyond the First Borwein Conjecture.

In the meantime, however, we realised that, instead of relying on Andrews'
sum representations for the decomposition polynomials, the saddle point
approximation idea could be directly applied to $P_n(q)$ and its powers, and, when doing
this, surprisingly the quantities that have to be approximated are very similar
to those that were at stake in~\cite{WANG2022108028}
(compare, for instance, the sum over~$m$ at the beginning of the
proof of Proposition~11.1 in \cite{WANG2022108028} with
\eqref{eq:P/P-sum} below, or
\cite[Lemma~B.3]{WANG2022108028} and Lemma~\ref{leIneqCosSum1}). There is a price to pay
though: while in~\cite{WANG2022108028} the (dominant) saddle points were located on the
real axis, with this new approach we have to deal with (dominant) saddle points located
at complex points. This makes the estimations that have to be
performed more delicate.\footnote{There is in fact a further subtlety not
present in~\cite{WANG2022108028} that makes the task of carrying
through this new approach more difficult, see Footnotes~\ref{foot:2}
and~\ref{foot:3}.}
On the positive side, it allows one to proceed in a more streamlined fashion
--- for example, here we do not have to deal with several different kinds of peaks
along the integration contour, as opposed to~\cite{WANG2022108028} where 
an unbounded number of peaks of two different kinds 
had to be considered; here we encounter only two
peaks that are (complex) conjugate to each other.
Most importantly, it allows us
to provide a {\it uniform} proof of the First {\it and\/} Second Borwein Conjecture, 
{\it as well as} a partial proof of the cubic conjecture, and
altogether this is not longer than the proof of ``just'' the First Borwein
Conjecture in~\cite{WANG2022108028}.

In the next section, we provide an outline of our proof of Conjectures~\ref{cjBorwein}
and \ref{cjBorwein2}, and of ``two thirds'' of Conjecture~\ref{cjBorwein4}.
Very roughly, the approach that we put forward consists of the following steps: 

\begin{enumerate} 
\item show that the conjectures hold for the ``first few'' and the
``last few'' coefficients (see Part~A in Section~\ref{seOutline}); 
\item represent the coefficients by a contour integral 
(see Part~B in Section~\ref{seOutline}); 
\item divide the contour into two parts, the ``peak part'' (the part
close to the dominant saddle points of the integrand) and the
remaining part, the ``tail part''
(see Part~C in Section~\ref{seOutline});
\item for ``large'' $n$ (where ``large'' is made precise),
bound the error made by approximating the ``peak part'' by a Gau\ss ian 
integral (the ``peak error'') 
(see Part~D in Section~\ref{seOutline});
\item for ``large'' $n$,
bound the error contributed by the ``tail part'' (the ``tail
error'') (see Part~D in Section~\ref{seOutline}); 
\item verify the conjectures for ``small'' $n$ (see Part~E in Section~\ref{seOutline}); 
\item put everything together to complete the proofs 
(see Part~E in Section~\ref{seOutline}).
\end{enumerate}

The details are then filled in in the subsequent sections. More precisely, 
in Section~\ref{seInfinite} we explain how prior results of Andrews,
of Kane, and of Borwein, Borwein and Garvan confirm the conjectures 
for the ``first few'' and the ``last few'' coefficients. 
Section~\ref{seConvention} prepares some notation and preliminary material on
log-derivatives of the ``Borwein polynomial'' $P_n(q)$ that is used
ubiquitously in the subsequent sections. In Section~\ref{seLocate}, 
we make our choice of contour for the integral representation precise:
it is a circle whose radius satisfies an equation, namely
\eqref{eqStationaryPoint}, that approximates the actual saddle point
equation. Lemma~\ref{leRadiusBound} presents fundamental properties
that this choice satisfies. In Section~\ref{secutoff}, we make precise
how we divide the contour into the ``peak part'' and the ``tail part''.
Lemma~\ref{leCutoffPrelim} in that section presents first properties
of this cutoff, to be used in the later parts of the paper.
The fundamental inequality that is derived from this subdivision of the
integral contour is the subject of Section~\ref{seError}. Namely,
Lemma~\ref{lem:ep0ep1} provides
a qualitative upper bound for the resulting approximation of the
coefficients of $P_n^\dd(q)$, with $\dd\in\{1,2,3\}$, in terms of a 
peak error term and a tail error term.
How to bound the peak error efficiently from above is shown in
Section~\ref{seEps0}. This section contains in particular a fundamental result on
the approximation of a (complex) function by a Gau\ss ian integral 
that may be of independent interest for other applications; see
Lemma~\ref{leGaussian}. Subsequently, Section~\ref{seEps1} is devoted
to bound the tail error from above. Finally, in Section~\ref{seMain}
we put everything together and complete the proofs of
Conjectures~\ref{cjBorwein} and~\ref{cjBorwein2}, and of ``two
thirds'' of Conjecture~\ref{cjBorwein4}.

Without any doubt, several of the arguments that we need are quite
technical. In the interest of not losing pace (too much) while guiding the reader
through our proofs, we have ``outsourced'' some of the auxiliary
results and have collected them in an appendix.

It must be emphasised though that a certain ``level of technicality''
is unavoidable since the approximations that we are carrying out here
go with an intrinsic subtlety (already present in~\cite{WANG2022108028})
that is absent in most applications of the saddle point approximation
technique:
our goal is to show that the coefficients of $q^m$ in the ``Borwein
polynomial'' $P_n(q)$ (respectively in its powers) obey a certain sign pattern,
with $m$ running through a range that includes the asymptotic orders
$O(n^\omega)$, where $1\le \omega\le2$. 
Consequently, our estimations must hold for that entire range,
which makes it necessary to manage expressions that contain 
the radius~$r$ of our contour that is solution of the approximate
saddle point equation \eqref{eqStationaryPoint} without further
specification of its asymptotic order, as for example in the
definition of the cutoff in \eqref{eqCutoffTheta0}. The ``best'' that
we can say about~$r$ is its range as given in
Lemma~\ref{leRadiusBound} (which again --- necessarily --- covers
several different asymptotic orders in terms of~$n$ at logarithmic scale).

The last section, Section~\ref{seDiscuss}, is devoted to a discussion
of our approach and further applications. 
We start by explaining what is missing for the
completion of the proof of Conjecture~\ref{cjBorwein4}. We discuss
the applicability of our methods for proving the Third Borwein Conjecture
(see Conjecture~\ref{cjBorwein3}), a conjecture of Ismail,
Kim and Stanton vastly generalising
the First Borwein Conjecture (see Conjecture~\ref{cjIsKimSt}), 
or related or similar conjectures,
including some new ones that we present in this last section 
(in particular Conjectures~\ref{cjmod4} and~\ref{cjmod7}).
We also point out that the Bressoud Conjecture might as well be amenable
to the ideas developed in this paper.
Finally, we contemplate on the question whether the Borwein
Conjecture(s) should be considered as combinatorial or analytic,
a question which is evidently raised by our proof(s) (and other
observations).

\section{An outline of the proof}\label{seOutline}
Here, we provide a brief outline of our proof of Conjectures~\ref{cjBorwein}
and \ref{cjBorwein2}, and of a part of Conjecture~\ref{cjBorwein4}.
From here on, we use the standard notation for $q$-shifted factorials,
\begin{align*}
(\alpha;q)_n&=(1-\alpha)(1-\alpha q)\cdots(1-\alpha q^{n-1}), \text{ for }n\geq1,\\
(\alpha;q)_0&=1.
\end{align*}
If $\vert q\vert<1$, or in the sense of formal power series in~$q$, 
this definition also makes sense for $n=\infty$.
Using this notation, the ``Borwein polynomial'' can be written as
$$
P_n(q)=\frac {(q;q)_{3n}} {(q^3;q^3)_n}.
$$
Furthermore, in the following we shall write $[q^m]P(q)$ for the coefficient of~$q^m$
in the polynomial~$P(q)$.

Our goal is to show that the sign pattern of the coefficients 
$$
[q^m]P_n^\dd (q),\quad m=0,1,2,\dots,
$$
is $+--+--+--\cdots$, where $\dd $ is $1$, $2$, or $3$.

Our proof is composed of several parts.

\medskip
{\sc A. The conjectures hold for the ``first'' $3n+1$ coefficients and the
``last'' $3n+1$ coefficients.} We observe that the first few coefficients
of $P_n^\dd (q)$ and $P_\infty^\dd (q)$, with $\dd \in\{1,2,3\}$, are identical.
More precisely, we have
\begin{equation} \label{eq:d=infty} 
[q^m]P_n^\dd (q)=[q^m]P_\infty^\dd (q)
\end{equation}
for $0\le m\le 3n$ and $\dd \in\{1,2,3\}$ (actually for all integers~$\dd $).
By a result of Andrews~\cite{MR1395410} this implies the sign pattern
of the first $3n+1$ coefficients of $P_n(q)$ as predicted by
Conjecture~\ref{cjBorwein}. Similarly, by a result of 
Kane~\cite{MR2052400}, this implies the sign pattern
of the first $3n+1$ coefficients of $P_n^2(q)$ as predicted by
Conjecture~\ref{cjBorwein2}. By using a result of Borwein, Borwein and
Garvan~\cite{MR1243610}, this also implies the sign pattern
of the first $3n+1$ coefficients of $P_n^3(q)$ as predicted by
Conjecture~\ref{cjBorwein4}. See Section~\ref{seInfinite} for the
details.

Combining the above observation with 
the fact that $P_n(q)$, and hence $P_n^\dd (q)$ for all~$\dd $, is
palindromic, it remains to show that the coefficients of $q^m$ in
$P_n^\dd (q)$ for $3n\le m\le(\dd \deg P_n)/2$ follow the sign pattern
predicted by Conjectures~\ref{cjBorwein}--\ref{cjBorwein4}.

\medskip
{\sc B. Contour integral representation of the coefficients of
  $P_n^\dd (q)$.}
From now on, for convenience, we shall often use $Q_n(q)$ to denote
$P_n^\dd (q)$, where $\dd $ is~$1$, $2$, or~$3$.

Using Cauchy's integral formula, the coefficient $[q^m]Q_n(q)$ 
can be represented as the integral
\[
\frac{1}{2\pi i}\int_{\Gamma} Q_n(q)\frac{dq}{q^{m+1}},
\]
where $\Gamma$ is any contour about~$0$ with winding number 1. We will choose $\Gamma$ as a circle centred at $0$ with radius $r$ for some $r\in\R^+$, so that the integral becomes
\begin{equation}
[q^m]Q_n(q)=\frac{r^{-m}}{2\pi}\int_{-\pi}^{\pi}Q_n\left(re^{i\theta}\right)e^{-im\theta}\,d\theta.\label{eqIntRep}
\end{equation}

\medskip
{\sc C. The saddle point approximation.}
The exact choice of $r$ is related to the \emph{saddle points\/} of
$q\mapsto\vert q^{-m}Q_n(q)\vert$, 
and we will elaborate on this in Section~\ref{seLocate}.
The appropriate choice for~$r$ is a value smaller than~$1$ but close to~$1$,
see Lemma~\ref{leRadiusBound}.

Figure \ref{fiCircle} illustrates the typical behaviour of 
$\theta\mapsto|P_n\left(re^{i\theta}\right)|$ on the circle
$\{z\in\C:|z|=r\}$. 
In particular, we can observe the following general features in the graph:
\begin{itemize}
\item the function has two peaks close to $\theta=2\pi/3$ and
  $\theta=-2\pi/3$;\footnote{\label{foot:2}The actual locations of the peaks have arguments
slightly off $\theta=\pm2\pi/3$. This is one of the delicate
points of the estimations to be performed.}
\item the function values outside small neighbourhoods of $\theta=
2\pi/3$
  and $\theta=-2\pi/3$ are very small compared to the peak value.
\end{itemize}

\begin{figure}
\includegraphics[width=0.8\textwidth]{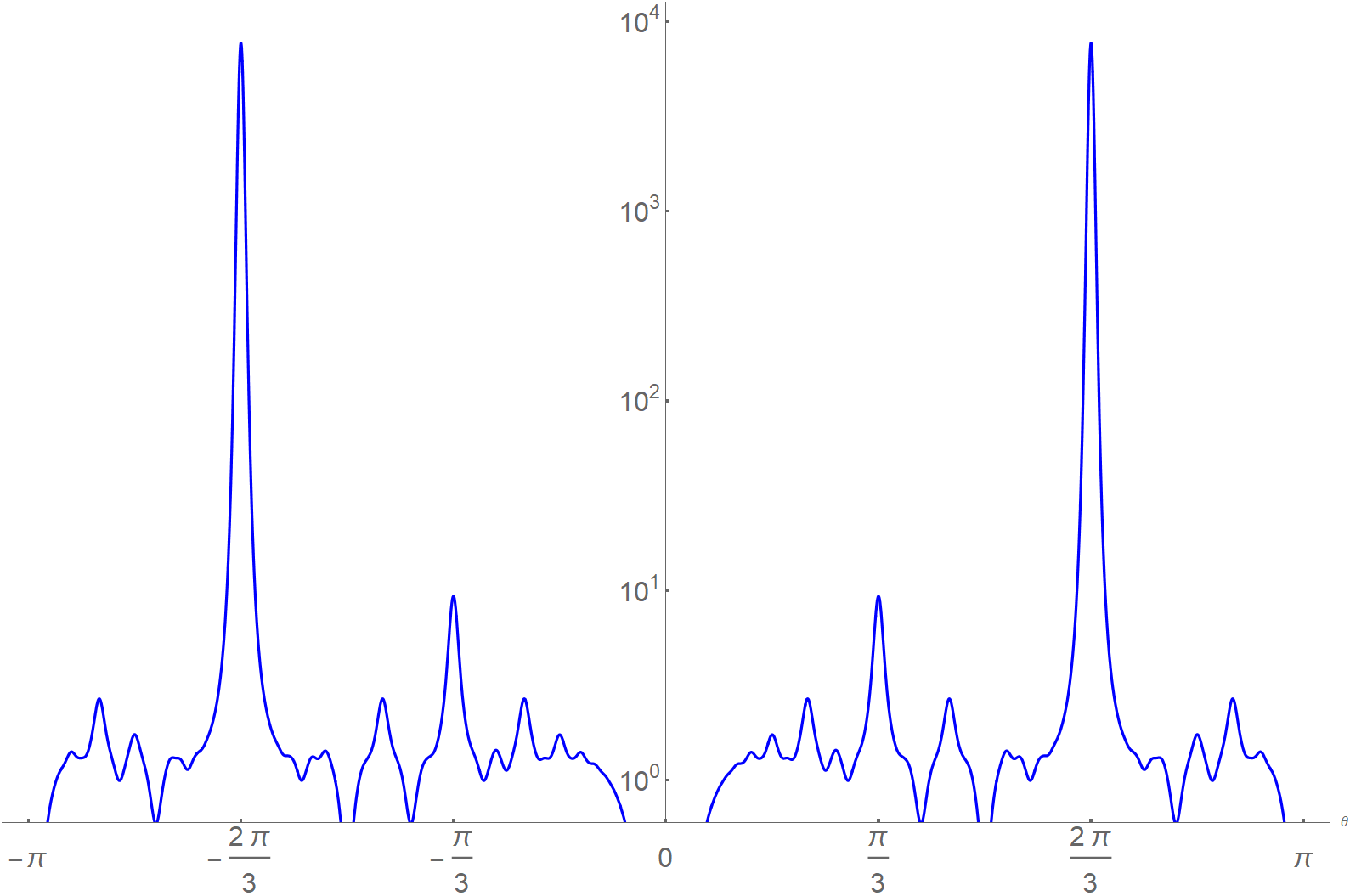}
\caption{Modulus of $P_{81}(0.95e^{i\theta})$. The vertical axis has logarithmic scale.}\label{fiCircle}
\end{figure}

Based on these heuristics, we choose a cutoff $\theta_0$ (to be
determined in \eqref{eqCutoffTheta0} in Section~\ref{secutoff}), 
and distinguish the following parts of the interval $[-\pi,\pi]$: 
\begin{itemize}
\item The \emph{peak part} $I_\text{peak}:=[-2\pi/3-\theta_0,-2\pi/3+\theta_0]\cup[2\pi/3-\theta_0,2\pi/3+\theta_0]$.
\item The \emph{tail part} $I_\text{tail}:=[-\pi,\pi]\setminus I_\text{peak}$.
\end{itemize}
\noindent
Naturally, the integral \eqref{eqIntRep} can be divided into two subintegrals corresponding to the two parts above.

We make the following observations concerning the subintegrals:

\smallskip
$\bullet$ The subintegral $\int_{I_{\text{peak}}}
  Q_n\left(re^{i\theta}\right)e^{-im\theta}\,d\theta$ can be
  approximated by a Gau\ss ian integral. 
More specifically, if we define
\begin{equation}\label{eqDefG}
g_{Q_n}(r)=\left.-\Re\frac{\partial^2}{\partial\theta^2}\log Q_n(re^{i\theta})\right|_{\theta=2\pi/3},
\end{equation}
then we have
\begin{align}
\int^{2\pi/3+\theta_0}_{2\pi/3-\theta_0}Q_n(re^{i\theta})e^{-im\theta}\,d\theta 
&=e^{-2\pi m i/3}\int_{-\theta_0}^{\theta_0}Q_n(re^{i(\theta+2\pi/3)})e^{-im\theta}\,d\theta \nonumber\\
&\approx e^{-2\pi m i/3}Q_n(re^{2\pi i/3})\int_{-\theta_0}^{\theta_0}e^{-g_{Q_n}(r)\theta^2/2}\,d\theta \nonumber\\
&= e^{-2\pi m i/3}Q_n(re^{2\pi i/3})\frac{\sqrt{2\pi}}{\sqrt{g_{Q_n}(r)}}\erf\left(\frac{\theta_0\sqrt{g_{Q_n}(r)}}{\sqrt2}\right).
\label{eqPart1}
\end{align}
Here, ``$\approx$'' means ``approximated by''.
Since $Q_n(q)$ is a polynomial with real coefficients, we have
$Q_n(\bar{z})=\overline{Q_n(z)}$. Therefore, an analogous
approximation holds for the other interval of $I_{\text{peak}}$,
that is, for the integral over~$\theta$ in $[-2\pi/3-\theta_0,-2\pi/3+\theta_0]$.
The error made by these approximations is
captured by the term $\epsilon_{0,Q_n}(m,r)$ defined below.

\smallskip
$\bullet$
The subintegral over $I_\text{tail}$ can be bounded above by
\begin{equation}\label{eqPart2}
\left|\int_{I_\text{tail}}Q_n(re^{i\theta})e^{-im\theta}\,d\theta\right|
\leq \left|Q_n(re^{2\pi i/3})\right|\int_{I_\text{tail}}\left|\frac{Q_n(re^{i\theta})}{Q_n(re^{2\pi i/3})}\right|\,d\theta.
\end{equation}
The error of this approximation is
captured by the term $\epsilon_{1,Q_n}(r)$ defined below.

\medskip
{\sc D. Bounding the errors.}
Our next step is to estimate the error in the approximation
\eqref{eqPart1} of the peak part, and to bound the tail part
\eqref{eqPart2} of the integral. Accordingly, we define the
error terms $\epsilon_{0,Q_n}(m,r)$ and $\epsilon_{1,Q_n}(r)$.
Both are {\it relative} errors, namely relative 
to the modulus of the (presumably, at this point) dominating part 
$$|Q_n(re^{2\pi i/3})|\frac{\sqrt{2\pi}}{\sqrt{g_{Q_n}(r)}}
\erf\left(\frac{\theta_0\sqrt{g_{Q_n}(r)}}{\sqrt2}\right)$$ 
(cf.\ \eqref{eqPart1}). 
Namely, we define
\begin{multline}
\epsilon_{0,Q_n}(m,r)\\
\label{eq:ep0}
:=2\left|
\frac{\sqrt{g_{Q_n}(r)}}
{\sqrt{2\pi}\erf\left({\theta_0\sqrt{g_{Q_n}(r)/2}}\right)}
\int_{-\theta_0}^{\theta_0}\left(\frac{Q_n(re^{i(\theta+2\pi/3)})}{Q_n(re^{2\pi i/3})}e^{-im\theta}-e^{-g_{Q_n}(r)\theta^2/2}\right)\,d\theta\right|
\end{multline}
and
\begin{equation}
\label{eq:ep1}
\epsilon_{1,Q_n}(r):=\frac{\sqrt{g_{Q_n}(r)}}
{\sqrt{2\pi}\erf\left({\theta_0\sqrt{g_{Q_n}(r)/2}}\right)}
\int_{I_\text{tail}}\left|\frac{Q_n(re^{i\theta})}{Q_n(re^{2\pi i/3})}\right|\,d\theta.
\end{equation}
In Lemma~\ref{lem:ep0ep1} in Section~\ref{seError}, we show that,
with these error terms, the coefficient of~$q^m$ in $Q_n(q)$ can be
approximated by
\begin{multline}
\left|\frac{r^m\sqrt{2\pi g_{Q_n}(r)}}
{\erf\left({\theta_0\sqrt{g_{Q_n}(r)/2}}\right)}
\frac {1} {|Q_n(re^{2\pi i/3})|}[q^m]Q_n(q)
-2\cos\left(\arg Q_n(re^{2\pi i/3})-2m\pi/3\right)
\right|
\\
\le\epsilon_{0,Q_n}(m,r)+\epsilon_{1,Q_n}(r).
\label{eqTargetMain}
\end{multline}

Therefore, there are two things to accomplish, the second required by
the first: 

\begin{enumerate} 
\item Show that the error terms $\epsilon_{0,Q_n}(m,r)$ and $\epsilon_{1,Q_n}(r)$ are small enough to satisfy the inequality
\begin{equation} \label{eqFinalTarget}
\epsilon_{0,Q_n}(m,r)+\epsilon_{1,Q_n}(r)<\left|2\cos\left(\arg Q_n(re^{2\pi i/3})-2m\pi/3\right)\right|.
\end{equation}
\item Get a control on $\arg Q_n(re^{2\pi i/3})$ and show that it is
less than $\frac {2\pi} {3}-\frac {\pi} {2}=\frac {\pi} {6}$ in
absolute value.
\end{enumerate}

Both together allow us to conclude that $[q^m]Q_n(q)$ has the
same sign as the cosine term on the right-hand side of \eqref{eqFinalTarget}, that is, it is positive
if $m\equiv0$~(mod~$3$) and negative otherwise, exactly as predicted by
Conjectures~\ref{cjBorwein}--\ref{cjBorwein4}.

The peak error $\epsilon_{0,Q_n}(m,r)$ is estimated in
Section~\ref{seEps0} (see Lemma~\ref{leIneqPeak}), 
and Section~\ref{seEps1} treats the tail error
$\epsilon_{1,Q_n}(r)$ (see Lemma~\ref{leIneqTail}). 

\medskip
{\sc E. Concluding the proof.}
As explained in the preceding Part~D, the tasks formulated in Items~(1) and~(2)
above must be accomplished. Task~(2) is taken care of in Lemma~\ref{leArg}. 
By combining this with the obtained bounds on $\epsilon_{0,Q_n}(m,r)$ and
$\epsilon_{1,Q_n}(r)$, Task~(1) is carried out in the remaining parts of 
Section~\ref{seMain} for ``large''~$n$.
In combination with suitable direct calculations for ``small''~$n$, 
this leads to full proofs of the First and Second Borwein Conjecture, and
to a partial proof of the Cubic Borwein Conjecture,
see Theorems~\ref{thMain1}, \ref{thMain2} and \ref{thMain3}.

\section{The infinite cases}\label{seInfinite}
In this section, we show that the first $3n+1$
coefficients of $P_n^\dd (q)$, where $\dd $ is $1$, $2$, or~$3$,
follow the sign pattern $+--+--+--\cdots$,
by using the simple fact, observed before
in \eqref{eq:d=infty}, that they agree with the corresponding
coefficients of $P_\infty^\dd (q)$, and by exploiting known properties
of the expansions of $P_\infty^\dd (q)$.

\medskip
Andrews \cite[Eqs.~(4.2)--(4.4)]{MR1395410} showed that
  \[
  P_{\infty}(q)=\frac{(q;q)_{\infty}}{(q^3;q^3)_\infty}=\frac{(q^{12},q^{15},q^{27};q^{27})_\infty-q(q^{6},q^{21},q^{27};q^{27})_\infty-q^2(q^{3},q^{24},q^{27};q^{27})_\infty}{(q^3;q^3)_\infty}.
  \]
Clearly, this implies that the sign pattern of the coefficients of
$P_\infty(q)$ is $+--+--+--\cdots$.\footnote{We point out that this
  sign pattern of the coefficients of $P_\infty(q)$ also follows from
a general result of Andrews~\cite[Theorem~2.1]{MR1395410} that,
according to \cite{MR1395410}, has also been independently obtained
by Garvan and P.~Borwein.}

\medskip
Using the circle method,
Kane~\cite{MR2052400} established the sign pattern $+--+--+--\cdots$
for the power series $(q;q)^2_{\infty}/(q^3;q^3)_\infty$,
except for the coefficient of $q^5$ which is equal to~$1$. A
multiplication with the series $(q^3;q^3)_\infty^{-1}$ (which has
positive coefficients) transforms this power series into
$P_{\infty}^2(q)$, and in the process removes the mentioned outlier. 

\medskip
Finally, it follows from results of Borwein, Borwein and Garvan~\cite{MR1243610} that 
\begin{equation} \label{eq:BBG} 
  \frac{(q;q)_\infty^3}{(q^3;q^3)_\infty}
=\sum_{m,n\in\Z}q^{3(m^2+mn+n^2)}
-q\sum_{m,n\in\Z}q^{3(m^2+mn+n^2+m+n)},
\end{equation}
where, as usual, $\Z$ denotes the set of integers.
To be precise, from Items~(ii) and~(iii) of Lemma~2.1
in~\cite{MR1243610}, one can derive the equation $b(q)=a(q^3)-c(q^3)$.
Proposition~2.2 in~\cite{MR1243610} shows that $b(q)$ equals the
left-hand side in \eqref{eq:BBG}, while the definitions of $a(q^3)$
and $b(q^3)$ from~\cite{MR1243610} are as stated on the right-hand
side of \eqref{eq:BBG}. As before, multiplication of both sides of
\eqref{eq:BBG} by $(q^3;q^3)_\infty^{-2}$, which is a power series
with non-negative coefficients, shows that the coefficients of
$P_n^3(q)$ follow the sign pattern $+--+--+--\cdot$.\footnote{We point
out that this sign pattern of the coefficients of $P_\infty^3(q)$ also
follows from a general result of Schlosser and Zhou~\cite[Theorem~6]{SchlZhou}.}

It should be noted however that \eqref{eq:BBG} also implies that 
the coefficients of $q^{3m+2}$ in $P_\infty^3(q)$
are zero for all~$m$. This observation, and its implications, will be
discussed in more detail in Item~(1) of Section~\ref{seDiscuss}.

\section{The log-derivatives of the ``Borwein polynomial'' $P_n(q)$}\label{seConvention}

In this section, we present some basic facts on derivatives of $\log
P_n(r e^{i\theta})$ with respect to~$\theta$. These will be used
ubiquitously in the subsequent sections.

By routine calculation, we see that the $j$-th derivative of $\log
P_n(r e^{i\theta})$, ``centred'' at $\theta=2\pi/3$, can be expressed as 
\begin{equation}\label{eqPnDerivs}
\left(\frac{\partial}{\partial\theta}\right)^j\log P_n(r e^{i\theta})=\frac12i^jU_j(n,re^{i(\theta-2\pi/3)})+\frac{\sqrt{3}}{2}i^{j-1}V_j(n,re^{i(\theta-2\pi/3)}),
\end{equation}
where
\begin{align}
U_j(n,z):=\sum_{k=1}^{n}\left((3k-2)^ju_j(z^{3k-2})+(3k-1)^ju_j(z^{3k-1})\right),\label{eqDerivU}\\
V_j(n,z):=\sum_{k=1}^{n}\left((3k-2)^jv_j(z^{3k-2})-(3k-1)^jv_j(z^{3k-1})\right),\label{eqDerivV}
\end{align}
and the rational functions $u_j$ and $v_j$ are given by
\begin{align}\label{eqDerivFuncs-u}
u_j(z)&:=\left(z\frac{d}{dz}\right)^{j-1} \frac{z(1+2z)}{1+z+z^2}, \\
\label{eqDerivFuncs-v}
v_j(z)&:=\left(z\frac{d}{dz}\right)^{j-1} \frac{z}{1+z+z^2}.
\end{align}
In particular, the first few of these functions are given by
\begin{align*}
  u_1(z)&=\frac{z(1+2z)}{1+z+z^2}, \\ v_1(z)&=\frac{z}{1+z+z^2}, \\
  u_2(z)&=\frac{z(1+4z+z^2)}{(1+z+z^2)^2}, \\ v_2(z)&=\frac{z(1-z^2)}{(1+z+z^2)^2}, \\
  u_3(z)&=\frac{z(1-z^2)(1+7z+z^2)}{(1+z+z^2)^3}, \\ v_3(z)&=\frac{z(1-z-6z^2-z^3+z^4)}{(1+z+z^2)^3}, \\
  u_4(z)&=\frac{z(1+12z-12z^2-56z^3-12z^4+12z^5+z^6)}{(1+z+z^2)^4}, \\ v_4(z)&=\frac{z(1-z^2)(1-4z-21z^2-4z^3+z^4)}{(1+z+z^2)^4}.
\end{align*}

We also define the sums
\begin{equation}\label{eqXmDef}
X_j(n,r):=\underset{3\nmid k}{\sum_{k=1}^{3n}}k^jr^k=\sum_{k=1}^{3n}k^jr^k-3^j\sum_{k=1}^nk^j(r^3)^k,
\end{equation}
and denote the corresponding infinite sum by $X_j(\infty,r)$. It is
easy to see that\break $(1-r^3)^{j+1}X_j(n,r)$ is a polynomial in $n$, $r$
and~$r^n$. 
Furthermore, $X_j(n,r)$ is increasing with respect to both $n$
and~$r$. A collection of inequalities between
various products of these sums is given in Lemma~\ref{leIneqX}. 
These inequalities are used in the estimations in Section~\ref{seEps0}.

\section{Locating the dominant (approximate) saddle points}\label{seLocate}
The results of Section~\ref{seInfinite}, and the fact that the
polynomial $P_n(q)$ is palindromic for all~$n$, together show that it
suffices to consider $[q^m]Q_n(q)$ for $m\in[3n,(\deg Q_n)/2]$, where
$Q_n$ is chosen as $P_n^\dd (q)$ for $\dd \in\{1,2,3\}$, as before.
The purpose of this section is to describe our choice of the radius
$r$ in \eqref{eqIntRep}.  

Ideally, in line with standard practice in analytic combinatorics, the
radius~$r$ in the integral in~\eqref{eqIntRep} should
be chosen such that the circle $\theta\mapsto re^{i\theta}$, 
$-\pi\le\theta\le\pi$, passes through the
dominant saddle point(s)\footnote{Here, ``dominant saddle point(s)'' means 
``the saddle point(s) with largest modulus of the
integrand''. We shall sometimes also abuse terminology and
speak of ``dominant peaks''.} of the function $q\mapsto \vert q^{-m}Q_n(q)\vert$. If
$Q_n(q)$ has non-negative coefficients, according to Pringsheim's
theorem, the dominant saddle point is
located on the positive real axis, and the problem is equivalent
to the minimisation of the quantity $r^{-m}Q_n(r)$. 

In our case however, the dominant saddle points are located near the complex
third roots of unity instead of on the positive real axis. In analogy
to the process above, we choose the radius $r$ so that the quantity
$r^{-m}\left|Q_n(re^{2\pi i/3})\right|$ is minimised. By taking a
log-derivative, and substituting $Q_n=P_n^\dd (q)$, we obtain an equation
in terms of $r$:\footnote{\label{foot:3}The reader must be warned: this is {\it not\/} the
  saddle point equation! The saddle point equation is $q\frac {d}
  {dq}P_n(q)=m/\dd $, as an equation for complex~$q$. It will have two
  solutions with arguments {\it close} to $\pm 2\pi/3$, but not {\it
    exactly} $\pm 2\pi/3$. Equation~\eqref{eqStationaryPoint} is
a ``saddle point-like equation'', in which the argument of the solution
is ``frozen'' to $2\pi/3$. In our analysis, it mimics the role of a
saddle point equation, but is in fact ``just'' an
``approximate'' saddle point equation. We made this deliberate choice
since we deemed it unfeasible to carry through the programme of
approximations without having a firm control on the arguments of the
(approximate or not) saddle points.
As it turns out, this is
nevertheless good enough for performing our estimations.}
\begin{equation}\label{eqStationaryPoint}
r\Re\left(\frac{d}{dr} \log P_n(re^{2\pi i/3})\right)=\frac{m}{\dd }.
\end{equation}
It must  be emphasised that the solution $r$ of this equation (it will
indeed be shown in Lemma~\ref{leRadiusBound} below that there is a
unique solution) depends on $n$ and $m$ (and $\dd$ of course). We will however
most of the time suppress this dependency 
in the interest of better readability. Only occasionally, when we
think that this is necessary, we will add
an index that indicates the dependency (as for example in
Lemmas~\ref{leRadiusBound} and~\ref{leIneqPeak}, or in the proofs of
Theorems~\ref{thMain1}, \ref{thMain2} and~\ref{thMain3}).

It turns out that, under the above restriction on~$m$, the minimiser
radius $r$ approaches~$1$ as $n\to\infty$. These observations are
proved in the following lemma. They are crucial in our estimations of
the error terms $\epsilon_{i,Q_n}$, $i=0,1$.

\begin{lemma}\label{leRadiusBound}
For all integers $n\geq1$ and $m\in(0,\dd \deg P_n)$, with
$\dd \in\{1,2,3\}$, the approximate saddle point equation~\eqref{eqStationaryPoint} has a unique solution $r=r_{m,n}\in\R^+$. Moreover, if\/ $3n\leq m\leq (\dd \deg P_n)/2$, then we have $r_0<r\leq1$, where
\begin{equation}\label{eqCutoffR0}
r_0=e^{-\sqrt{4\dd /27n}}.
\end{equation}
Furthermore, as a function in~$m$, the solution $r=r_{m,n}$
to~\eqref{eqStationaryPoint} is increasing.
\end{lemma}

\begin{proof}
We infer from \eqref{eqPnDerivs} that the left-hand side of \eqref{eqStationaryPoint} can be written as
\begin{equation}
r\Re\left(\frac{d}{dr} \log P_n(re^{2\pi i/3})\right)=\frac12\underset{3\nmid k}{\sum_{k=1}^{3n}} ku_1(r^k),
\end{equation}
where $u_1(x)=x(1+2x)/(1+x+x^2)$ is defined as in Section~\ref{seConvention}.

Therefore, Equation \eqref{eqStationaryPoint} is equivalent to
\begin{equation}\label{eqStationaryPoint2}
\underset{3\nmid k}{\sum_{k=1}^{3n}}ku_1(r^k)=\frac{2m}{\dd }.
\end{equation}

Note that 
\begin{equation} \label{eq:Du1} 
u_1'(r)=\frac{1+4r+r^2}{(1+r+r^2)^2}>0,
\end{equation}
so $u_1$ is increasing. Moreover, we have the special values
\begin{align} \label{eq:u1}
u_1(0)&=0,&
u_1(1)&=1,&
\lim_{r\to+\infty}u_1(r)&=2.
\end{align}
Along with the fact that 
$$\deg P_n=\underset{3\nmid k}{\sum_{k=1}^{3n}}k,$$ 
these special
values imply that the sum  
$$\underset{3\nmid k}{\sum_{k=1}^{3n}}ku_1(r^k)/2$$ 
tends to $0$, $(\deg P_n)/2$, and $\deg P_n$ when $r\to0,1,+\infty$, respectively. The existence and uniqueness of solution, as well as the upper bound $r\leq1$, follow from the intermediate value theorem.

It remains to prove the lower bound on~$r$. Since $u_1$ is increasing,
it suffices to show that 
$$\underset{3\nmid k}{\sum_{k=1}^{3n}}ku_1(r_0^k)<\frac {6n}\dd .$$
Equation \eqref{eq:Ungl1} in Lemma~\ref{leUVBounds} implies that
\begin{multline*}
\underset{3\nmid k}{\sum_{k=1}^{3n}}ku_1(r_0^k)
<\frac{2}{\sqrt3}\underset{3\nmid k}{\sum_{k=1}^{3n}}kr_0^{k}
<\frac{2}{\sqrt3}\underset{3\nmid k}{\sum_{k=1}^{\infty}}kr_0^{k}\\
=\frac{2}{\sqrt3}\frac{r_0(1+2r_0+2r_0^3+r_0^4)}{(1-r_0^3)^2}<\frac{8}{9}(-\log r_0)^{-2}=\frac{6n}{\dd },
\end{multline*}
where the last inequality used the fact that the maximum of the function 
$$r\mapsto\frac{2r(1+2r+2r^3+r^4)(-\log r)^2}{\sqrt3(1-r^3)^2}$$ 
on $[0,1]$ is approximately $0.881906<8/9$.

\medskip
For the additional assertion at the end of  the lemma, we recall
from~\eqref{eq:Du1} that $u_1(r)$ is increasing in~$r$. Therefore,
by~\eqref{eqStationaryPoint2}, if~$m$ is increasing, so must be~$r$.
\end{proof}

\section{The choice of cutoff}\label{secutoff}
Our choice of the cutoff $\theta_0$ announced in
Part~C of Section~\ref{seOutline} is
\begin{equation}\label{eqCutoffTheta0}
\theta_0:=\Ct\frac{1-r^3}{1-r^{3n}},
\end{equation}
where the constant $\Ct$ is chosen as $\frac{10}{81}$.

We give some immediate consequences of \eqref{eqCutoffR0} and
\eqref{eqCutoffTheta0}, to be used in the following two sections.

\begin{lemma}\label{leCutoffPrelim}
With $\Z^+$ denoting the set of positive integers,
suppose that $n\in\Z^+$, $\dd \in\{1,2,3\}$, $t\geq0$, and~$r_0$ and~$\theta_0$
are defined as in \eqref{eqCutoffR0} and \eqref{eqCutoffTheta0},
respectively. 
Then the following results hold for $r\in(r_0,1]$ and $\theta\in[-t\theta_0,t\theta_0]$:
\begin{enumerate}
  \item For $n\geq4$, we have
  \begin{equation}\label{eqThetaBoundAux}
  \frac{1-r_0^3}{1-r_0^{3n}}\leq-3\log r_0,
  \end{equation}
  and consequently
  \begin{equation}\label{eqLogZBoundAux}
  \left|\log re^{i\theta}\right|<(1+3t\Ct)(-\log r_0)\leq\frac{2(1+3t\Ct)}{3\sqrt{n}}.
  \end{equation}
  \item For $k\in[0,3n]$, the complex number $r^ke^{ik\theta}$ belongs
    to the region $S_{3t\Ct}$, where $S_\rho$ is defined by 
  \begin{equation}\label{eqRegionS}
  S_\rho:=\left\{Re^{i\Theta}:0\leq R\leq1\text{ \em 
and } |\Theta|\leq\rho\frac{-\log R}{1-R}\right\}
  \end{equation}
  for $\rho>0$.
  \item Suppose $|\theta|\leq t\theta_0$ for some $t\geq0$. For $r\in(r_0,1]$ and $\ell\in\Z^+$, we have
  \begin{equation}\label{eqLogZGammaBound}
  \sup_{k\in[0,3n]}\left|\log re^{i\theta}\right|^{\ell}k^{\ell}r^k\leq \ell^\ell(e^{-1}+3t\Ct)^\ell.
  \end{equation}
\item For $j\geq0$, let $X_j(n,r)$ be defined as in \eqref{eqXmDef}. Then, for $n\geq400$ and $r\in(r_0,1]$, we have
\begin{align}
\label{eq:X0}
X_0(n,r)&>0.95\sqrt{n},\\ 
\label{eq:X1}
X_1(n,r)&>1.35n,\\  
\label{eq:X3}
X_3(n,r)&>16n^2,\\ 
\label{eq:X4}
X_4(n,r)&>94n^{5/2}.
\end{align}
\end{enumerate}
\end{lemma}
\begin{proof}
(1) We  have $-\log r_0=\sqrt{{4\dd }/{27n}}\leq\sqrt{{12}/{108}}=1/3 $.
Next we substitute $x:=-3\log r_0$ in the inequality $(1-e^{-x})/x\leq1-e^{-1/x}$ (valid for $0\leq x\leq1$) to obtain
  \[
  \frac{1-r_0^{3}}{-3\log r_0}\leq 1-e^{\frac{1}{3\log r_0}}<1-e^{3n\log r_0}=1-r_0^{3n},
  \]
  where the last inequality holds because $9n(\log r_0)^2=4\dd /3>1$.
  The inequality \eqref{eqLogZBoundAux} follows from
  \[
  \left|\log re^{i\theta}\right|\leq -\log r+t\theta_0\leq -\log r_0+t\Ct\frac{1-r_0^3}{1-r_0^{3n}}\leq(1+3t\Ct)(-\log r_0).
  \]

\medskip
(2) The definition \eqref{eqCutoffTheta0} implies that
  \[
  k|\theta|\leq kt\Ct\frac{1-r^3}{1-r^{3n}}\leq kt\Ct\frac{-\log r^3}{1-r^k}=3t\Ct\frac{-\log r^k}{1-r^k}.
  \]

\medskip
(3) We first note that
  \[
  \left(\sup_{k\in[0,3n]}k^{\ell}r^k\right)^{1/\ell}=\begin{cases}
                                 \frac{\ell}{e(-\log r)}, & \mbox{if } r\leq e^{-\ell/(3n)}, \\
                                 3nr^{3n/\ell}, & \mbox{if } r>e^{-\ell/(3n)}.
                               \end{cases}
  \]
  On the other hand, we have $\left|\log re^{i\theta}\right|\leq -\log r+|\theta|\leq-\log r+t\theta_0$, and therefore
  \begin{align*}
    \left|\log
    re^{i\theta}\right|\left(\sup_{k\in[0,3n]}k^{\ell}r^k\right)^{1/\ell}&\leq
    \frac \ell e+t\theta_0\begin{cases}
                                 \frac{\ell}{e(-\log r)}, & \mbox{if } r\leq e^{-\ell/(3n)}, \\
                                 3nr^{3n/\ell}, & \mbox{if } r>e^{-\ell/(3n)},
                               \end{cases} \\
    &=\frac \ell e+3\ell t\Ct\begin{cases}
                                 \frac{(1-r^3)}{e(-\log r^3)(1-r^{3n})}, & \mbox{if } r\leq e^{-\ell/(3n)}, \\
                                 \frac{n(1-r^3)r^{3n/\ell}}{\ell(1-r^{3n})}, & \mbox{if } r>e^{-\ell/(3n)},
                               \end{cases} \\
    &\leq\frac \ell e+3\ell t\Ct\begin{cases}
                                 \frac{1}{e(1-e^{-\ell})}, & \mbox{if } r\leq e^{-\ell/(3n)}, \\
                                 \frac{r^{3n/\ell}(-\log r^{3n/\ell})}{1-r^{3n/\ell}}, & \mbox{if } r>e^{-\ell/(3n)},
                               \end{cases} \\
    &\leq\frac \ell e+3\ell t\Ct\begin{cases}
                                 \frac{1}{e(1-e^{-\ell})}, & \mbox{if } r\leq e^{-\ell/(3n)}, \\
                                 1, & \mbox{if } r>e^{-\ell/(3n)},
                               \end{cases} \\
    &\leq \frac \ell e+3\ell t\Ct.
  \end{align*}

\medskip
(4) We first note that, for all $j,n$ and $r\in[0,1]$, we have
  \[
  X_j(\infty,r)-X_j(n,r)=\underset{3\nmid k}{\sum_{k=1}^{\infty}}r^{3n+k}(3n+k)^j<r^{3n}
\underset{3\nmid k}{\sum_{k=1}^{\infty}}r^k(3nk+k)^j=r^{3n}(3n+1)^jX_j(\infty,r).
  \]
Thus,
  \[
  X_j(n,r)>X_j(\infty,r)\left(1-(3n+1)^jr^{3n}\right).
  \]

The only place where $\dd $ figures in the
inequalities \eqref{eq:X0}--\eqref{eq:X4} is in~$r_0$, which, in its turn, 
determines the range for~$r$, namely the interval
$(r_0,1]$. This interval is largest for
  $\dd =3$. Clearly, it suffices to consider that case. Hence, from here on we
  assume that $\dd =3$ and correspondingly $r_0=e^{-2/(3\sqrt n)}$. 

By the above considerations, we have
  \begin{align*}
  X_j(n,r)&>X_j(n,r_0)>X_j(\infty,r_0)\left(1-(3n+1)^jr_0^{3n}\right)\\
  &=X_j(\infty,r_0)\left(1-(3n+1)^je^{-2\sqrt{n}}\right)\\
  &\geq \left(-3\log r_0\right)^{-j-1} \left(X_j(\infty,r_0)(1-r_0^3)^{j+1}\right)\left(1-(3n+1)^je^{-2\sqrt{n}}\right)\\
  &\geq n^{(j+1)/2}2^{-j-1}
\left(X_j(\infty,r_0)(1-r_0^3)^{j+1}\right)\left(1-(3n+1)^je^{-2\sqrt{n}}\right).
  \end{align*}
  Since $X_j(\infty,r_0)(1-r_0^3)^{j+1}$ is a polynomial in $r_0$ with
  non-negative coefficients (and therefore increasing with respect
  to~$n$) and $(3n+1)^je^{-2\sqrt{n}}$ is evidently decreasing with
  respect to~$n$ whenever $n\geq j^2$, the
  inequalities~\eqref{eq:X0}--\eqref{eq:X4} follow from evaluating the factor
  \[
  2^{-j-1}
\left(X_j(\infty,r_0)(1-r_0^3)^{j+1}\right)\left(1-(3n+1)^je^{-2\sqrt{n}}\right)
  \]
at $n=400$ and $j=0,1,3,4$. 
\end{proof}

\section{The fundamental error inequality}\label{seError}

In this section we prove the fundamental inequality, claimed in 
\eqref{eqTargetMain}, that provides an upper bound for the
approximation of the coefficient of~$q^m$ in $Q_n(q)=P_n^\dd(q)$,
where $\dd\in\{1,2,3\}$, in terms of the error terms
$\epsilon_{0,Q_n}(m,r)$ and $\epsilon_{1,Q_n}(r)$ defined
in~\eqref{eq:ep0} and~\eqref{eq:ep1}. 

\begin{lemma} \label{lem:ep0ep1}
With the notations from Section~\ref{seOutline}, we have
\begin{multline}
\left|\frac{r^m\sqrt{2\pi g_{Q_n}(r)}}
{\erf\left({\theta_0\sqrt{g_{Q_n}(r)/2}}\right)}
\frac {1} {|Q_n(re^{2\pi i/3})|}[q^m]Q_n(q)
-2\cos\left(\arg Q_n(re^{2\pi i/3})-2m\pi/3\right)
\right|
\\
\le\epsilon_{0,Q_n}(m,r)+\epsilon_{1,Q_n}(r).
\label{eq:ep0eep1}
\end{multline}
\end{lemma}

\begin{proof}
Denoting the argument of $Q_n(re^{2\pi i/3})$ temporarily by~$\gamma$,
from the integral representation~\eqref{eqIntRep} of the coefficient
of~$q^m$ in $Q_n(q)$ and the division of the integration interval
$[-\pi,\pi]$ into $I_{\text{peak}}$ and $I_{\text{tail}}$ (see Part~C
in Section~\ref{seOutline}), we obtain
\begin{align*}
&\frac 1 {|Q_n(re^{2\pi i/3})|}[q^m]Q_n(q)
=e^{i\gamma}
\frac{r^{-m}}{2\pi}\int_{-\pi}^{\pi}
\frac {Q_n\left(re^{i\theta}\right)} {Q_n(re^{2\pi i/3})}
e^{-im\theta}\,d\theta\\
&=e^{i\gamma}
\frac{r^{-m}}{2\pi}
\left(
\int_{2\pi/3-\theta_0}^{2\pi/3+\theta_0}
\frac {Q_n\left(re^{i\theta}\right)} {Q_n(re^{2\pi i/3})}
e^{-im\theta}\,d\theta
+
\int_{-2\pi/3-\theta_0}^{-2\pi/3+\theta_0}
\frac {Q_n\left(re^{i\theta}\right)} {Q_n(re^{2\pi i/3})}
e^{-im\theta}\,d\theta\right.\\
&\kern3cm
\left.
+
\int_{I_{\text{tail}}}
\frac {Q_n\left(re^{i\theta}\right)} {Q_n(re^{2\pi i/3})}
e^{-im\theta}\,d\theta
\right)\\
&=
\frac{r^{-m}}{2\pi}
\left(
e^{i\gamma-2m\pi i/3}\int_{-\theta_0}^{\theta_0}
\frac {Q_n\left(re^{i(\theta+2\pi/3)}\right)} {Q_n(re^{2\pi i/3})}
e^{-im\theta}\,d\theta\right.\\
&\kern1cm
\left.
+e^{i\gamma+2m\pi i/3-2i\gamma}
\int_{-\theta_0}^{\theta_0}
\frac {Q_n\left(re^{i(\theta-2\pi/3)}\right)} {Q_n(re^{-2\pi i/3})}
e^{-im\theta}\,d\theta
+e^{i\gamma}
\int_{I_{\text{tail}}}
\frac {Q_n\left(re^{i\theta}\right)} {Q_n(re^{2\pi i/3})}
e^{-im\theta}\,d\theta
\right),
\end{align*}
where we used the earlier observed fact that $Q_n(\bar{z})=\overline{Q_n(z)}$ twice to 
obtain the last line.
Using this relation and the definitions \eqref{eq:ep0} and \eqref{eq:ep1}
of the error terms, we are led to the following estimation:
\begin{align*}
\notag
&\left|\frac{r^m\sqrt{2\pi g_{Q_n}(r)}}
{\erf\left({\theta_0\sqrt{g_{Q_n}(r)/2}}\right)}
\frac {1} {|Q_n(re^{2\pi i/3})|}[q^m]Q_n(q)\right.\\
&\kern2cm\left.
-\frac{\sqrt{g_{Q_n}(r)}}
{\sqrt{2\pi}\erf\left({\theta_0\sqrt{g_{Q_n}(r)/2}}\right)}
\left(e^{i(\gamma-2m\pi/3)}+e^{-i(\gamma-2m\pi/3)}\right)
\int_{-\theta_0}^{\theta_0}e^{-g_{Q_n}(r)\theta^2/2}\,d\theta
\right|
\\
&\le
\left|\frac{\sqrt{ g_{Q_n}(r)}}
{\sqrt{2\pi}\erf\left({\theta_0\sqrt{g_{Q_n}(r)/2}}\right)}
\int_{-\theta_0}^{\theta_0}\left(
\frac {Q_n\left(re^{i(\theta+2\pi/3)}\right)} {Q_n(re^{2\pi i/3})}
e^{-im\theta}
-e^{-g_{Q_n}(r)\theta^2/2}\right)\,d\theta
\right|\\
&\kern1cm
+\left|\frac{\sqrt{ g_{Q_n}(r)}}
{\sqrt{2\pi}\erf\left({\theta_0\sqrt{g_{Q_n}(r)/2}}\right)}
\int_{-\theta_0}^{\theta_0}\left(
\frac {Q_n\left(re^{i(\theta-2\pi/3)}\right)} {Q_n(re^{2\pi i/3})}
e^{-im\theta}
-e^{-g_{Q_n}(r)\theta^2/2}\right)\,d\theta
\right|\\
&\kern1cm
+\left|\frac{\sqrt{ g_{Q_n}(r)}}
{\sqrt{2\pi}\erf\left({\theta_0\sqrt{g_{Q_n}(r)/2}}\right)}
\int_{I_{\text{tail}}}
\frac {Q_n\left(re^{i(\theta+2\pi/3)}\right)} {Q_n(re^{2\pi i/3})}
e^{-im\theta}\,d\theta
\right|\\
&\le\epsilon_{0,Q_n}(m,r)+\epsilon_{1,Q_n}(r).
\end{align*}
By the definition of the Gau{\ss} error function, this turns out to
be equivalent to~\eqref{eq:ep0eep1}.
\end{proof}

\section{Bounding the peak error}\label{seEps0}

The goal of the section is to provide a bound for the peak error term
$\epsilon_{0,Q_n}(m,r)=\epsilon_{0,P_n^\dd}(m,r)$ 
(cf.~\eqref{eq:ep0}). We will derive it from a 
general bound on relative errors for the approximation of a (complex) function
by a Gau{\ss}ian, given in Lemma~\ref{leGaussian} below.
To serve our purpose, we must apply this lemma to the function 
in~\eqref{eq:Qfun}. In order to be able to do this, we have to first
provide bounds for the various constants, defined by the derivatives
of the function, that appear in the lemma. This is done in
Lemma~\ref{leDerivBounds}. After these preparations, our bound for 
$\epsilon_{0,Q_n}(m,r)$ is presented, and proved, in Lemma~\ref{leIneqPeak}.

\medskip
Here is the announced general result about bounding relative errors of
the approximation of a (complex) function by a Gau{\ss}ian from above. 

\begin{lemma}\label{leGaussian}
Suppose that $x_0>0$ and $f\in C^4([-x_0,x_0];\C)$ with $f(0)=0$. We define $f_k:=f^{(k)}(0)$ for $k=1,2$ as well as
\[
f_3:=3\int^{1}_{0}(1-t)^2\sup_{|x|\leq t x_0}|f^{(3)}(x)|\,dt
\]
and
\[f_{4}:=4\int^{1}_{0}(1-t)^3\sup_{|x|\leq t x_0}\left|f^{(4)}(x)\right|\,dt,
\]
and we write $g=-\Re f_2$ for simplicity.

Suppose further that $f_1\in\R$, $g>0$, that
$\mu_3:=\frac{x_0f_{3}}{3g}\in(0,1)$, 
and that\break $\mu_4:=\frac{x_0\sqrt{f_{4}}}{\sqrt{8g}}\in(0,1)$. Then we have
\begin{multline}
\left|\sqrt{\frac{g}{2\pi}}
\int_{-x_0}^{x_0}\left(e^{f(x)}-e^{-g x^2/2}\right)\,dx\right|
\leq\erf\left(x_0\sqrt{\frac{g}{2}}\right)\cosh(f_1x_0)\\
\times
\left(\frac{|\Im f_2|+f_1^2}{2g}+\frac{4f_3\beta_1(\mu_3)}{9\sqrt{\pi}g^3}+\frac{f_4\beta_3(\mu_4)}{3\sqrt{\pi}g^2}
+\frac{4f_1f_3\beta_2(\mu_3)}{3\sqrt{\pi}g^2}+\frac{\sqrt{2}f_1f_4\beta_4(\mu_4)}{3\sqrt{\pi}g^{5/2}}\right),
\label{eq:betamu}
\end{multline}
where the functions $\beta_i$, $i=1,2,3,4$, 
are as defined in Lemma~\ref{leIneqBeta}.
\end{lemma}
\begin{proof}
Let $R_2(x)=f(x)-f_1x-f_2x^2/2$ be the second order Taylor remainder term of $f(x)$ at $x=0$, and let $R_{e}(x)=(R_2(x)+R_2(-x))/2$. Taylor's theorem (with the remainder in integral form) implies that
\begin{equation}
|R_2(x)|\leq\frac{f_{3}}{6}|x|^3 \quad \text{and}\quad  
|R_{e}(x)|\leq\frac{f_{4}}{24}|x|^4.
\label{eq:R2Re}
\end{equation}

We split the function $e^{f(x)}-e^{-g x^2/2}$ as follows:
\begin{align*}
e^{f(x)}-e^{-g x^2/2}&=e^{-g x^2/2}(e^{f_1x+i \Im f_2 x^2/2}-1)\\
&\kern2cm
+\cosh(f_1x)e^{f_2x^2/2}\left(e^{R_{2}(x)}-1\right)\\
&\kern2cm
+\sinh(f_1x)e^{f_2x^2/2}\left(e^{R_{2}(x)}-1\right).
\end{align*}
Subsequently, we consider the integral of each term over $[-x_0,x_0]$.

\medskip 
The integral of the first term is controlled by
  \begin{align*}
&\left|\int^{x_0}_{-x_0}e^{-g x^2/2}\left(e^{f_1x+i \Im f_2 x^2/2}-1\right)\right|\,dx\\
&\kern2cm
=\left|\int^{x_0}_{0}e^{-g x^2/2}(e^{f_1x+i \Im f_2 x^2/2}+e^{-f_1x+i \Im f_2 x^2/2}-2)\right|\,dx\\
&\kern2cm
\leq \int^{x_0}_{0}e^{-g x^2/2} \left(\left|\left(e^{f_1x}+e^{-f_1x}\right)\left(e^{i \Im f_2 x^2/2}-1\right)\right|+\left|e^{f_1x}+e^{-f_1x}-2\right|\right)\,dx\\
&\kern2cm
\leq \int^{x_0}_{0}e^{-g x^2/2} \left(2\cosh(f_1x_0)|\Im f_2|\frac {x^2}2+\cosh(f_1x_0)f_1^2x^2\right)\,dx\\
&\kern2cm
=\cosh(f_1x_0)\left(|\Im f_2|+f_1^2\right)\int^{x_0}_{0}x^2e^{-g x^2/2}\,dx\\
&\kern2cm
<\frac{\cosh(f_1x_0)(|\Im f_2|+f_1^2)}{g}\sqrt\frac{\pi}{2g}\erf\left(x_0\sqrt{\frac{g}{2}}\right).
  \end{align*}

\medskip For the second term, we utilise \eqref{eqIneqComplexE},
\eqref{eq:R2Re}, \eqref{coIneqBeta1} (with $u=g/2$ and $v=f_3/6$), 
and \eqref{coIneqBeta3} (with $u=g/2$ and $v=f_4/24$) to conclude that
  \begin{align*}
&\left|\int^{x_0}_{-x_0}\cosh(f_1x)e^{f_2x^2/2}\left(e^{R_{2}(x)}-1\right)\,dx\right|\\
&\kern3cm
=\left|\int^{x_0}_{0}\cosh(f_1x)e^{f_2x^2/2}\left(e^{R_{2}(x)}+e^{R_{2}(-x)}-2\right)\,dx\right|\\
&\kern3cm
\leq \cosh(f_1x_0)\int^{x_0}_{0}e^{-gx^2/2}\left|e^{R_{2}(x)}+e^{R_{2}(-x)}-2\right|\,dx\\
&\kern3cm
\leq 2\cosh(f_1x_0)\int^{x_0}_{0}e^{-gx^2/2}\left(\cosh\left(\frac
     {f_3|x|^3}6\right)-1+\sinh\left(\frac {f_4|x|^4}{24}\right
)\right)\,dx\\
&\kern3cm
\leq 2\cosh(f_1x_0)\erf\left(x_0\sqrt{\frac{g}{2}}\right)\left(\frac{8\sqrt2f_3^2\beta_1(\mu_3)}{36g^{7/2}}+\frac{4\sqrt2f_4\beta_3(\mu_4)}{24g^{5/2}}\right).
  \end{align*}

\medskip
For the third term, we utilise \eqref{eqIneqComplexO}, \eqref{coIneqBeta2} (with $u=g/2$ and $v=f_3/6$), and \eqref{coIneqBeta4} (with $u=g/2$ and $v=f_4/24$) to conclude that
  \begin{align*}
&\left|\int^{x_0}_{-x_0}\sinh(f_1x)e^{f_2x^2/2}\left(e^{R_{2}(x)}-1\right)\,dx\right|\\
&\kern3cm
=\left|\int^{x_0}_{0}\sinh(f_1x)e^{f_2x^2/2}\left(e^{R_{2}(x)}-e^{R_{2}(-x)}\right)\,dx\right|\\
&\kern3cm
\leq \int^{x_0}_{0}\sinh(f_1x)e^{-gx^2/2}\left|e^{R_{2}(x)}-e^{R_{2}(-x)}\right|\,dx\\
&\kern3cm
\leq 2f_1\cosh(f_1x_0)\int^{x_0}_{0}xe^{-gx^2/2}\left(\sinh\left(\frac
{f_3|x|^3}6\right)+\sinh\left(\frac {f_4|x|^4}{24}\right)\right)\,dx\\
&\kern3cm
\leq 2f_1\cosh(f_1x_0)\erf\left(x_0\sqrt{\frac{g}{2}}\right)\left(\frac{4\sqrt2f_3\beta_2(\mu_3)}{6g^{5/2}}+\frac{8f_4\beta_4(\mu_4)}{24g^3}\right).
  \end{align*}

Combining the above bounds, we get
\begin{multline*}
\left|\sqrt{\frac{g}{2\pi}}\int_{-x_0}^{x_0}
\left(e^{f(x)}-e^{-g x^2/2}\right)\,dx\right|\leq \cosh(f_1x_0)\erf\left(x_0\sqrt{\frac{g}{2}}\right)\\
\times
\left(\frac{|\Im f_2|+f_1^2}{2g}+\frac{4f_3\beta_1(\mu_3)}{9\sqrt{\pi}g^3}+\frac{f_4\beta_3(\mu_4)}{3\sqrt{\pi}g^2}
+\frac{4f_1f_3\beta_2(\mu_3)}{3\sqrt{\pi}g^2}+\frac{\sqrt{2}f_1f_4\beta_4(\mu_4)}{3\sqrt{\pi}g^{5/2}}\right),
\end{multline*}
which is exactly the assertion of the lemma.
\end{proof}

As announced at the beginning of this section.,
our plan is to apply Lemma~\ref{leGaussian} to the function 
$$x\mapsto\log\frac{e^{-imx}Q_n(re^{i(x+2\pi/3)})}{Q_n(re^{2\pi i/3})}$$ 
in order to get bounds on $\epsilon_{0,Q_n}(m,r)$. (The reader is reminded
from Part~B of the proof outline in Section~\ref{seOutline} that
$Q_n(q)=P_n^\dd (q)$ with $P_n(q)$ the ``Borwein polynomial'' from \eqref{eqPolynomial}.)
This application however requires upper and lower bounds for the various
constants in Lemma~\ref{leGaussian}, which we give next. 

\begin{lemma}\label{leDerivBounds}
Suppose that $n\geq400$, $m\in[3n,(\dd \deg P_n)/2]$, and $r$ is the
unique solution of the approximate saddle point
equation~\eqref{eqStationaryPoint} determined by $n$ and~$m$. Let 
$$f(\theta):=\dd \left(\log P_n(re^{i(\theta+2\pi/3)})-\log P_n(re^{2\pi
  i/3})-i m \theta\right),$$ 
and let the constants $f_j$, $j=1,2,3,4$, be defined as in Lemma~\ref{leGaussian} with the bound $\theta_0$ chosen as in \eqref{eqCutoffTheta0}. Then we have the following inequalities for the constants $f_j$:
\begin{align}
\label{eq:fX1}
f_1&< \frac{7}{40}\dd X_0(n,r),\\
\label{eq:fX2}
\frac{1}{3}\dd X_2(n,r)\leq-\Re f_2&<\frac{3}{5}\dd X_2(n,r), \\
\label{eq:fX3}
|\Im f_2|&<\frac{1}{3}\dd X_1(n,r),\\
\label{eq:fX4}
f_3&< \frac{2}{3}\dd X_3(n,r), \\
\label{eq:fX5}
f_{4}&< \frac{18}{25}\dd X_4(n,r),
\end{align}
with the quantities $X_j(n,r)$ defined in \eqref{eqXmDef}.
\end{lemma}
\begin{proof}
Since all four constants are linear in $f$ and therefore proportional
to $\dd $, we assume $\dd =1$ in subsequent arguments without loss of
generality. 

We first give expressions respectively preliminary upper
bounds on these constants.
For $f_1$, we have
\begin{align}
\notag
  f_1&=\left(\frac {d} {d\theta}\log
  P_n(re^{i(\theta+2\pi/3)})\right)\bigg\vert_{\theta=0}-i m \\
\notag
&=\Re\left(\frac {d} {d\theta}\log
  P_n(re^{i(\theta+2\pi/3)})\right)\bigg\vert_{\theta=0}
+i\Re\left(r\frac {d} {dr}\log
  P_n(re^{i(2\pi/3)})\right)-i m \\
&=\frac{\sqrt3}{2}V_1(n,r),
\label{eq:f1}
\end{align}
where we used \eqref{eqPnDerivs} with $j=1$ and
the approximate saddle point equation \eqref{eqStationaryPoint} to get
the last line. Still using \eqref{eqPnDerivs}, we have
\begin{align}
\label{eq:f2}
  f_2&=-\frac12U_2(n,r)+\frac{\sqrt3i}{2}V_2(n,r),\\
\label{eq:f3}
  f_3&\leq3\int^1_0(1-t)^2\sup_{|\theta|\leq t\theta_0}\left(\frac12\left|U_3(n,re^{i\theta})\right|+\frac{\sqrt3}{2}\left|V_3(n,re^{i\theta})\right|\right),\\
\label{eq:f4}
  f_4&\leq4\int^1_0(1-t)^3\sup_{|\theta|\leq t\theta_0}\left(\frac12\left|U_4(n,re^{i\theta})\right|+\frac{\sqrt3}{2}\left|V_4(n,re^{i\theta})\right|\right).
\end{align}
Therefore the problem is reduced to proving upper and lower bounds for~$U_j$ and~$V_j$.

\medskip
{\sc Upper and lower bounds for $U_2(n,r)$.} 
The quantities $U_j$ are comparable to the corresponding $X_j$;
indeed, by comparing \eqref{eqDerivU} and \eqref{eqXmDef} and using
\eqref{eq:Ungl1}, we immediately obtain 
\[
\frac23X_2(n,r)\leq U_2(n,r)<\frac65 X_2(n,r),
\]
which translates into 
\[
\frac{1}{3}X_2(n,r)\leq-\Re f_2<\frac{3}{5}X_2(n,r),
\]
establishing \eqref{eq:fX2}.

\medskip
{\sc Upper bounds for $U_3(n,r)$ and $U_4(n,r)$.} 
Upper bounds for $U_3$ and $U_4$ can also be obtained by the same
comparison. In fact, for arbitrary~$j$ we have 
\begin{align*}
\sup_{|\theta|\leq t\theta_0}\left|U_j(n,re^{i\theta})\right|&\leq X_j(n,r)\sup_{|\theta|<t\theta_0}\sup_{0\leq k\leq 3n}\left|\frac{u_j(r^ke^{ik\theta})}{r^ke^{ik\theta}}\right|\\
&\leq X_j(n,r)\sup_{z\in S_{3t\Ct}}\left|\frac{u_j(z)}{z}\right|,
\end{align*}
where $S_\rho$ is defined in \eqref{eqRegionS}.

Remembering from \eqref{eqCutoffTheta0} that $\Ct=10/81$,
we use Lemma~\ref{leUVBounds}(2) to conclude that
\begin{align*}
3\int^1_0(1-t)^2\sup_{|\theta|\leq t\theta_0}|U_3(n,re^{i\theta})|&\leq 3X_3(n,r) \int^1_0(1-t)^2\sup_{z\in S_{3t\Ct}}\left|\frac{u_3(z)}{z}\right|\\
&\leq \left(\frac18\sup_{z\in S_{3\Ct}}\left|\frac{u_3(z)}{z}\right|+\frac78\sup_{z\in S_{3\Ct/2}}\left|\frac{u_3(z)}{z}\right|\right)X_3(n,r)\\
&\leq \left(\frac18\times1.44+\frac78\times1.3\right)X_3(n,r)=1.3175X_3(n,r),
\end{align*}
and similarly
\begin{align*}
4\int^1_0(1-t)^3\sup_{|\theta|\leq t\theta_0}|U_4(n,re^{i\theta})|&\leq 4X_4(n,r) \int^1_0(1-t)^3\sup_{z\in S_{3t\Ct}}\left|\frac{u_4(z)}{z}\right|\\
&\kern-10pt
\leq \left(\frac{1}{16}\sup_{z\in S_{3\Ct}}\left|\frac{u_4(z)}{z}\right|+\frac{15}{16}\sup_{z\in S_{3\Ct/2}}\left|\frac{u_4(z)}{z}\right|\right)X_4(n,r)\\
&\kern-10pt
\leq \left(\frac{1}{16}\times1.721+\frac{15}{16}\times1.409\right)X_4(n,r)=1.4285X_4(n,r).
\end{align*}

\medskip
{\sc A preliminary upper bound for $V_j(n,r)$.} 
As opposed to the $U_j$'s,
the quantities $V_j$, as alternating sums, are expected to be much smaller than $X_j(n,r)$. Indeed, let $w_j(k,z):=k^jv_j(z^k)$. Using Lemma~\ref{leAlternatingSum} for the function $w_j$, we see that
\begin{align}
  |V_j(n,z)|&\leq \frac{1}{3}\left|w_j(3n,z)-w_j(0,z)\right|+\frac{2}{3}|w_j''(3n,z)-w_j''(0,z)|+\frac{11n}{96}\sup_{k\in[0,3n]}\left|w_j^{(4)}(k,z)\right|\nonumber\\
  &=\frac{1}{3}\left|w_j(3n,z)\right|+\frac{2}{3}|w_j''(3n,z)|+\frac{11n}{96}\sup_{k\in[0,3n]}\left|w_j^{(4)}(k,z)\right|,\label{eqVjPrelimUpperBound}
\end{align}
since direct calculations reveal that $w_j(0,z)=w_j''(0,z)=0$ for $j=1,2,3,4$.

In order to treat the derivatives of the functions $w_j$, we note that
\eqref{eqPnDerivs} implies that  
\begin{equation}\label{eqVDeriv}
\left(\frac{\partial}{\partial k}\right)^{\ell} v_j(z^k)=(\log z)^\ell
v_{j+\ell}(z^k),
\quad \text{for }\ell\geq0.
\end{equation}
With this representation in mind, we proceed to give upper bounds for
the right-hand side of \eqref{eqVjPrelimUpperBound} for $j=1,2,3,4$, by making frequent
use of inequalities from Lemma~\ref{leUVBounds}.

\medskip
{\sc Upper bound for $V_1(n,r)$.} 
By using \eqref{eq:Ungl2a} and subsequently \eqref{eq:Ungl2}, we have
\[
\frac{w_1(3n,r)/3}{X_0(n,r)}=\frac{1-r^3}{(-\log r)(1+r)}\frac{r^{3n-1}(-\log r^n)}{1-r^{9n}}\leq\frac32\frac{r^{(3-1/400)n}(-\log r^n)}{1-r^{9n}}<0.201
\]
for the main term.
Using \eqref{eqVDeriv} and \eqref{eq:Ungl4}, we get
\begin{align*}
|w_1''(3n,r)|&=\frac{1}{3n}\left|2(\log r^{3n})v_{2}(r^{3n})+(\log r^{3n})^2v_{3}(r^{3n})\right|<\frac{1}{9n}
\end{align*}
for the second derivative. On the other hand, using \eqref{eqVDeriv}
and \eqref{eq:Ungl5}, we have
\begin{align*}
|w_1^{(4)}(k,r)|&=|\log r|^3\,\left|4v_{4}(r^k)+(\log r^{k})v_{5}(r^k)\right|<\frac{9}{8}|\log r|^3
\end{align*}
for the fourth derivative. 
Substitution of these bounds in \eqref{eqVjPrelimUpperBound} with $j=1$, if combined
with~\eqref{eq:X0} and the fact from Lemma~\ref{leRadiusBound} that $\vert \log r\vert
<\vert \log r_0\vert\le \frac {2} {3}$, then yields
\begin{align*}
|V_1(n,r)|&\leq 0.201X_0(n,r)+\frac{2}{27n}+\frac{33n}{256}|\log
r|^3\\
&<\left(0.201+\frac{2}{27\times0.95n^{3/2}}+\frac{11}{288\times0.95n}\right)X_0(n,r)
<0.202X_0(n,r).
\end{align*}

\medskip
{\sc Upper bound for $V_2(n,r)$.} Similarly to above, using
\eqref{eqX1Lower} in Lemma~\ref{leIneqX}, and subsequently
\eqref{eq:Ungl3a} and \eqref{eq:Ungl3}, we obtain
\begin{align*}
\frac{w_2(3n,r)/3}{X_1(n,r)}&\leq\frac{3(1-r^3)^2}{(1+2r+2r^3+r^4)(-\log r)^2}\frac{r^{3n-1}(1-r^{6n})(-\log r^n)^2}{(1-r^{9n})(1-r^{3n/2})(1+r^{3n}+r^{6n})}\\
&<\frac92\frac{r^{(3-1/400)n}(1-r^{6n})(-\log r^n)^2}{(1-r^{9n})(1-r^{3n/2})(1+r^{3n}+r^{6n})}<0.378
\end{align*}
for the main term. Using \eqref{eqVDeriv} and \eqref{eq:Ungl6}, we get
\begin{align*}
|w_2''(3n,r)|&=\left|2v_2(r^{3n})+2(\log r^{3n})v_{3}(r^{3n})+(\log r^{3n})^2v_{4}(r^{3n})\right|<0.21
\end{align*}
for the second derivative. By \eqref{eqVDeriv} and \eqref{eq:Ungl7}, we infer
\begin{align*}
|w_2^{(4)}(k,r)|&=|\log r|^2\,\left|12v_{4}(r^k)+8(\log r^{k})v_{5}(r^k)+(\log r^{k})^2v_{6}(r^k)\right|<3.61|\log r|^2
\end{align*}
for the fourth derivative. 
Substitution of these bounds in \eqref{eqVjPrelimUpperBound} with $j=2$, if combined
with~\eqref{eq:X1} and the earlier mentioned fact that $\vert \log r\vert
< \frac {2} {3}$, then yields
\begin{align*}
V_2(n,r)&\leq 0.378X_1(n,r)+0.14+0.42n|\log r|^2\\
&<0.378X_1(n,r)+0.14+0.19<0.38X_1(n,r).
\end{align*}

\medskip
{\sc Upper bounds for $V_3(n,re^{i\theta})$ and
  $V_4(n,re^{i\theta})$.} 
For these two quantities, instead of proving $V_j=O(X_{j-1})$ as above, we prove $V_j=o(X_j)$ as $n\to\infty$. Observe that Lemma~\ref{leCutoffPrelim}(2) and \eqref{eqVDeriv} imply that for $a=0,2,4$ we have
    \[
    |w_j^{(a)}(k,re^{i\theta})|\leq r^k\sum_{\ell=0}^a \frac{a!}{\ell!}\binom{j}{a-\ell}k^{j-a+\ell}|\log re^{i\theta}|^{\ell}\sup_{z\in S_{3t\Ct}}\left|\frac{v_{j+\ell}(z)}{z}\right|.
    \]
Therefore, by \eqref{eqLogZGammaBound} and
    \eqref{eqVjPrelimUpperBound}, we get
    \begin{multline*}
    \sup_{|\theta|<t\theta_0}|V_j(n,re^{i\theta})|\leq \frac13(3n)^jr^{3n}\sup_{z\in S_{3t\Ct}}\left|\frac{v_{j}(z)}{z}\right|\\
    +\frac23(3n)^jr^{3n}\sum_{\ell=0}^2\frac{2}{\ell!}\binom{j}{2-\ell}\frac{|\log re^{i\theta}|^{\ell}}{(3n)^{2-\ell}}\sup_{z\in S_{3t\Ct}}\left|\frac{v_{j+\ell}(z)}{z}\right|\\
    +\frac{11n}{96}|\log re^{i\theta}|^{a-j}\sum_{\ell=0}^4\frac{24}{\ell!}\binom{j}{4-\ell}(j-a+\ell)^{j-a+\ell}(e^{-1}+3t\Ct)^{j-a+\ell}\sup_{z\in S_{3t\Ct}}\left|\frac{v_{j+\ell}(z)}{z}\right|.
    \end{multline*}
Here we put $t=1$ (thus raising the bound on the right-hand side since
here $0\le t\le 1$).
    Substitution of the upper bounds from \eqref{eqLogZBoundAux} (with
     $t=1$) and
from    Lemma~\ref{leUVBounds}(2) leads to
    \begin{align*}
    \sup_{|\theta|<\theta_0}|V_3(n,re^{i\theta})|&<(3n)^3r^{3n}\left(0.34+1.17n^{-1}+1.25n^{-3/2}+0.45n^{-2}\right)+45.1\sqrt{n}\\
    &\leq 0.344(3n)^3r^{3n}+45.1\sqrt{n},\\
    \sup_{|\theta|<\theta_0}|V_4(n,re^{i\theta})|&<(3n)^4r^{3n}\left(0.34+3.04n^{-1}+3.40n^{-3/2}+0.91n^{-2}\right)+1135n\\
    &\leq 0.349(3n)^3r^{3n}+1135n.
    \end{align*}
    We now note that for $j\in\Z^+$ we have
    \[
    \frac{X_j(n,r)}{(3n)^jr^{3n}}\geq\frac{X_j(n,1)}{(3n)^j}>\frac{2}{j+1}(n-1).
    \]
Hence, by also using \eqref{eq:X3} and \eqref{eq:X4}, we have
    \begin{align*}
      \frac{\sup_{|\theta|<\theta_0}|V_3(n,re^{i\theta})|}{X_3(n,r)}&<\frac{2\times0.344}{n-1}+\frac{45.1}{16n^{3/2}}<\frac{5}{6n},\\
      \frac{\sup_{|\theta|<\theta_0}|V_4(n,re^{i\theta})|}{X_4(n,r)}&<\frac{5\times0.349}{2(n-1)}+\frac{1135}{94n^{3/2}}<\frac{3}{2n}.
    \end{align*}
    
By combining all the bounds above and using them in
\eqref{eq:f1}--\eqref{eq:f4}, we obtain 
    \begin{align*}
      f_1&<\frac{\sqrt3}{2}0.202X_0(n,r)<\frac{7}{40}X_0(n,r), \\
      |\Im f_2|&<\frac{\sqrt3}{2}0.38X_1(n,r)<\frac{1}{3}X_1(n,r),\\
      f_3&<\left(\frac12\times1.3175+\frac{\sqrt3}{2}\times\frac{5}{6n}\right)X_3(n,r)<\frac{2}{3}X_3(n,r),\\
      f_4&<\left(\frac12\times1.4285+\frac{\sqrt3}{2}\times\frac{3}{2n}\right)X_4(n,r)<\frac{18}{25}X_4(n,r),
    \end{align*}
thereby establishing the remaining inequalities.
\end{proof}

We are now ready for presenting, and proving, our upper bound for the
peak error term $\epsilon_{0,P_n^\dd }(m,r)$ as defined in~\eqref{eq:ep0}.

\begin{lemma}\label{leIneqPeak}
Let $n\geq400$ and $\dd \in\{1,2,3\}$. Furthermore, for $m\in[3n,\dd (\deg P_n)/2]$,
let $r=r_{n,m,\dd }$ be the solution of the approximate saddle point
equation~\eqref{eqStationaryPoint}, and let $\theta_0$ be the cutoff
as defined in \eqref{eqCutoffTheta0}. Then we have the following upper
bound for the peak error term $\epsilon_{0,P_n^\dd }(m,r)$:
\begin{equation}\label{eqIneqPeak}
{\epsilon_{0,P_n^\dd }(m,r)}<(146.2\dd ^{-1}+6.46+0.124\dd )\frac{X_1(n,r)}{X_2(n,r)}+\frac{7.222}{\sqrt{\dd X_2(n,r)}},
\end{equation}
where the $X_j(n,r)$ are as defined in \eqref{eqXmDef} and
$g_{Q_n}(r)$ is defined in \eqref{eqDefG}. 
Moreover, the right-hand side of \eqref{eqIneqPeak} is decreasing with respect to~$r$.
\end{lemma}
\begin{proof}
We apply Lemma~\ref{leGaussian} with $x_0=\theta_0$ to the function 
\begin{equation} \label{eq:Qfun} 
x\mapsto\log\frac{e^{-imx}Q_n(re^{i(x+2\pi/3)})}{Q_n(re^{2\pi
    i/3})}.
\end{equation}
This produces a bound for $\epsilon_{0,P_n^\dd }(m,r)$ in terms of
the quantities $f_1,f_2,f_3,f_4,g$ and
$\beta_1(\mu_3),\beta_2(\mu_3),\beta_3(\mu_4),\beta_4(\mu_4)$.
We now need to estimate the individual terms in \eqref{eq:betamu} using
the inequalities in
Lemma~\ref{leDerivBounds} and Corollary~\ref{coIneqX}, and the
estimates for the particular values in Lemma~\ref{leIneqBeta}. 
In order to justify the use of Lemma~\ref{leIneqBeta}, we have to verify that
$\mu_3\le20/27$ and $\mu_4\le2/3$. Indeed, using \eqref{eq:fX2},
 \eqref{eq:fX4}, and the observation that, by definition, $g=-\Re f_2$ and
 $X_0(n,r)=r(1+r)(1-r^{3n})/(1-r^3)$, we have
\begin{equation*}
\mu_3=\frac{\theta_0f_3}{3g}\leq \frac{2r(r+1)\Ct X_3(n,r)}{3X_0(n,r)X_2(n,r)}\leq 6\Ct=\frac{20}{27},
\end{equation*}
where we used \eqref{eqIneqR3_02}.
Similarly, using in addition \eqref{eq:fX5}, we get
\begin{equation*}
\mu_4=\sqrt{\frac{\theta_0^2f_4}{8g}}\leq\sqrt{\frac{27\Ct^2X_4(n,r)r^2(r+1)^2}{100X_0^2(n,r)X_2(n,r)}}\leq\frac{27}{5}\Ct=\frac23,
\end{equation*}
where we used \eqref{eqIneqR4_002}.
Knowing these bounds, the application of Lemma~\ref{leDerivBounds} and
Corollary~\ref{coIneqX} in order to bound the individual terms in
\eqref{eq:betamu} with our choices of function~$f$ and $x_0=\theta_0$ 
is now straightforwardly done in the same way as the above estimations
for~$\mu_3$ and~$\mu_4$.

\medskip
The monotonicity with respect to~$r$ is proved by noticing that both $X_2$ and $X_2/X_1$ are increasing with respect to~$r$; this is obvious for $X_2$, and we have
\[
\frac{\partial}{\partial r}\frac{X_2(n,r)}{X_1(n,r)}=\frac{X_3(n,r)X_1(n,r)-X_2^2(n,r)}{rX_1^2(n,r)}\geq0,
\]
where the last inequality is a consequence of the Cauchy--Schwarz inequality.
\end{proof}

\section{Bounding the tails}\label{seEps1}

The goal of this section is to provide a bound for the tail error term
$\epsilon_{1,Q_n}(r)=\epsilon_{1,P_n^\dd}(r)$. 
By the definition~\eqref{eq:ep1} of $\epsilon_{1,P_n^\dd}(r)$, what we
need is upper bounds for $\left|\frac{P_n(re^{i\theta})}{P_n(re^{2\pi i/3})}\right|$.
Phrased differently, the objective is to get good lower bounds for the quantity
\begin{align}
\notag
-\log\left|\frac{P_n(re^{i\theta})}{P_n(re^{2\pi i/3})}\right|&=
-\underset{3\nmid k}{\sum_{k=1}^{3n}} \log\left|\frac{1-\left(re^{i\theta}\right)^k}{1-r^ke^{2\pi i/3}}\right|\\
&=-\frac12\underset{3\nmid k}{\sum_{k=1}^{3n}} \log\frac{1-2r^k\cos(k\theta)+r^{2k}}{1+r^k+r^{2k}}
\label{eq:P/P-sum}
\end{align}
in terms of $\theta$, $r$, and $n$.
Depending on the ranges of these parameters, we shall in fact
establish two different lower bounds, presented in
Lemmas~\ref{leTailBound2}
and~\ref{leTailBound1} below. Lemma~\ref{leTailIneq2} provides a preliminary
estimate that is used in the proof of Lemma~\ref{leTailBound2}.
After these preparations, our bound for 
$\epsilon_{1,P_n^\dd}(r)$ is stated, and proved, in Lemma~\ref{leIneqTail}.

\medskip
In the following,
we shall use two possible lower bounds for the summand in~\eqref{eq:P/P-sum}:

\begin{enumerate}
  \item For $x\in[-1/3,1]$, we have $-\log(1-x)\geq x$.
In this inequality, we replace $x$ by $\frac{r^k}{1+r^k+r^{2k}}(1+2\cos(k\theta))$ to obtain
  \begin{equation}\label{eqSummandBoundMain}
-\log\frac{1-2r^k\cos(k\theta)+r^{2k}}{1+r^k+r^{2k}}\geq \frac{r^k}{1+r^k+r^{2k}}(1+2\cos(k\theta)).
  \end{equation}
  \item For $z\in\C$ with $|z|\leq1$, we have $|1-z^k|\leq
    k|1-z|$. Use of this inequality for $z=re^{i\theta}$ implies that
  \begin{equation}\label{eqSummandBoundAlternate}
-\log\frac{1-2r^k\cos(k\theta)+r^{2k}}{1+r^k+r^{2k}}\geq\log(1+r^k+r^{2k})-\log(1-2r\cos\theta+r^{2})-2\log k.
  \end{equation}
\end{enumerate}

\begin{lemma}\label{leTailIneq2}
For $r\in(0,1]$ and $\theta\in\R$, we have
\begin{equation} \label{eq:sum-cos} 
-\log\left|\frac{P_n(re^{i\theta})}{P_n(re^{2\pi i/3})}\right|\geq\frac{1}{3}\sum_{k=1}^{n}r^{3k}\left(1-\cos 3k\theta\right)-\frac{0.8}{|1-re^{i\theta}|}.
\end{equation}
\end{lemma}
\begin{proof}
We use \eqref{eqSummandBoundMain} to perform the following estimations:
\begin{align*}
&-\log\left|\frac{P_{n}(re^{i\theta})}{P_{n}(re^{2\pi i/3})}\right|\geq\underset{3\nmid k}{\sum_{k=1}^{3n}}\frac{r^k}{1+r^k+r^{2k}}\left(\frac12+\cos k\theta\right)\\
&\kern1cm
=\sum_{k=1}^{3n}\frac{r^k}{1+r^k+r^{2k}}\left(\frac12+\cos k\theta\right)-\sum_{k=1}^{n}\frac{r^{3k}}{1+r^{3k}+r^{6k}}\left(\frac12+\cos 3k\theta\right)\\
&\kern1cm
=\sum_{k=1}^{n}\frac{r^{3k}}{1+r^{3k}+r^{6k}}\left(1-\cos 3k\theta\right)+\sum_{k=1}^{3n}\frac{r^k\cos k\theta}{1+r^k+r^{2k}}\\
&\kern1cm
\quad+\frac12\sum_{k=1}^{n}\left(\frac{r^{3k-2}}{1+r^{3k-2}+r^{6k-4}}+\frac{r^{3k-1}}{1+r^{3k-1}+r^{6k-2}}-\frac{2r^{3k}}{1+r^{3k}+r^{6k}}\right)\\
&\kern1cm
\geq \frac{1}{3}\sum_{k=1}^{n}r^{3k}\left(1-\cos 3k\theta\right)+\sum_{k=1}^{3n}\frac{r^k\cos k\theta}{1+r^k+r^{2k}},
\end{align*}
where we used $1/(1+r^{3k}+r^{6k})\ge1/3$ and the fact that the
function $r^k/(1+r^k+r^{2k})$ is decreasing as a function in~$k$.
We apply Lemma~\ref{leCosSum1} with $\varphi=0$ to the last cosine sum to conclude that
\begin{align*}
\left|\sum_{k=1}^{3n}\frac{r^k\cos k\theta}{1+r^k+r^{2k}}\right|&\leq\frac{1}{|1-re^{i\theta}|}\left((1-r)\sum_{k=1}^{3n}\frac{r^{k}}{1+r^{k}+r^{2k}}+2\frac{r^{3n+1}}{1+r^{3n}+r^{6n}}\right)\\
&\leq\frac{1}{|1-re^{i\theta}|}\left((1-r)\int_{0}^{3n}\frac{r^{k}\,dk}{1+r^{k}+r^{2k}}+2\frac{r^{3n}}{1+r^{3n}+r^{6n}}\right)\\
&=\frac{1}{|1-re^{i\theta}|}\left(\frac{1-r}{-\log r}\frac{2}{\sqrt3}\left(\frac{\pi}{3}-\arctan\frac{1+2r^{3n}}{\sqrt3}\right)+2\frac{r^{3n}}{1+r^{3n}+r^{6n}}\right)\\
&<\frac{1}{|1-re^{i\theta}|}\left(\frac{2}{\sqrt3}\left(\frac{\pi}{3}-\arctan\frac{1+2r^{3n}}{\sqrt3}\right)+2\frac{r^{3n}}{1+r^{3n}+r^{6n}}\right).
\end{align*}

In order to complete the proof, we determine the maximum value of the function
\begin{equation}
f(s):=\frac{2\pi}{3\sqrt3}-\frac{2}{\sqrt3}\arctan\frac{1+2s}{\sqrt3}+\frac{2s}{1+s+s^2}
\end{equation}
on $[0,1]$.
Since $f'(s)=\frac{1-s-3s^2}{(1+s+s^2)^2}$ is decreasing with respect
to~$s$, we see that the unique maximum point of $f$ is located at the
unique zero of $f'(s)$ in $[0,1]$, namely\break $s_0=(\sqrt{13}-1)/6$, giving a value of
\[
f(s_0)\approx0.7937<0.8.
\qedhere
\]
\end{proof}

In order to find a closed-form lower bound
for the quantity $-\log\left|\frac{P_n(re^{i\theta})}{P_n(re^{2\pi
    i/3})}\right|$, we apply Lemma~\ref{leIneqCosSum1} to the sum on
the right-hand side of \eqref{eq:sum-cos}. In this manner, we obtain
the following estimate.

\begin{lemma}\label{leTailBound2}
If $\theta=2h\pi/3+\rho\frac{1-r^3}{1-r^{3n}}$ for some $h\in\Z$ and some $\rho\in\R$ such that $|\rho|\frac{1-r^3}{1-r^{3n}}\leq\pi/3$, then we have
\[
-\log\left|\frac{P_n(re^{i\theta})}{P_n(re^{2\pi i/3})}\right|\geq-\frac{0.8}{|1-re^{i\theta}|}+\frac{r^{3}(1+r^3)}{6}\frac{(1-r^{n/2})}{(1-r^{3})}\left(1-\sqrt\frac{1}{1+18\rho^2}\right).
\]
\end{lemma}

\begin{rem}
The slightly unusual looking scaling of the deviation of $\theta$ from
$2h\pi/3$ above has its motivation in the desire of having the same
scaling as in the definition of the cutoff~$\theta_0$;
cf.~\eqref{eqCutoffTheta0}
(remember that $r$ depends on~$n$ and~$m$!).
\end{rem}

\begin{proof}[Proof of Lemma \ref{leTailBound2}]
Lemmas~\ref{leTailIneq2} and \ref{leIneqCosSum1} (with the substitutions $r\mapsto r^3$, $\theta\mapsto 3\theta$) imply the inequality
\begin{align*}
-\log\left|\frac{P_n(re^{i\theta})}{P_n(re^{2\pi i/3})}\right|&\geq\frac{1}{3}\sum_{k=1}^{n}r^{3k}\left(1-\cos 3k\theta\right)-\frac{0.8}{|1-re^{i\theta}|}\\
&\geq \frac{r^{3}}{3}\frac{1-r^{3n}}{1-r^3}\left(1-\sqrt{\frac{1}{1+4\kappa\tan^2(3\theta/2)}}\right)-\frac{0.8}{|1-re^{i\theta}|},
\end{align*}
where
\[
\kappa=\frac{(1+r^3)(1-r^{3n})(1-r^{n/2})}{(1-r^3)^2}.
\]

We note that
\[
\left|\tan\frac{3\theta}2\right|=\left|\tan \frac{3\rho}{2}\frac{1-r^3}{1-r^{3n}}\right|\geq\frac{3|\rho|}{2}\frac{1-r^3}{1-r^{3n}}
\]
if $|\rho|\frac{1-r^3}{1-r^{3n}}\leq\pi/3$. 
We use this inequality to get rid of the tangent function:
\[
1-\sqrt{\frac{1}{1+4\kappa\tan^2(3\theta/2)}}> 1-\sqrt{\frac{1}{1+\kappa\frac{9(1-r^3)^2}{(1-r^{3n})^2}\rho^2}}.
\]

By making use of the inequality
\[
1-\sqrt\frac{1}{1+cx}\geq c\left(1-\sqrt\frac{1}{1+x}\right)
\]
for $0<c\leq1$ and $x>0$, and by choosing
\[
c=\frac \kappa2\frac{(1-r^3)^2}{(1-r^{3n})^2}=\frac{(1+r^3)}{2}\frac{(1-r^{n/2})}{(1-r^{3n})}\leq1,
\]
we arrive at the claimed result:
\begin{align*}
-\log\left|\frac{P_{n}(re^{i\theta})}{P_{n}(re^{2\pi i/3})}\right|&\geq-\frac{0.8}{|1-re^{i\theta}|}+\frac{r^{3}(1+r^3)}{6}\frac{(1-r^{n/2})}{(1-r^{3})}\left(1-\sqrt\frac{1}{1+18\rho^2}\right).
\qedhere
\end{align*}
\end{proof}

Note that the lower bound in Lemma~\ref{leTailBound2} ceases
to be effective when $|1-re^{i\theta}|$ is small. For this case, we
present an alternative bound.

\begin{lemma}\label{leTailBound1}
If\/ $|1-re^{i\theta}|<\frac13$, then we have
\[
-\log\left|\frac{P_n(re^{i\theta})}{P_n(re^{2\pi i/3})}\right|\geq\frac16\frac{(r+r^2)(1-r^{3n})}{1-r^3}-5.44.
\]
\end{lemma}
\begin{proof}
Making reference to the sum representation \eqref{eq:P/P-sum},
we define a cutoff 
$$k_0=\min\left\{\left\lfloor \frac1{3|1-re^{i\theta}|}\right\rfloor,n\right\}.$$
Note that the condition on $|1-re^{i\theta}|$ implies that $k_0\geq1$.

\medskip
The part of the sum on the right-hand side of \eqref{eq:P/P-sum} where
$k<3k_0$ is treated by \eqref{eqSummandBoundAlternate}: 
\begin{align}
\notag
-\frac12\underset{3\nmid k}{\sum_{k=1}^{3k_0}} &\log\frac{1-2r^k\cos(k\theta)+r^{2k}}{1+r^k+r^{2k}}\\
&\geq \frac12\underset{3\nmid k}{\sum_{k=1}^{3k_0}}\left(\log(1+r^k+r^{2k})-\log(1-2r\cos\theta+r^{2})-2\log k\right)\nonumber\\
&=-k_0\log(1-2r\cos\theta+r^{2})-\log\frac {(3k)!} {3^kk!}+
\frac12\underset{3\nmid k}{\sum_{k=1}^{3k_0}}\log(1+r^k+r^{2k}).\nonumber
\end{align}
Now we use the inequality
$\frac{(3k)!}{3^kk!}<\sqrt3(3k/e)^{2k}$,
and the convexity of $k\mapsto \log(1+r^k+r^{2k})$, and obtain
\begin{align}
\notag
-{}&\frac12\underset{3\nmid k}{\sum_{k=1}^{3k_0}} \log\frac{1-2r^k\cos(k\theta)+r^{2k}}{1+r^k+r^{2k}}\\
&>-k_0\log(1-2r\cos\theta+r^{2})-\frac12\log3-2k_0(\log(3k_0)-1)+
\frac12\underset{3\nmid k}{\sum_{k=1}^{3k_0}}\log(1+r^k+r^{2k})\nonumber\\
&>-k_0\log(1-2r\cos\theta+r^{2})-\frac12\log3-2k_0(\log(3k_0)-1)+k_0\log(1+r^{3k_0/2}+r^{3k_0})\nonumber\\
&=-2k_0\log(3k_0|1-re^{i\theta}|)+2k_0-\frac12\log3+k_0\log(1+r^{3k_0/2}+r^{3k_0})\nonumber\\
&\geq 2k_0-\frac12\log3+k_0\log(1+r^{3k_0/2}+r^{3k_0}),\label{eqTailPart1}
\end{align}
%
where we used the definition of $k_0$ to get the last line.

\medskip
For the part where $k>3k_0$, we use \eqref{eqSummandBoundMain}, split
the sum according to the residue classes of $k$ modulo~$3$, 
and apply Lemma~\ref{leCosSum1} to each subsum, to get
\begin{multline}
-\frac12\underset{3\nmid k}{\sum_{k=3k_0+1}^{3n}} \log\frac{1-2r^k\cos(k\theta)+r^{2k}}{1+r^k+r^{2k}}\geq\frac12\underset{3\nmid k}{\sum_{k=3k_0+1}^{3n}} \frac{r^k(1+2\cos(k\theta))}{1+r^k+r^{2k}}\\
\geq \left(\frac12-\frac{1-r^3}{|1-r^3e^{3i\theta}|}\right)
\underset{3\nmid k}{\sum_{k=3k_0+1}^{3n}} \frac{r^k}{1+r^k+r^{2k}}-\frac{4}{|1-r^3e^{3i\theta}|}\frac{r^{3n}}{1+r^{3n}+r^{6n}}. \label{eqTailPart2}
\end{multline}

We first observe that in the case where $k_0=n$ the
estimate~\eqref{eqTailPart1} provides the lower bound 
\[
2n-\frac12\log3\geq \frac{(r+r^2)(1-r^{3n})}{1-r^3}-\frac12\log3>\frac16\frac{(r+r^2)(1-r^{3n})}{1-r^3}-5.44,
\]
as desired.

Therefore, we assume $0<k_0<n$ from now on. By combining \eqref{eqTailPart1} and \eqref{eqTailPart2}, we obtain
\begin{multline}
-\log\left|\frac{P_n(re^{i\theta})}{P_n(re^{2\pi i/3})}\right|\geq \left(\frac12-\frac{1-r^3}{|1-r^3e^{3i\theta}|}\right)\underset{3\nmid k}{\sum_{k=3k_0}^{3n}} \frac{r^k}{1+r^k+r^{2k}}\\
+(2+\log(1+r^{3k_0/2}+r^{3k_0}))k_0-\frac{4}{|1-r^3e^{3i\theta}|}\frac{r^{3n}}{1+r^{3n}+r^{6n}}-\frac12\log3. \label{eqTailBound1}
\end{multline}

We split the right-hand side of \eqref{eqTailBound1} into several parts:
\[
-\log\left|\frac{P_n(re^{i\theta})}{P_n(re^{2\pi i/3})}\right|\geq I_1+I_2+I_3+I_4,
\]
where
\begin{align*}
  I_1 &= \frac12 \underset{3\nmid k}{\sum_{k=3k_0}^{3n}}
  \frac{r^k}{1+r^k+r^{2k}}+\frac13k_0,
\\
  I_2 &= -\frac{1-r^3}{|1-r^3e^{3i\theta}|} 
\underset{3\nmid k}{\sum_{k=3k_0+1}^{3n}} \frac{r^k}{1+r^k+r^{2k}},
\\
  I_3 &= k_0(\log(1+r^{3k_0/2}+r^{3k_0})-\log(1+r^{3n/2}+r^{3n})),\\
  I_4 &= \left(\frac53+\log(1+r^{3n/2}+r^{3n})\right)k_0-\frac{4}{|1-r^3e^{3i\theta}|}\frac{r^{3n}}{1+r^{3n}+r^{6n}}-\log\sqrt3.
\end{align*}

For $I_1$ we have
$$
  I_1 \geq\frac12 \underset{3\nmid k}{\sum_{k=1}^{3n}} \frac{r^k}{1+r^k+r^{2k}}
  \geq\frac12 \underset{3\nmid k}{\sum_{k=1}^{3n}} \frac{r^k}{3}=\frac16\frac{(r+r^2)(1-r^{3n})}{1-r^3}.
$$
It should be noted that the right-hand side in this inequality 
is exactly the main term in the desired lower bound.
Consequently, what we need to prove is $I_2+I_3+I_4\geq -5.44$.

From here on, we write $z=re^{i\theta}$ for simplicity of notation.

We first deal with $I_4$. By utilising the inequality 
$$\left|\log(1+\sqrt s+s)-\frac{3s}{1+s+s^2}\right|\leq\frac{1}{10},
\qquad 0\le s\le 1,$$ 
for $s=r^{3n}$, we infer that
$$
I_4\geq \left(\frac{47}{30}+\frac{3r^{3n}}{1+r^{3n}+r^{6n}}\right)k_0-\frac{12(k_0+1)}{|1+z+z^2|}\frac{r^{3n}}{1+r^{3n}+r^{6n}}-\log\sqrt3.
$$
Now we note that for $0<k_0<n$ we have
$$k_0=\left\lfloor \frac1{3|1-z|}\right\rfloor
\geq\frac{1}{3|1-z|}-1.$$
We use this in the above estimate for $I_4$ to get
\begin{align*}
I_4
&\geq
\frac{47}{30}\left(\frac{1}{3|1-z|}-1\right)+\frac{3r^{3n}\left(\frac{1}{3|1-z|}-1\right)}{1+r^{3n}+r^{6n}}\\
&\kern4cm
-\frac{12}{3|1-z|\cdot|1+z+z^2|}\frac{r^{3n}}{1+r^{3n}+r^{6n}}-\log\sqrt3\\
&> \left(\frac{47}{30}-\left(\frac{12}{|1+z+z^2|}-3\right)\frac{r^{3n}}{1+r^{3n}+r^{6n}}\right)\frac{1}{3|1-z|}-\log\sqrt3-\frac{47}{30}-\frac{3r^{3n}}{1+r^{3n}+r^{6n}}\\
&\geq \left(\frac{47}{30}-\left(\frac{12}{|1+z+z^2|}-3\right)\frac13\right)\frac{1}{3|1-z|}-\log\sqrt3-\frac{77}{30}\\
&=\left(\frac{77}{30}-\frac{4}{|1+z+z^2|}\right)\frac{1}{3|1-z|}-\log\sqrt3-\frac{77}{30}.
\end{align*}

In order to bound $I_2$, we argue that
\begin{align*}
\underset{3\nmid k}{\sum_{k=3k_0+1}^{3n}} \frac{r^k}{1+r^k+r^{2k}}&\leq 
\underset{3\nmid k}{\sum_{k=3k_0+1}^{3n}} r^k=\frac{(r+r^2)(r^{3k_0}-r^{3n})}{1-r^3}\leq \frac{2(r^{3k_0}-r^{3n})}{1-r^3},
\end{align*}
and consequently
\begin{align*}
I_2&\geq-\frac{2}{|1-z^3|}(r^{3k_0}-r^{3n}).
\end{align*}

Writing $h(x)=\frac{6}{|1+z+z^2|}x-\log(1+\sqrt{x}+x)$,
we combine the above estimate for $I_2$ into one for $I_2+I_3$:
\begin{align*}
I_2+I_3&\ge-\frac{2}{|1-z^3|}(r^{3k_0}-r^{3n})\\
&\kern2cm
+\left(\frac {1}
{3|1-z|}-1\right)(\log(1+r^{3k_0/2}+r^{3k_0})-\log(1+r^{3n/2}+r^{3n}))\\
&=
-\left(h(r^{3k_0})-h(r^{3n})\right)\frac{1}{3|1-z|}
-(\log(1+r^{3k_0/2}+r^{3k_0})-\log(1+r^{3n/2}+r^{3n}))\\
&\geq -\left(h(r^{3k_0})-h(r^{3n})\right)\frac{1}{3|1-z|}-\log3\\
&\geq -\left(\max_{0\leq x\leq1}h(x)-\min_{0\leq x\leq1}h(x)\right)\frac{1}{3|1-z|}-\log3.
\end{align*}
Note that the function $h$ is convex with respect to~$x$. Hence,
the maximum of $h(x)$ is either $h(0)$ or $h(1)$. Since
$|1+z+z^2|\leq3$, we have $h(1)\geq2-\log3>0=h(0)$. Therefore, 
$$\max_{0\leq x\leq1}h(x)=h(1)=\frac{6}{|1+z+z^2|}-\log3.$$ 
On the other hand, again using that $|1+z+z^2|\le3$, we have
\begin{align*}
\min_{0\leq x\leq1}h(x)\geq \min_{0\leq x\leq1}(2x-\log(1+\sqrt{x}+x))\approx-0.1496>-\frac{3}{20},
\end{align*}
which in turn implies
\[
I_2+I_3\geq -\left(\frac{6}{|1+z+z^2|}-\log3+\frac{3}{20}\right)\frac{1}{3|1-z|}-\log3.
\]

Combining all the inequalities above, we obtain 
\begin{align*}
I_2+I_3+I_4&\geq \left(\frac{29}{12}+\log3-\frac{10}{|1+z+z^2|}\right)\frac{1}{3|1-z|}-\frac32\log3-\frac{77}{30}.
\end{align*}

We write $u=|1-z|$. By the assumptions of the lemma, we have
$u\in[0,1/3]$. We claim that $|1+z+z^2|\geq3-3u+u^2$ for $u\in[0,1/3]$.
This can be proved by writing 
$1-z=u e^{i\varphi}$ for some $\varphi$, expressing $z$ in terms of
$u$ and $\varphi$, and minimising $|1+z+z^2|$ with respect
to~$\varphi$. 
In addition, we point out that the function $u\mapsto \left(\frac{29}{12}+\log3-\frac{10}{3-3u+u^2}\right)\frac{1}{3u}$ is decreasing with respect to~$u$, and therefore
\begin{align*}
I_2+I_3+I_4&\geq \left(\frac{29}{12}+\log3-\frac{10}{3-3u+u^2}\right)\frac{1}{3u}-\frac32\log3-\frac{77}{30}\\
&\geq \frac{29}{12}+\log3-\frac{90}{19}-\frac32\log3-\frac{77}{30}\\
&=-\frac{1857}{380}-\frac12\log3>-5.44,
\end{align*}
as desired.
\end{proof}

%

We are now ready to provide, and prove, an explicit upper bound 
for the tail error term $\epsilon_{1,P}(n,r)$ as defined in~\eqref{eq:ep1}.

\begin{lemma}\label{leIneqTail}
Suppose that $n\in\Z^+$, and that $r_0$ is defined as in
\eqref{eqCutoffR0}. Then, for $\dd \in\{1,2,3\}$ and $r\in(r_0,1]$, we have
\begin{multline}
\epsilon_{1,P_n^\dd }(r)
<\frac {\sqrt{54\dd}} {\sqrt{5\pi}}\left/\erf{\sqrt{\frac{40\dd(1-r^{3n})}{243(1-r^3)}}}\right.\\
\times\left(4\left(\frac{1-r^{3n}}{1-r^3}\right)^{1/2}\int^{4}_{10/81}\exp\left(\frac{0.8\dd }{\sqrt3-(1+3\rho)(-\log r_0)}-\dd \phi(n,r,\rho)\right)\,d\rho\right.\\
\left.{}+2\pi\left(\frac{1-r^{3n}}{1-r^3}\right)^{3/2}\exp(5.44\dd -\dd \phi(n,r,4))\right), \label{eqIneqTail}
\end{multline}
where
\[
\phi(n,r,\rho):=\frac{r^3(1+r^3)}{6}\left(1-\sqrt\frac{1}{1+18\rho^2}\right)\frac{(1-r^{n/2})}{(1-r^{3})}.
\]
Moreover, for $n>546$, the right-hand side of \eqref{eqIneqTail} is decreasing with respect to~$r$.
\end{lemma}
\begin{proof}
Lemmas~\ref{leTailBound2} and \ref{leTailBound1} imply that
\begin{align}
\notag
-\log\left|\frac{P_{n}(re^{i\theta})}{P_{n}(re^{2\pi i/3})}\right|&>
-\frac{0.8}{|1-re^{i\theta}|}+\phi(n,r,|\rho|), \\
\label{eq:0.8}
&\kern4cm
 \text{ for } \theta=\pm\frac{2\pi i}{3}+\rho\frac{1-r^{3}}{1-r^{3n}}\text{ and }|1-re^{i\theta}|\geq\frac13,\\
-\log\left|\frac{P_{n}(re^{i\theta})}{P_{n}(re^{2\pi i/3})}\right|&>
-5.44+\phi(n,r,+\infty), \quad  \text{ for } |1-re^{i\theta}|<\frac13.
\label{eq:5.44}
\end{align}

For $\theta:=\pm\frac{2\pi i}{3}+\rho\frac{1-r^{3}}{1-r^{3n}}$,
we have
\begin{align}
\notag
|1-re^{i\theta}|&=\left|(1-e^{\pm2\pi i/3})+(e^{\pm2\pi
  i/3}-e^{i\theta})+(e^{i\theta}-re^{i\theta})\right|\\
\notag
&\ge|1-e^{\pm2\pi i/3}|-|e^{\pm2\pi
  i/3}-e^{i\theta}|-|e^{i\theta}-re^{i\theta}|\\
\notag
&\geq \sqrt3-|\rho|\frac{1-r^{3}}{1-r^{3n}}-(1-r)
\geq \sqrt3-|\rho|\frac{1-r_0^3}{1-r_0^{3n}}-(1-r_0)\\
&\geq \sqrt3-(3|\rho|+1)(-\log r_0),
\label{eq:1-reit}
\end{align}
where we used that $r\in(r_0,1]$ to get the next-to-last line, and
\eqref{eqThetaBoundAux} to obtain the last line.

Here, in order to estimate the integral in \eqref{eq:ep1}, we divide
the tail part $I_{\text{tail}}$ into two disjoint subsets. Namely,
we define
$$I_{\text{tail1}}:=\left\{\pm2\pi/3+\rho\frac{1-r^{3}}{1-r^{3n}}:C_0<
|\rho|<4\right\}$$
and the complementary subset
$I_{\text{tail2}}:=I_{\text{tail}}\backslash I_{\text{tail1}}$.
The set $I_{\text{tail1}}$ consists of four distinct intervals. 
By \eqref{eq:0.8} and~\eqref{eq:1-reit}, 
the integral over these intervals can be estimated by
      \begin{multline}
        \int_{I_\text{tail1}}\left|\frac{Q_n(re^{i\theta})}{Q_n(re^{2\pi
            i/3})}\right|\,d\theta \\
 < 4\frac{1-r^{3}}{1-r^{3n}}\int^{4}_{C_0}\exp\left(\frac{0.8\dd }{\sqrt3-(1+3|\rho|)(-\log r_0)}-\dd \phi(n,r,\rho)\right)\,d\rho.
\label{eq:tail1}
      \end{multline}
For the remaining part of $I_{\text{tail}}$, 
$I_{\text{tail2}}:=I_{\text{tail}}\setminus I_{\text{tail1}}$, we note that the quantity
  \[
  -\log\left|\frac{P_{n}(re^{i\theta})}{P_{n}(re^{2\pi i/3})}\right|
  \]
  can be bounded below by either $-2.4+\phi(n,r,4)$ (if
  $|1-re^{i\theta}|\geq\frac13$, using~\eqref{eq:0.8}) or
  $-5.44+\phi(n,r,+\infty)$ (if $|1-re^{i\theta}|<\frac13$, using~\eqref{eq:5.44}), and a common lower bound for the two cases can be chosen as $-5.44+\phi(n,r,4)$. This implies that
      \begin{equation}
        \int_{I_\text{tail2}}\left|\frac{Q_n(re^{i\theta})}{Q_n(re^{2\pi i/3})}\right|\,d\theta  < 2\pi\exp(5.44\dd -\dd \phi(n,r,4)).
\label{eq:tail2}
      \end{equation}  

By combining the two bounds \eqref{eq:tail1} and \eqref{eq:tail2}, 
we obtain the following upper bound for
the integral in \eqref{eq:ep1}: 
\begin{multline*}
\int_{I_\text{tail}}\left|\frac{Q_n(re^{i\theta})}{Q_n(re^{2\pi
    i/3})}\right|\,d\theta\\
<4\frac{1-r^{3}}{1-r^{3n}}\int^{4}_{10/81}\exp\left(\frac{0.8\dd }{\sqrt3-(1+3|\rho|)(-\log r_0)}-\dd \phi(n,r,\rho)\right)\,d\rho\\
 +2\pi\exp(5.44\dd -\dd \phi(n,r,4)).
\end{multline*}

We recall that the definition \eqref{eq:ep1} of ${\epsilon}_{1,Q_n}(r)$
contains the factor 
\[
\sqrt{g_{Q_n}(r)}\left/\erf\left({\theta_0\sqrt{g_{Q_n}(r)/2}}\right)\right.
\]
in addition to the left-hand
side of the above inequality. We note that, using the upper
bound for $-\Re f_2$ in \eqref{eq:fX2} and the inequality
\eqref{eqIneqR2_000}, we have 
\[
g_{P_n^\dd }(r)<\frac{108\dd }{5}\left(\frac{1-r^{3n}}{1-r^3}\right)^3.
\]
Therefore, using the fact that $x/\erf x$ is increasing with respect to $x$ and recalling the definition of $\theta_0$ in \eqref{eqCutoffTheta0}, we obtain 
\begin{multline*}
\epsilon_{1,P_n^\dd }(r)
<\frac {\sqrt{54\dd}} {\sqrt{5\pi}}\left/\erf\left({\sqrt{\frac{40\dd(1-r^{3n})}{243(1-r^3)}}}\right)\right.\\
\times\left(4\left(\frac{1-r^{3n}}{1-r^3}\right)^{1/2}\int^{4}_{10/81}\exp\left(\frac{0.8\dd }{\sqrt3-(1+3\rho)(-\log r_0)}-\dd \phi(n,r,\rho)\right)\,d\rho\right.\\
\left.+2\pi\left(\frac{1-r^{3n}}{1-r^3}\right)^{3/2}\exp(5.44\dd -\dd \phi(n,r,4))\right),
\end{multline*}
as desired.

\medskip
It remains to show that the right-hand side of \eqref{eqIneqTail} is
decreasing with respect to~$r$. To this end, we first note that the factor
$1/\erf\left({\sqrt{\frac{40\dd(1-r^{3n})}{243(1-r^3)}}}\right)$ is decreasing with respect to $r$. 

We claim that the other factor on the right-hand side of
\eqref{eqIneqTail} is also decreasing with respect to~$r$. To see
this, let $r_1,r_2\in[r_0,1]$ such that $r_1<r_2$. We then use
Lemma~\ref{leMonotone1} with $r$ replaced by $r^3$ and
\begin{equation} \label{eq:lambda} 
\lambda=C\dd\frac{r_1^3(1+r_1^3)}{6}\left(1-\sqrt{\frac{1}{1+18\rho^2}}\right),
\end{equation}
to get
$$
\frac{1-r_2^{3n}}{1-r_2^3}\exp
\left(-C\dd\frac{r_1^3(1+r_1^3)}{r_2^3(1+r_2^3)}\phi(n,r_2,\rho)\right)
\le
\frac{1-r_1^{3n}}{1-r_1^3}\exp
\left(-C\dd\phi(n,r_1,\rho)\right),
$$
provided
\begin{equation} \label{eq:546}
546\ge 6+\frac {36} {\lambda}.
\end{equation}
(Recall that $n>546$ by assumption.)

Let us for the moment assume that the condition~\eqref{eq:546} is satisfied.
Then, since $r_2^3(1+r_2^3)>r_1^3(1+r_1^3)$, we obtain
\begin{equation} \label{eq:r-r0}
\frac{1-r_2^{3n}}{1-r_2^3}\exp
\left(-C\dd\phi(n,r_2,\rho)\right)
\le
\frac{1-r_1^{3n}}{1-r_1^3}\exp
\left(-C\dd\phi(n,r_1,\rho)\right),
\end{equation}
again provided \eqref{eq:546} holds.
It can be checked that, for $C=2$, the inequality~\eqref{eq:546} holds for
$\frac{10}{81}\leq\rho\leq 4$. Therefore, setting $C=2$ in
\eqref{eq:r-r0} and taking square roots of both sides, we obtain
\begin{multline}\label{eqMono1}
\left(\frac{1-r_2^{3n}}{1-r_2^3}\right)^{1/2}\exp\left(-\dd \phi(n,r_2,\rho)\right)\leq
\left(\frac{1-r_1^{3n}}{1-r_1^3}\right)^{1/2}\exp\left(-\dd \phi(n,r_1,\rho)\right),
\\
\text{for }\frac{10}{81}\leq\rho\leq 4.
\end{multline}
For $C=2/3$, the inequality \eqref{eq:546} only holds for $\rho=4$.
By doing these substitutions in \eqref{eq:r-r0} and raising both sides
to the power~$3/2$, we arrive at
\begin{equation}\label{eqMono2}
\left(\frac{1-r_2^{3n}}{1-r_2^3}\right)^{3/2}\exp\left(-\dd \phi(n,r_2,4)\right)\leq \left(\frac{1-r_1^{3n}}{1-r_1^3}\right)^{3/2}\exp\left(-\dd \phi(n,r_1,4)\right).
\end{equation}
The inequalities \eqref{eqMono1} and \eqref{eqMono2} together show
that the second factor on the right-hand side of \eqref{eqIneqTail}
is indeed also decreasing in~$r$.

\medskip
It remains to justify the use of Lemma~\ref{leMonotone1}, that is, of
the validity of the condition~\eqref{eq:546}.

\begin{itemize}
  \item We note that $n>546$ implies that
  \[
  r_0^3=\exp\left(-3\sqrt{\frac{4\dd}{27n}}\right)>\exp\left(-3\sqrt{\frac{12}{27\times546}}\right)>\frac{11}{12},
  \]
  and consequently
  \[
  \frac{r_1^3(1+r_1^3)}{6}\geq \frac{r_0^3(1+r_0^3)}{6}>0.292.
  \]
  \item Therefore, with the choice $C=2$ and $10/81\le \rho\le 4$,
the constant $\lambda$ in \eqref{eq:lambda} is at least 
\[
  2\times0.292\dd  (1-(1+18(10/81)^2)^{-1/2})>\dd /15.
  \]
Hence, the condition \eqref{eq:546} holds, which confirms
\eqref{eqMono1}. On the other hand, with the choice $C=2/3$ 
and $\rho= 4$,
the constant $\lambda$ in \eqref{eq:lambda} is at least 
  \[
  2/3 \times0.292 \dd (1-(1+18\times4^2)^{-1/2})>\dd /6.
  \]
Hence, again, the condition \eqref{eq:546} is satisfied, confirming
\eqref{eqMono2}. 
  \item The condition in Lemma~\ref{leMonotone1} on the range of $r$ is verified by noting that
  \[
  -\log r^3<\frac{\sqrt{\dd }}{9\sqrt{5}}<\frac{\sqrt{\dd }}{20}\leq\frac{\dd }{20}<\frac{8}{9}\lambda.
\qedhere  \]
\end{itemize}
\end{proof}

\section{Completion of the proofs}\label{seMain}
In this section, we combine the results of the two previous sections
to prove the First and Second Borwein Conjecture and ``two thirds'' of 
the cubic Borwein conjecture. We begin by giving a result that allows
us to control the
argument of $P_n(re^{2\pi i/3})$. As mentioned in Part~D of
Section~\ref{seOutline}, this is needed for accomplishing Task~(2)
below~\eqref{eqFinalTarget}.

\begin{lemma}\label{leArg}
For $n\in \Z^+$, $\arg P_n(re^{2\pi i/3})$ is increasing with respect to~$r$. Moreover, for $r\in(0,1]$ and $n\in\Z^+$, we have $\arg P_n(re^{2\pi i/3})\in(-\pi/18,0]$.
\end{lemma}
\begin{proof}
For $x\in\R$, define
\[
f(r,x):=\arg(1-r^{x}e^{2\pi i/3})=-\arctan\frac{\sqrt3r^x}{r^x+2}.
\]
By elementary manipulations, we have
\begin{align*}
\arg P_n(re^{2\pi i/3})&\equiv\sum_{k=1}^{n}\left(\arg(1-r^{3k-2}e^{2\pi i/3})+\arg(1-r^{3k-1}e^{-2\pi i/3})\right)\pmod{2\pi}\\
&=\sum_{k=1}^{n}\left(\arg(1-r^{3k-2}e^{2\pi i/3})-\arg(1-r^{3k-1}e^{2\pi i/3})\right)\\
&=-\sum_{k=1}^{n}\left(f(r,3k-1)-f(r,3k-2)\right).
\end{align*}

We claim that $f(r,3k-1)-f(r,3k-2)$ is decreasing with respect to~$r$,
and that
$$\sum_{k=1}^{n}\left(f(r,3k-1)-f(r,3k-2)\right)\in[0,\pi/18).$$
In order to see this, we note that
\[
\sum_{k=1}^{n}\left(f(r,3k-1)-f(r,3k-2)\right)=\sum_{k=1}^{n}\int^{3k-1}_{3k-2}f_x(r,x)\,dx,
\]
where as usual $f_x(r,x)=\frac {\partial} {\partial x}f(r,x)$.
Both the lower bound of~$0$ and the monotonicity with respect to~$r$ follow from the expression
\[
f_x(r,x)=\frac{\sqrt3r^x(-\log r)}{2(1+r^x+r^{2x})}.
\]

In order to prove the upper bound of $\pi/18$, we define
$g(r,x):=\sum_{k\in\Z}f_x(r,3k+x)$ and claim that
\begin{equation} \label{eq:12-03} 
\int^2_1 g(r,x)\,dx\leq \frac13\int^3_0g(r,x)\,dx.
\end{equation}
If we assume the truth of this inequality for a moment, then,
since $f_x$ is even with respect to~$x$, we see that

\begin{align*}
\sum_{k=1}^{n}\left(f(r,3k-1)-f(r,3k-2)\right)&<\frac12\sum_{k\in\Z}\left(f(r,3k-1)-f(r,3k-2)\right)\\
&\kern-4pt
=\frac12\int^2_1 g(r,x)\,dx\leq \frac16\int^3_0g(r,x)\,dx=\frac16\int_{-\infty}^{\infty} f_x(r,x)\,dx\\
&\kern-4pt
=\frac16f(r,x)\Big|^{+\infty}_{-\infty}=\pi/18,
\end{align*}
as required.

Hence, it remains to verify \eqref{eq:12-03}.
As a matter of fact, this inequality can be proved by a Fourier expansion of $g(r,x)$. 
To be precise, we define
\[
g_k(r):=\int^3_0 g(r,x)\cos(2\pi k x/3)\,dx=\int_{\R}f_x(r,x)\cos(2\pi k x/3)\,dx,
\]
so that
\[
g(r,x)=\frac13g_0(r)+\frac23\sum_{k=1}^{\infty}g_k(r) \cos(2k\pi x/3).
\]
To get an explicit expression for $g_k(r)$, we note that, since
$f_x(r,x)$ is even, we may express the Fourier coefficients as
\[
g_k(r)=\int_{\R}f_x(r,x)\exp(2\pi k i x/3)\,dx.
\]
We integrate the function $f_x(r,x)\exp(2\pi k i x/3)$ (clockwise)
along a rectangular contour with corners located at $\pm M$ and $\pm
M-2\pi i/(-\log r)$. In the limit as $M\to\infty$, the
integral along the two vertical parts of the contour converges to
zero, while the two parts of the integral along the horizontal parts
of the contour are proportional to each other. More precisely, we may
conclude that the integral along this rectangular contour, in the
limit as $M\to\infty$, 
is equal to 
$$\left(\exp\left(4k\pi^2/(-3\log r)\right)-1\right)\cdot
g_k(r).$$ 
The integrand has exactly two poles inside this rectangle,
namely at $x=-2\pi i/(-3\log r)$ and at $x=-4\pi i/(-3\log r)$, with
residues equal to $i \exp(4k\pi^2/(-9\log r))$ and to $-i
\exp(8k\pi^2/(-9\log r))$, respectively. Therefore we obtain that
\begin{align*}
g_k(r)&=\frac{\pi}{1+2\cosh\left(\frac{4k\pi^2}{9(-\log r)}\right)}.
\end{align*}
We are now in the position to accomplish a proof of \eqref{eq:12-03}
by employing the above facts: 
\begin{align*}
\left(\frac13\int^3_0-\int^2_1\right)g(r,x)\,dx&=\frac23\sum_{k=1}^{\infty}g_k(r)\left(\frac13\int^3_0-\int^2_1\right)\cos(2k\pi x/3)\,dx\\
&=\sum_{k=1}^{\infty}\frac{8(-1)^{k-1}g_k(r)\sin^3(k\pi/3)}{3k\pi}\\
&=\frac{\sqrt3}{\pi}\left(g_1(r)-\frac{g_2(r)}{2}+\frac{g_4(r)}{4}-\frac{g_5(r)}{5}+\cdots\right)>0,
\end{align*}
where the last inequality is due to the fact that $g_k(r)$ is decreasing with respect to~$k$.
\end{proof}
With concrete bounds on $\arg P_n(re^{2\pi i/3})$ proven, all three
pieces of the Borwein puzzle are now in place, and we can now 
present the announced proofs of the
First and Second Borwein Conjecture, and of ``two thirds'' of the Cubic
Borwein Conjecture.

We begin with the (in view of \cite{WANG2022108028}: alternative) 
proof of the First Borwein Conjecture.
In the arguments below, we always use $r_m$ to denote the solution of
the approximate saddle point equation~\eqref{eqStationaryPoint2} (that
depends on $n$, $m$, and $\dd$). 

\begin{theorem}\label{thMain1}
The First Borwein Conjecture,  Conjecture~\ref{cjBorwein}, is true.
\end{theorem}
\begin{proof}
We prove this claim by verifying \eqref{eqFinalTarget} for
``large''~$n$, with the help of the various bounds and inequalities we have derived,
and by a direct computation for ``small''~$n$, using the computer.

By Lemma~\ref{leArg}, we have $\arg P_n(r_me^{2\pi
  i/3})\in[-\pi/18,0]$. Hence, by Lemma~\ref{lem:cos-cos}, we infer
\begin{equation} \label{eq:cos-cos2} 
\left|2\cos\left(\arg P_n(r_me^{2\pi
  i/3})-2m\pi/3\right)\right|
\ge2\min\{1/2,\cos(7\pi/18)\}>0.684.
\end{equation}
Furthermore, for $n\geq5300$ and $m\in[3n,\deg P_n]$ (so that
$r_m\in(r_0,1]$ by Lemma~\ref{leRadiusBound}), we use
Lemma~\ref{leIneqPeak} and Lemma~\ref{leIneqTail} to see that
\begin{equation} \label{eq:ept1}
  \epsilon_{0,P_n}(m,r_m)<0.407, \qquad \epsilon_{1,P_n}(r_m)<0.275.
\end{equation}
Comparing the bounds in \eqref{eq:cos-cos2} and \eqref{eq:ept1}, we see
that \eqref{eqFinalTarget} holds. Hence, by \eqref{eqTargetMain},
the First Borwein Conjecture is true for $n\ge5300$.

A full computer verification for $n\leq7000$ of the First
Borwein Conjecture has already been done,
cf.\ \cite[Sec.~13]{WANG2022108028}. (But see also
Remark~\ref{rem:comp} below.)
This finishes the proof. 
\end{proof}

Next we finish the proof of the Second Borwein Conjecture.

\begin{theorem}\label{thMain2}
The Second Borwein Conjecture,  Conjecture~\ref{cjBorwein2}, is true.
\end{theorem}
\begin{proof}
Again,
we prove this claim by verifying \eqref{eqFinalTarget} for
``large''~$n$ and a direct computation for ``small''~$n$.

By Lemma~\ref{leArg}, we have $\arg P_n^2(r_me^{2\pi
  i/3})\in[-\pi/9,0]$. Then, by Lemma~\ref{lem:cos-cos}, we may
conclude that
\begin{equation} \label{eq:cos-cos3} 
\left|2\cos\left(\arg P_n^2 (r_me^{2\pi
  i/3})-2m\pi/3\right)\right|
\ge
\left|2\cos\left(\pi/3-\arg P_n^2 (r_me^{2\pi
  i/3})\right)\right|.
\end{equation}
In particular, we have
\begin{equation} \label{eq:cos2} 
\left|2\cos\left(\arg P_n^2 (r_me^{2\pi
  i/3})-2m\pi/3\right)\right|\ge
2\cos(4\pi/9)>0.347.
\end{equation}
Furthermore, for $n\geq7000$ and $m\in[3n,(\deg P_n^2)/2]$ (so that $r_m\in(r_0,1]$ by Lemma~\ref{leRadiusBound}), we
use Lemma~\ref{leIneqPeak} and Lemma~\ref{leIneqTail} to see that
\begin{equation} \label{eq:ep2}
  \epsilon_{0,P_n^2}(m,r_m)<0.262, \qquad \epsilon_{1,P_n^2}(r_m)<0.079.
\end{equation}
Comparing the bounds in \eqref{eq:cos2} and \eqref{eq:ep2}, we see
that \eqref{eqFinalTarget} holds. Hence, by \eqref{eqTargetMain},
the Second Borwein Conjecture is true for $n\ge7000$.

\medskip
We now discuss the range $546<n<7000$.
Again referring to Lemma~\ref{leArg}, the argument $\arg P_n(r_me^{2\pi i/3})$ is
increasing as a function in~$r_m$. Consequently, the right-hand side
of \eqref{eq:cos-cos3} is also increasing in~$r_m$.
On the other hand,
we note that, according to Lemma~\ref{leIneqPeak} and Lemma~\ref{leIneqTail},
for $n>546$ the left-hand side of \eqref{eqFinalTarget} with $\dd=2$ has an upper bound that is decreasing
with respect to~$r_m$. Therefore, for
$n>546$, there exists $r^*=r^*(n)$ such that \eqref{eqFinalTarget}
with $\dd=2$
holds for $r\in[r^*,1]$. For each specific $n$, $r^*(n)$ can be
calculated by any
method for the numerical approximation of zeroes of a function with
sufficient accuracy. If we substitute $r^*(n)$ in
\eqref{eqStationaryPoint2} then we can compute a
corresponding~$m^*(n)$. 
Now \eqref{eqTargetMain} implies that, for $m\in[m^*(n),(\deg
  P_n^2)/2]$, the coefficient $[q^m]P_n^2(q)$ has the predicted sign. 

It turns out that $m^*(n)<25281$ in the region $546<n<7000$. Hence, it
remains to calculate the first 25281 coefficients of $P_n^2(q)$ for
$546<n<7000$, and {\it all\/} coefficients of $P_n^2(q)$ for $n\le546$.  
We programmed the corresponding calculations 
using~C with the GMP library \cite{GMP}. They took
less than one day on a personal laptop computer.
\end{proof}

\begin{rem} \label{rem:comp}
A line of argument similar to the one in the preceding proof makes it
possible to reduce the amount of calculation reported in the proof
of Theorem~\ref{thMain1} significantly. 
Namely, this line of argument shows that only a full calculation of the
coefficients of $P_n(q)$ for 
$1\leq n\leq 546$, and a calculation of the coefficients
$[q^m]P_n(q)$ for $m\in[0,34168]$ and $546<n<5300$ is needed. 
The corresponding calculations took about 4 hours on a personal laptop
computer, as opposed to the computations reported in \cite[Sec.~13]{WANG2022108028}
which took 2~days using a multiple-core cluster. 
\end{rem}

Finally, the theorem below says that ``two thirds'' of the Cubic Borwein
Conjecture, Conjecture~\ref{cjBorwein4}, are true. 

\begin{theorem}\label{thMain3}
The coefficient $[q^m]P_n^3(q)$ is positive if $3|m$, and is negative
if\/ $m\leq 3(\deg P_n)/2$ and $m\equiv1\pmod3$. 
\end{theorem}

\begin{rem} 
While it may seem at first sight that
the statement in Theorem~\ref{thMain3} is just ``one half'' 
of Conjecture~\ref{cjBorwein4}, it is indeed ``two thirds'' of
that conjecture. To understand this, we should recall that
$P_n(q)$ is palindromic, and therefore also $P_n^3(q)$.
Consequently, Theorem~\ref{thMain3} also implies that
the coefficient $[q^m]P_n^3(q)$ is negative
if $m\geq 3(\deg P_n)/2$ and $m\equiv2\pmod3$. 
\end{rem}

\begin{proof}[Proof of Theorem \ref{thMain3}]
The proof and calculations are completely analogous to the ones of
Theorems~\ref{thMain1} and~\ref{thMain2}, with the key difference
being that the constraint\break $m\equiv0,1\pmod{3}$ implies that,
again using Lemma~\ref{lem:cos-cos}, a lower
bound for\break $\left|2\cos\left(\arg P_n^\dd (r_me^{2\pi
  i/3})-2m\pi/3\right)\right|$ is actually $1$. We calculate for
$n\geq3150$ that
\begin{align*}
  \epsilon_{0,P_n^3}(m,r_m)<0.335, \qquad \epsilon_{1,P_n^3}(r_m)&<0.614,
\end{align*}
and perform a full calculation of the coefficients of $P_n^3(q)$ for
$1\leq n\leq 546$, as well as a calculation of the coefficients
$[q^m]P_n^3(q)$ for $m\in[0,8864]$ for $546<n<3150$. 
Since we have $\sup_{546<n<3150}m^*(n)<8864$, this suffices for the proof.
\end{proof}

\begin{rem} \label{rmZero}
The reason why we cannot prove Conjecture~\ref{cjBorwein4} for
$m\equiv2\pmod3$ with $m\leq 3(\deg P_n)/2$ 
is that the right-hand side of \eqref{eqFinalTarget}
can get arbitrarily close to $0$ since, by Lemma~\ref{leArg},
we can only conclude that $\arg P_n^3(r_me^{2\pi   i/3})\in[-\pi/6,0]$. 
We will elaborate on
this in Item~(1) of the next, and final, section.
\end{rem}

\section{Discussion and outlook}\label{seDiscuss}

In this paper, 
we proved the First and Second Borwein Conjecture, and --- partially
--- a Cubic Borwein Conjecture, by developing an asymptotic framework
that allowed us to verify these conjectures for ``large''~$n$, meaning
that in each case a specific $n_0$ of very modest size 
was given, and it was proved that
the corresponding conjecture held for $n\ge n_0$. Together with a
direct calculation for the remaining ``small''~$n$ using a computer,
the proofs could be completed. We are convinced that this framework
can be further enhanced and extended to a machinery that is capable
of establishing the positivity/negativity of coefficients in more
general products/quotients of $q$-shifted factorials. We discuss this
perspective in this section. 

We start our discussion by going back to the Cubic Borwein Conjecture,
Conjecture~\ref{cjBorwein4}, and work out what prevented us at this
stage to come up with a full proof (see Item~(1)). Indeed, that ``failure'' strongly
points out one direction where our method needs refinement.
Subsequently, we turn our attention to the Third Borwein Conjecture
and other ``Borwein-like'' sign pattern conjectures that one finds in the
literature, in particular a conjecture of Ismail, Kim and Stanton (see
Item~(2)). 
As we argue there, we have no doubt that our ideas that we
presented here will lead to substantial progress, if not full proof,
of these. Then we report on computer experiments that we performed
that led us to discover new Borwein-type conjectures for the moduli~4
and~7 and make other intriguing observations concerning sign patterns
in such polynomials (see Item~(3)). Bressoud's conjecture that was
mentioned in the introduction is a vast generalisation of the First
Borwein Conjecture. Although, from the outset, it does not seem that
our method has anything to say about that conjecture, we show that
Bressoud's alternating sum expression can be converted into a
double contour integral of a product of $q$-shifted factorials.
Therefore our ideas do apply. Whether progress can be made in this way remains to
be seen. We close this section by a discussion of the ``nature'' of
the Borwein Conjectures, whether they should be considered as
``combinatorial'' or as ``analytic''.

\medskip
(1) {\sc Which are the obstacles to complete the proof of the Cubic
  Borwein Conjecture, Conjecture~\ref{cjBorwein4}?}
It may have come somewhat unexpected that, with the machinery
developed here, we proved ``only'' ``two thirds'' of
Conjecture~\ref{cjBorwein4} and left non-positivity of the
coefficients of $q^{3m+2}$ in $P_n^3(q)$, $0\le m< (\deg P_n)/2$,
(and consequently also the non-positivity of the
coefficients of $q^{3m+1}$ in $P_n^3(q)$, $(\deg P_n)/2\le m\le \deg P_n$), open.

The main reason for this ``failure'', as mentioned in Remark~\ref{rmZero},
is that the right-hand side of \eqref{eqFinalTarget} can get
arbitrarily close to $0$. Indeed, by applying
Lemma~\ref{leAlternatingSum} to the function $x\mapsto f(r,x)$ defined
in the proof of Lemma~\ref{leArg}, we are able to obtain a much more
accurate estimate for the argument of $P_n(re^{2\pi i/3})$, namely
\begin{equation}\label{eqArgPn}
\arg P_n(re^{2\pi
  i/3})=-\frac{\pi}{18}+\frac13\arctan\frac{\sqrt{3}r^{3n}}{2+r^{3n}}+O(n^{-1}r^{3n}). 
\end{equation}
This implies that, for $\dd=3$ and $m\equiv2\pmod{3}$, the right-hand
side of \eqref{eqFinalTarget} is equal to
\begin{align*}
2\cos\left(3\arg P_n (re^{2\pi i/3})+2\pi/3\right)&=2\cos\left(\frac\pi2+\arctan\frac{\sqrt{3}r^{3n}}{2+r^{3n}}\right)+O(n^{-1}r^{3n})\\
&=\frac{\sqrt3\,r^{3n}}{\sqrt{1+r^{3n}+r^{6n}}}+O(n^{-1}r^{3n}),
\end{align*}
which, for values of $r=\exp(-\Theta(n^{-1/2}))$ near the cutoff $r_0$, is
of the order\break $\exp(-\Theta(n^{1/2}))$. In comparison, the bound for
the peak error term
$\epsilon_{0,P_n^\dd}(m,r_m)$ that results from
Lemma~\ref{leIneqPeak} is of the order $O(n^{-1/2})$ for
$r=\exp(-\Theta(n^{-1/2}))$. Therefore, in this regime for $r$, the
inequality \eqref{eqFinalTarget} does not hold in the $n\to\infty$
limit. 
Roughly speaking, this issue is caused by the addition of the two peak
contributions in~\eqref{eqIntRep}, which are complex conjugates of
each other (cf.\ Part~C in Section~\ref{seOutline}), but
in this case happen to have real part very close to zero (approaching
zero as $n\to\infty$), and therefore largely cancel each other.
What this observation implies is that
the peak contribution --- and thus the coefficient of $P_n^3(q)$ itself --- 
is ``unusually'' small in this case.
This is also mirrored by the earlier observed fact
(cf.\ the end of Section~\ref{seInfinite}) that the coefficient
$[q^m]P_\infty^3(q)$ is always zero if $m\equiv2\pmod{3}$. 
So, again roughly speaking, what is at stake here is to determine the
``next'' term(s) in the asymptotic expansion of the peak part of the integral
in order to allow for a more precise estimate of the error made
by approximating the peak part by a Gau\ss ian integral.

\medskip
(2) {\sc What about other ``Borwein-like'' conjectures?}
As we said in the introduction, {\it three} Borwein Conjectures were
reported in \cite{MR1395410}: Conjectures~\ref{cjBorwein}
and~\ref{cjBorwein2}, and the {\it Third Borwein Conjecture}, an
analogue of the First Borwein Conjecture (Conjecture~\ref{cjBorwein})
in which the modulus~3 is replaced by~5.

\begin{Conjecture}[\sc P. Borwein]\label{cjBorwein3}
For all positive integers $n$, the sign pattern of the coefficients
in the expansion of the polynomial $S_n(q)$ defined by
\begin{equation*}
S_n(q):=\frac {(q;q)_{5n}} {(q^5;q^5)_n}
\end{equation*}
is $+----+----+----\cdots$,
with the same convention concerning zero coefficients as in Conjectures~\ref{cjBorwein}
and \ref{cjBorwein2}.
\end{Conjecture}

It should be clear that the approach that we presented in this paper
can also be applied to this conjecture, in adapted form. 
In order to show that the ``first few'' and the ``last few''
coefficients of $S_n(q)$ obey the predicted sign pattern (necessary for
completing the analogue of Part~A in Section~\ref{seOutline}), we would
quote \cite[Eq.~(2.5)]{MR1395410} with $p=5$,
\begin{equation} \label{eq:And5} 
\frac {(q;q)_\infty} {(q^5;q^5)_\infty}
=\sum_{k=-2}^2 (-1)^kq^{k(3k+1)/2}
\frac
    {(q^{75};q^{75})_\infty\,(q^{40+15k};q^{75})_\infty\,(q^{35-15k};q^{75})_\infty}
    {(q^5;q^5)_\infty},
\end{equation}
which Andrews derived by using Euler's pentagonal number theorem and
Jacobi's triple product identity. For the contour integral
representation of $[q^m]S_n(q)$ (the analogue of Part~B in
Section~\ref{seOutline}), we would again
choose a circle of radius~$r$, $0<r\le1$. 
Here, we would have to deal with {\it four} approximate saddle points (analogue
of Part~C in 
Section~\ref{seOutline}): $re^{\pm 2\pi i/5}$ and $re^{\pm4\pi i/5}$,
with $r$ being a solution of the obvious approximate saddle point equation
analogous to \eqref{eqStationaryPoint}. All these four approximate
dominant saddle points would contribute peaks of the same asymptotic order to
the contour integral. Clearly, the estimations in
Sections~\ref{seEps0} and~\ref{seEps1} would have to be adapted
accordingly. We expect however that this approach can prove that
the coefficients $[q^{5m}]S_n(q)$, $[q^{5m+1}]S_n(q)$,
$[q^{5m+2}]S_n(q)$ have the predicted signs for $n\le m\le\frac {1}
{10}\deg (S_n(q))$. On the other hand, in the case of the
coefficients $[q^{5m+3}]S_n(q)$ and $[q^{5m+4}]S_n(q)$ we would face
the same difficulty as we do for the coefficients $[q^{3m+2}]P_n^3(q)$
as discussed above: from \eqref{eq:And5} we see that the coefficients
$[q^{5m+3}]S_\infty(q)$ and  $[q^{5m+4}]S_\infty(q)$ are all zero, and
this indicates that the corresponding coefficients in $S_n(q)$ are
relatively small, and therefore it will require much more accurate
estimations in order to show that these coefficients are negative.

\medskip
Ismail, Kim and Stanton \cite[Conj.~1 in Sec.~7]{MR1660081}
generalised the First Borwein Conjecture, Conjecture~\ref{cjBorwein}, in
a direction different from the earlier mentioned Bressoud Conjecture.

\begin{Conjecture}[\sc Ismail, Kim and Stanton]\label{cjIsKimSt}
Let $a$ and $K$ be relatively prime positive integers, $1\le a\le
K/2$, with $K$ being odd. Put
$$
\prod _{i=0} ^{n-1}(1-q^{a+iK})(1-q^{K-a+iK})=\sum_{m\ge0}b_mq^m.
$$
The sign of\/ $b_m$ is determined by $m$ modulo~$K$. More precisely,
if $m\equiv\break \pm(2l+1)a$~{\em(mod~$K$)} for some $l$ with $0\le l<
K/2$, then $b_m\le0$, otherwise $b_m\ge0$.
\end{Conjecture}

Our approach is certainly tailored for an attack on this conjecture.
As already pointed out in \cite{MR1660081}, the ``infinite'' case
(the analogue of Part~A) follows easily from the Jacobi triple product
identity. For the contour integral representation of the coefficients
we would again choose a circle, with approximate saddle points of the
modulus of the integrand at $re^{\pm 2\pi ib/K}$, where $2ab\equiv 1$~mod~$K$.
The fact that this conjecture contains additional parameters ---
namely~$K$ and~$a$ --- may be an obstacle for a full proof, in particular in the checking
part (for small~$n$) of our approach. A proof of Conjecture~\ref{cjIsKimSt} 
for sufficiently large~$n$ should however definitely be feasible.

\medskip
It is reasonable to believe that, with the approach in this paper,
the sign-pattern problem for a general polynomial of the form
  \begin{equation}\label{eqGeneralPoly}
    Q_n(q):=\prod_{j}(q^{\alpha_j};q^K)_n\,(q^{K-\alpha_j};q^K)_n
  \end{equation}
can be reduced to an "infinite case" analogous to what is proved in
Section~\ref{seInfinite}, and an inequality analogous to \eqref{eqFinalTarget},
where the error terms tend to zero uniformly as $n\to\infty$. Naturally,
the sign pattern of the polynomial coefficients would be determined by
analogues of the right-hand side of~\eqref{eqFinalTarget}, which would
turn out to be
essentially a sum of the cosines of ``arguments'' over all dominant peaks.
Analogous to \eqref{eqArgPn}, the arguments of these peak values can
be well approximated as functions of the quantity~$r^{Kn}$.
Here, the~$r$ is the solution of an approximate saddle point
equation, which at the same time connects it to an index~$m$, and thus
to the coefficient of~$q^m$ in the polynomial~\eqref{eqGeneralPoly}.
Below we list a rough correspondence of the orders of magnitude
of the quantities~$r$ and~$m$,
which can in principle be obtained by arguments similar to those in
Section~\ref{seLocate}:

\vskip10pt\noindent
\begin{tabular}{r|ccc}
  \hline
  Coefficients & $r$ & $r^{Kn}$ & $m$ \\
  \hline
  near the cutoff & $\exp(-\Theta(n^{-1/2}))$ & $\exp(-\Theta(n^{1/2}))$ & $O(n)$ \\
  somewhere in the ``interior'' & $\exp(-\Theta(n^{-1}))$ & $\Theta(1)$ & $\Theta(n^2)$ \\
  the central coefficient & $1$ & $1$ & $\frac12(\deg Q_n)=\Theta(n^2)$ \\
  \hline
\end{tabular}
\vskip10pt

From the table above we can see that, as the index $m$ ranges from
$\Theta(n)$ --- where the coefficients of $Q_n(q)$ start to differ
from  $Q_\infty(q)$ --- to $\Theta(n^2)$ --- where we find the
central coefficient of $Q_n(q)$ ---, the quantity $r^{Kn}$ is expected
to take any values from $0$ to $1$. This allows us to predict the sign
patterns for polynomials or power series of the form
\eqref{eqGeneralPoly} by the following process: 

\medskip
{\it Step 1.} Identify the pair(s) of dominant peaks among $\varphi(K)/2$
    candidates located near primitive $K$-th roots of unity, where
    $\varphi(\,.\,)$ denotes Euler's totient function.

\smallskip
{\it Step 2.}
 For each pair of dominant peaks (say, located at arguments $\pm\theta$ where $0<\theta<\pi$), calculate the arguments of the function values at these places and approximate them by functions of $r^{Kn}$. Using Maclaurin summation techniques similar to Lemma~\ref{leAlternatingSum}, we claim that each factor $(q^{\alpha_j},q^{K-\alpha_j};q^K)_n$ in \eqref{eqGeneralPoly} contributes an amount of
      \begin{equation}\label{eqGeneralArg}
      -\frac{K-2\alpha_j}{K}\arctan\frac{(1-r^{Kn})\cot(\alpha_j\theta/2)}{1+r^{Kn}}+O(r^{Kn}n^{-1})
      \end{equation}
      to the argument of $Q_n(re^{i\theta})$.

\smallskip
{\it Step 3.}
 Therefore, the analogue of the right-hand side of \eqref{eqFinalTarget} would (approximately) be 
  \begin{equation}\label{eqGeneralTarget}
    \sum_{\ell }2\cos\left(-im\theta_\ell -\sum_{j}\frac{K-2\alpha_j}{K}\arctan\frac{(1-r^{Kn})\cot(\alpha_j\theta_\ell /2)}{1+r^{Kn}}\right),
  \end{equation}
  where the outer sum is over all pairs of arguments $\pm\theta_\ell $ of dominant peaks, and the inner sum is over all factors in \eqref{eqGeneralPoly}.
  By substituting different values for $r^{Kn}$ (remember that $r$
  depends on~$m$) and different residue
  classes of $m$ modulo $K$, we can read off the general behaviour of
  $[q^m]Q_n(q)$ from~\eqref{eqGeneralTarget}.

\medskip
(3) {\sc More conjectures.}
We have performed extensive computer calculations in order to see
whether, apart from the new Cubic Borwein Conjecture,
Conjecture~\ref{cjBorwein4}, there are more sign pattern phenomena in
Borwein-type polynomials that have not been discovered yet. 
Our most striking findings are the following two conjectures.
In the first of the two, we use the truth notation $\chi(\mathcal
A)=1$ if $\mathcal A$ is true and $\chi(\mathcal A)=0$ otherwise. 

\begin{Conjecture}[\sc A modulus $4$ ``Borwein Conjecture'']\label{cjmod4}
Let $n$ be a positive integer and $\dd\in\{1,2,3\}$. Furthermore,
consider the expansion of the polynomial
$$
\frac {(q;q)_{4n}^\dd} {(q^4;q^4)_n^\dd}
=\sum_{m=0}^Dc_m^{(\dd)}(n)q^m,
$$
which has degree $D=6\dd n^2$.
Then 
\begin{equation} \label{eq:mod4-3j} 
c_{4m}^{(\dd)}(n)\ge0\quad \text{and}\quad  c_{4m+2}^{(\dd)}(n)\le0,
\quad \text{for all~$m$ and $n$},
\end{equation}
while
\begin{equation} \label{eq:mod4-4j+1} 
c_{4m+1}^{(\dd)}(n)\le0,\quad \text{for }
\begin{cases} 
0\le m\le \frac {1} {8}(6\dd n^2-8),&\text{if $n$ is even,}\\
0\le m\le \frac {1} {8}(6\dd n^2-8+2\dd),&\text{if $n$ is odd,}
\end{cases}
\end{equation}
and
\begin{equation} \label{eq:mod4-4j+3} 
c_{4m+3}^{(\dd)}(n)\ge0,\quad \text{for }
\begin{cases} 
0\le m\le \frac {1} {8}(6\dd n^2-8),&\text{if $n$ is even,}\\
0\le m\le \frac {1} {8}(6\dd n^2-6\dd+8\chi(\dd=3)),&\text{if $n$ is odd,}
\end{cases}
\end{equation}
with the exception of two coefficients: for $\dd=1$ and $n=5$, 
we have $c_{71}^{(1)}(5)=-1$ and $c_{79}^{(1)}(5)=1$.
\end{Conjecture}

\begin{rem}
Roughly speaking, what the above conjecture says is that
all coefficients $c_{4m}^{(\dd)}(n)$ are non-negative, all coefficients
$c_{4m+2}^{(\dd)}(n)$ are non-positive, the ``first half'' of the
coefficients $c_{4m+1}^{(\dd)}(n)$ is non-positive, and the ``first half'' of the
coefficients $c_{4m+3}^{(\dd)}(n)$ is non-negative (with the mentioned
exceptions in the case where $n=5$). Since the polynomial 
$(q;q)_{4n}/(q^4;q^4)_n$ is palindromic for even~$n$ and
skew-palindromic for odd~$n$, we have
\begin{equation*} 
c_{m}^{(\dd)}(n)=(-1)^{\dd n}c_{6\dd n^2-m}^{(\dd)}(n).
\end{equation*}
Consequently, the statements \eqref{eq:mod4-4j+1} and \eqref{eq:mod4-4j+3} 
imply that the coefficients $c_{4m+1}^{(\dd)}(n)$ are non-negative for
$m$ outside the ranges given in~\eqref{eq:mod4-4j+1} (with two exceptions
for $n=5$), and similarly
the coefficients $c_{4m+3}^{(\dd)}(n)$ are non-positive for $m$
outside the ranges given in~\eqref{eq:mod4-4j+3}.
\end{rem}

\begin{Conjecture}[\sc A modulus $7$ ``Borwein Conjecture'']\label{cjmod7}
For positive integers~$n$,
consider the expansion of the polynomial
$$
\frac {(q;q)_{7n}} {(q^7;q^7)_n}
=\sum_{m=0}^{21n^2}d_m(n)q^m.
$$
Then 
\begin{equation} \label{eq:mod7-3j} 
d_{7m}(n)\ge0\quad \text{and}\quad
d_{7m+1}(n),d_{7m+3}(n),d_{7m+4}(n),d_{7m+6}(n)\le0,
\quad \text{for all~$m$ and $n$},
\end{equation}
while
\begin{equation} \label{eq:mod7-7j+2} 
d_{7m+5}(n)\begin{cases}
\ge0,\quad \text{for }m\le 3\al(n)n^2,\\
\le0,\quad \text{for }m> 3\al(n)n^2,
\end{cases}
\end{equation}
where $\al(n)$ seems to stabilise around $0.302$.
\end{Conjecture}

\begin{rem} \label{rem:7}
(1)
Since the polynomial $(q;q)_{7n}/(q^7;q^7)_n$ is palindromic,
the above conjecture makes also a prediction for the signs of the
coefficients $d_{7m+2}(n)$.

\medskip
(2)
The existence and approximate position of the sign change for the
coefficients of $q^m$ with $m\equiv2,5~(\text{mod } 7)$ predicted in~\eqref{eq:mod7-7j+2} 
can in fact be explained by the general procedure for approaching
proofs of sign patterns in the polynomial \eqref{eqGeneralPoly}, here specialised to
$K=7$ and $\al_j=j$ for $j=1,2,3$. 
As a matter of fact, the function $(q;q)_{7n}/(q^7;q^7)_n$ has three
pairs of dominant peaks (of the same order of
magnitude) located at $re^{\pm 2\pi i \ell/7}$ for $\ell=1,2,3$. We set
$\alpha_j=j$, for $j=1,2,3$, and $\theta_\ell=2\pi i \ell/7$, for
$\ell=1,2,3$, in~\eqref{eqGeneralTarget} to conclude that, for
$m\equiv5~(\text{mod } 7)$, the sum~\eqref{eqGeneralTarget} evaluates
to $2\sqrt7\cos(3\pi/7)$ for $r^{7n}=0$, and to $-1$ for $r^{7n}=1$.
This indicates a sign change somewhere in the middle. More precisely, in this
case we can pinpoint the zero of \eqref{eqGeneralTarget} as
$r^{7n}\approx0.6089$. For convenience, let us write
$s_0:=0.6089$. The analogue of the approximate saddle point equation
\eqref{eqStationaryPoint} for our situation here can be calculated as
\[
\frac13\underset{7\nmid k}{\sum_{{k=1}}^{7n}}k\frac{r^k-7r^{7k}+6r^{8k}}{(1-r^k)(1-r^{7k})}=2m.
\]
Therefore, for $r^{7n}=s_0$, making the substitution $k\mapsto 7nu$,
we get
\begin{align*}
\lim_{n\to\infty}\frac{m}{21n^2}
&=\lim_{n\to\infty}\frac{7}{18(7n)^2}
\left(
{\sum_{{k=1}}^{7n}}
k\frac{r^k-7r^{7k}+6r^{8k}}{(1-r^k)(1-r^{7k})}
-
{\sum_{{k=1}}^{n}}
7k\frac{r^{7k}-7r^{49k}+6r^{56k}}{(1-r^{7k})(1-r^{49k})}
\right)\\
&=\frac {7} {18}\times\frac {6} {7}
\int^1_0u\frac{s_0^u-7s_0^{7u}+6s_0^{8u}}{(1-s_0^{u})(1-s_0^{7u})}du\approx0.30214,
\end{align*}
which explains the occurrence of the constant $0.302$ in Conjecture~\ref{cjmod7}.
\end{rem}

Many similar conjectures could be proposed. For example, it seems that
the coefficient of $q^{6m}$ in $(q;q)_{6n}/(q^6;q^6)_n$ is
non-negative for all~$m$,
the coefficient of $q^{6m+3}$ in $(q;q)_{6n}/(q^6;q^6)_n$ is non-positive
for all~$m$,
while, for large enough~$n$, the other sequences of coefficients in
congruence classes modulo~6 of the exponents of~$q$ seem to satisfy sign
patterns similar to the one in \eqref{eq:mod7-7j+2}.
Similarly, for $\dd\in\{2,3\}$, it seems that
the coefficient of $q^{5m}$ in $(q;q)_{5n}^\dd/(q^5;q^5)_n^\dd$ is
non-negative for all~$m$,
while, for large enough~$n$, the other sequences of coefficients in
congruence classes modulo~5 of the exponents of~$q$ seem to also satisfy sign
patterns similar to the one in \eqref{eq:mod7-7j+2}.

\medskip
(4) {\sc The Bressoud Conjecture.}
Inspired by sum representations of the decomposition polynomials
$A_n(q),B_n(q),C_n(q)$ defined in \eqref{eq:ABC} which Andrews found
by the use of the $q$-binomial theorem (cf.\ \cite[Eqs.~(3.4)--(3.6)]{MR1395410}),
Bressoud \cite[Conj.~6]{MR1392489} came up with the following
far-reaching generalisation of the First Borwein Conjecture.
For the statement of Bressoud's conjecture we need to introduce the
usual $q$-binomial coefficients, defined by
$$
\begin{bmatrix} A\\B\end{bmatrix}_q:=\begin{cases} 
\displaystyle 
\frac {(q;q)_A} {(q;q)_B\,(q;q)_{A-B}},&\text{for }0\le B\le A,\\
0,&\text{otherwise.}
\end{cases}
$$

\begin{Conjecture}[\sc Bressoud]
Suppose that $M,N\in\Z^+$, $\alpha$ and $\beta$ are positive rational
numbers, and $K$ is a positive integer such that $\alpha K$ and $\beta
K$ are integers. If $1\leq\alpha+\beta\leq2K+1$ {\em(}with strict
inequalities if $K=2${\em)} and $\beta-K\leq n-M\leq K-\alpha$, then
the polynomial 
\begin{equation} \label{eq:Bressoud} 
\sum_{j=-\infty}^\infty(-1)^jq^{j(K(\alpha+\beta)j+K(\alpha-\beta))/2}
\begin{bmatrix} {M+N}\\{M+Kj}\end{bmatrix}_q
\end{equation}
has non-negative coefficients.
\end{Conjecture}

Conjecture~\ref{cjBorwein} turns out to be a special case of this
conjecture for the choices $\alpha=5/3$, $\beta=4/3$ and $K=3$.

To this day, Bressoud's conjecture has only been proved when
$\alpha,\beta\in\Z$ (corresponding to a result of Andrews et
al.~\cite{MR930170} on partitions with restricted hook differences), 
and some sporadic parametric infinite families (see
\cite{MR4119403,MR2140441,MR1874535,MR2009544}). 

If one tries a direct attack on proving non-negativity of the
coefficients of the polynomial~\eqref{eq:Bressoud} using contour
integral methods (in the style of \cite{WANG2022108028}, where
however different sum representations of $A_n(q),B_n(q),C_n(q)$ were
used as starting point), then one would discover that a large amount
of cancellation is going on in \eqref{eq:Bressoud} which is impossible
to control.

Instead, we could apply the $q$-binomial theorem \cite[Ex.~1.2({\it
  vi})]{GaRaAA} to express the $q$-binomial coefficient as
$$
\begin{bmatrix} A\\B\end{bmatrix}_q=q^{-\binom B2}[z^B](-z;q)_A.
$$
This leads to
\begin{align}
\notag
\sum_{j=-\infty}^\infty&(-1)^jq^{j(K(\alpha+\beta)j+K(\alpha-\beta))/2}
\begin{bmatrix} {M+N}\\{M+Kj}\end{bmatrix}_q\\
\notag
&=
\sum_{j=-\infty}^\infty
[z^{M+Kj}]
(-1)^jq^{\frac {1} {2}j(K(\alpha+\beta)j+K(\alpha-\beta))-\binom {M+Kj}2}(-z;q)_{M+N}\\
&=[z^M]q^{-\binom {M}2}(-z;q)_{M+N}
\sum_{j=-\infty}^\infty
(-1)^jq^{\frac {1} {2}(j^2K(\alpha+\beta-K)+jK(\alpha-\beta+1-M))}z^{-Kj}.
\label{eq:Brint}
\end{align}
If we assume that $|q|<1$ and $\alpha+\beta> K$, then we may now apply the
Jacobi triple product identity (cf.\ \cite[Eq.~(1.6.1)]{GaRaAA}),
\begin{equation} \label{eq:JTP} 
\sum_{j=-\infty}^\infty (-1)^j q^{\binom j2}u^j
=
(q,u,q/u;q)_\infty,
\end{equation}
where $(\alpha_1,\alpha_2,\dots,\alpha_s;q)_\infty$ is short for the
product $\prod _{i=1} ^{s}(\alpha_i;q)_\infty$. As a consequence, we obtain
\begin{multline*}
  \sum_{j=-\infty}^\infty(-1)^jq^{j(K(\alpha+\beta)j+K(\alpha-\beta))/2}
\begin{bmatrix} {M+N}\\{M+Kj}\end{bmatrix}_q
=[z^M]q^{-\binom {M}2}(-z;q)_{M+N}\\
\cdot
(q^{K(\alpha+\beta-K)},z^{-K}q^{K(2\alpha-K+1-M)/2},
z^Kq^{K(2\beta-K-1+M)/2};
q^{K(\alpha+\beta-K)})_\infty.
\end{multline*}
The coefficient of $q^m$ of the right-hand side can be represented
as a double contour integral over~$z$ and~$q$ of a product of (finite and infinite)
shifted $q$-factorials and is therefore --- at
least in principle --- amenable to the ideas that we developed in this paper.

If $\alpha+\beta< K$, then we would assume $|q|>1$ and try an
analogous approach. On the other hand, if $\alpha+\beta=K$, then the sum in
\eqref{eq:Brint} can be evaluated by summing a geometric series.\footnote{The 
reader should keep in mind that, for fixed~$M$ and~$N$, the sum over~$j$
is a finite sum.} Hence, again,
we obtain an expression that can be converted into a double contour integral that 
is amenable to the ideas developed in this paper.

\medskip
(5) {\sc Are the Borwein Conjectures combinatorial or analytic in nature?}
This question is somewhat on the provocative side.
It seems that it has been commonly believed that the Borwein Conjecture(s) is (are) 
combinatorial in nature, in the sense that the most promising approaches for
a proof are combinatorial, may it be by an injective argument, or by $q$-series 
manipulations, or by a combination of the two. However, we believe that 
by now considerable evidence has accumulated for the feeling 
that this might have been a misconception. On the superficial
level, one must simply admit that, despite considerable effort, 
until now ``combinatorial'' attacks have not led to any progress on  
the Borwein Conjectures (but undeniably to further intriguing
discoveries). By contrast, the first proof of the First Borwein
Conjecture in \cite{WANG2022108028} has been accomplished using
analytic methods, as well as the proof in this paper. More
substantially, several of the more recently discovered related or
similar results and conjectures, such as Conjecture~\ref{cjmod7}
(cf.\ in particular Remark~\ref{rem:7}(2)), the many conjectures by 
Bhatnagar and Schlosser in~\cite{MR4039550}, or Kane's
result~\cite{MR2052400} that we used in Section~\ref{seInfinite} seem
to indicate that {\it``typically''} such sign pattern results hold for
``large''~$n$, and in {\it some} cases --- such as in the case of the Borwein
Conjectures --- they {\it``accidentally''} also hold for ``small''~$n$.
This is not to say that we do not think that it is desirable to find a
combinatorial proof of, say, the First Borwein Conjecture. On the
contrary! However, one should be aware that such a proof would most
likely not have anything to say about the Second Borwein Conjecture
or the Cubic Borwein Conjecture,
while, by our analytic approach, we could do the First and Second
Borwein Conjecture (and large parts of the Cubic Borwein Conjecture) 
--- so-to-speak --- in one stroke.
Obviously, the last word in this matter has not yet been spoken.

\section*{Appendix: auxiliary inequalities}

\setcounter{equation}{0}
\setcounter{theorem}{0}
\global\def\thesection{\mbox{A}}
Here we collect several auxiliary inequalities of very technical
nature that we need in the
main text. We put them here so
as to not disturb the flow of arguments in the main text. 

\subsection{Bounds for certain rational functions in $s$ and $\log s$}

In the lemma below, we collect various bounds for the auxiliary
functions $u_j(z)$ and $v_j(z)$ from Section~\ref{seConvention}.
They are used ubiquitously in Sections~\ref{seLocate}, \ref{seEps0}, and~\ref{seEps1}.

\begin{lemma}\label{leUVBounds}
Suppose that $u_j(z)$ and $v_j(z)$, $j\in\Z^+$, are as defined
in~\eqref{eqDerivFuncs-u} and~\eqref{eqDerivFuncs-v}. 
Furthermore, for $\rho\in\R^+$, let the region $S_\rho$ be defined as in \eqref{eqRegionS}.

\medskip
{\em(1)}
For $s\in(0,1]$, we have the following inequalities:
\begin{align}
\label{eq:Ungl1}
\frac{u_1(s)}{s}\leq\frac2{\sqrt3}, \quad \frac23\leq\frac{u_2(s)}{s}&<\frac65,  \\
\label{eq:Ungl2a}
\frac{1-s^3}{(-\log s)(1+s)}&\le\frac {3} {2},\\
\label{eq:Ungl2}
\frac{s^{3-1/400}(-\log s)}{1-s^9}&<0.134,\\
\label{eq:Ungl3a}
\frac{(1-s^3)^2}{(-\log s)^2(1+2s+2s^3+s^4)}&\le\frac {3} {2},\\
\label{eq:Ungl3}
\frac{s^{3-1/400}(1-s^6)(-\log s)^2}{(1-s^9)(1-s^{3/2})(1+s^3+s^6)}&<0.084,\\
\label{eq:Ungl4}
\left|2(\log s)v_2(s)+(\log s)^2v_3(s)\right|&<\frac13,\\
\label{eq:Ungl5}
\left|4v_4(s)+(\log s) v_5(s)\right|&<\frac98,\\
\label{eq:Ungl6}
\left|2v_2(s)+2(\log s)v_3(s)+(\log s)^2v_4(s)\right|&<0.21,\\
\label{eq:Ungl7}
\left|12v_4(s)+8(\log s)v_5(s)+(\log s)^2v_6(s)\right|&<3.7.
\end{align}

\medskip
{\em(2)} We have upper bounds for $u_j(z)/z$ and $v_j(z)/z$ as given in the following table:

\bigskip
\begin{tabular}{|c|c|cccccc|}
  \hline
  {} & $j=$ & $3$ & $4$ & $5$ & $6$ & $7$ & $8$\\
  \hline
  \multirow{1}{*}{$z\in S_{5/27}$} & $|u_{j}(z)/z|<$ & $1.3$ & $1.409$ & {} & {} & {} & {}\\
  \hline
  \multirow{2}{*}{$z\in S_{10/27}$} & $|u_{j}(z)/z|<$ & $1.44$ & $1.721$ & {} & {} & {} & {}\\
  {} & $|v_{j}(z)/z|<$ & $1.01$ & $1.02$ & $2.09$ & $5.46$ & $19.1$ & $73$\\
  \hline
\end{tabular}
\end{lemma}

\begin{proof}
The inequalities \eqref{eq:Ungl1} are inequalities for rational
functions and therefore are straightforward to prove using standard
methods from classical analysis (or by the use of CAD; see
Footnote~\ref{foot:1}). For the inequalities
\eqref{eq:Ungl2a}--\eqref{eq:Ungl7}, we apply a numerical approach
(analogous to the one in the proof of Lemma~\ref{leIneqBeta} below).
Let $\LHS(s)$ denote the left-hand side of one such inequality.
We choose $M=10^6$ equally spaced points in the interval $[0,1]$. Then
we have
\[
\sup_{s\in[0,1]}\LHS(s)\leq
\sup_{0\le m\le M}\LHS\left(\frac{m}{M}\right)+\frac{1}{M}
\sup_{s\in[0,1]}\left|\frac{d\LHS}{d s}(s)\right|.
\]
The supremum of the derivative can easily be bounded since it is a
rational function in~$s$ and $\log s$ that has a finite value at $s=0$.

For the inequalities in Part (2) of the lemma, we also apply this
numerical approach. This is indeed feasible since, by the maximum
modulus principle, the maximum modulus of an analytic function on a
compact domain  (which, in
our case, are the sets $S_{5/27}$ respectively $S_{10/27}$)
is attained at the boundary of the domain.
%
%
%
\end{proof}

\subsection{Bounds for certain truncated perturbed Gau\ss ian integrals}

The central result of this subsection is Lemma~\ref{leIneqBeta} which
provides estimates for the constants that appear in
Lemma~\ref{leGaussian}, and which are used in Lemma~\ref{leIneqPeak}. 
A simple corollary of the lemma that is used in the proof of
Lemma~\ref{leGaussian} is stated separately as Corollary~\ref{cor:beta}.
The lemma below 
gives an estimate involving the lower incomplete gamma function 
that is needed in the proof of Lemma~\ref{leIneqBeta}.

Below, we will occasionally make use of the effective form of
Stirling's formula
\begin{equation} \label{eq:Stirling} 
\Ga(x)=\left(\frac {x} {e}\right)^x\frac {(2\pi)^{1/2}}
   {x^{1/2}}e^{\sigma(x)},    \qquad x>0,
\end{equation}
where
$$
0<\sigma(x)<\frac {1} {12x}.
$$
Here, the left inequality follows from \cite[Theoorem~1.6.3(i)]{AnARAA},
while the right inequality follows from \cite[Theorem~1.4.2
  with $m=1$]{AnARAA}.

\begin{lemma}\label{leIneqGamma}
Let $\gamma(s,a):=\int^a_0 e^{-x}x^{s-1}dx$ be the lower incomplete gamma function. Suppose that $c,d,\mu\in\R^+$ with $d> c$. Then we have
\[
\sup_{w\in\R^+}w^{-c}\gamma(d,\mu w)\leq \frac{\mu^c\Gamma(d-c+1)}{c\sqrt{2\pi(d-c)}}.
\]
\end{lemma}
\begin{proof}
  We note that the limit of $w^{-c}\gamma(d,\mu w)$ is $0$ for both
  $w\to0^+$ (here we use that $d>c$) 
and $w\to+\infty$. This implies that the maximum value of
  $w^{-c}\gamma(d,\mu w)$ with $w\in\R^+$ occurs at a point where $\frac
  d {dw} (w^{-c}\gamma(d,\mu w))=0$. It is straightforward to see that
  this latter equation is equivalent to
  \[
  \gamma(d,\mu w)=\frac{e^{-\mu w}(\mu w)^d}{c}.
  \]
Therefore, we have
  \[
  \sup_{w\in\R^+}w^{-c}\gamma(d,\mu w)\leq \sup_{w\in\R^+} \frac{e^{-\mu w}w^{d-c}\mu^d}{c}.
  \]
  Another differentiation shows that the supremum of the latter expression occurs at $w=(d-c)/\mu$, which gives a final bound of
  \[
  \sup_{w\in\R^+}w^{-c}\gamma(d,\mu w)\leq \frac{\mu^ce^{-(d-c)}(d-c)^{d-c}}{c}<\frac{\mu^c\Gamma(d-c+1)}{c\sqrt{2\pi(d-c)}},
\]
where, to get the last bound, 
we used the lower bound in~\eqref{eq:Stirling}.
This is exactly what we wanted to prove.
\end{proof}

\begin{lemma}\label{leIneqBeta}
There exist functions $\beta_i:(0,1)\to\R^+$ for $i=1,2,3,4$, defined by
\begin{align}
\beta_1(\mu)&:=\sup_{w>0}\frac{w^{3/2}}{\erf(\mu\sqrt{w})}\int^{\mu}_0e^{-w y^2}\left(\cosh(wy^3)-1\right)\,dy,\label{eqIneqBeta1}\\
\beta_2(\mu)&:=\sup_{w>0}\frac{w^{3/2}}{\erf(\mu\sqrt{w})}\int^{\mu}_0ye^{-w y^2}\sinh(wy^3)\,dy,\label{eqIneqBeta2}\\
\beta_3(\mu)&:=\sup_{w>0}\frac{w^{3/2}}{\erf(\mu\sqrt{w})}\int^{\mu}_0e^{-w y^2}\sinh(wy^4)\,dy,\label{eqIneqBeta3}\\
\beta_4(\mu)&:=\sup_{w>0}\frac{w^{2}}{\erf(\mu\sqrt{w})}\int^{\mu}_0ye^{-w y^2}\sinh(wy^4)\,dy. \label{eqIneqBeta4}
\end{align}
Moreover, we have the following estimates for particular values:
\begin{align*}
\beta_1(20/27)&<1.39, & \beta_2(20/27)&<1.14, & \beta_3(2/3)&<0.73, & \beta_4(2/3)&<1.15.
\end{align*}
\end{lemma}
\begin{proof}
We provide here only the proof concerning $\beta_1$. The
proofs for the other three suprema are completely analogous. 

\medskip
We must first show that the supremum in \eqref{eqIneqBeta1} is always finite.
Let
\[
b_1(\mu,w):=\frac{w^{3/2}}{\erf(\mu\sqrt{w})}\int^{\mu}_0e^{-w y^2}\left(\cosh(wy^3)-1\right)\,dy
\]
abbreviate the function of which we want to take the supremum.
We first note that the integrand in the above integral 
is bounded above by $\exp(-w y^2(1-y))$ and therefore also by~1. Hence,
\[
b_1(\mu,w)\leq \frac{\mu w^{3/2}}{\erf(\mu\sqrt{w})}.
\]

On the other hand, we perform a Taylor expansion of $\cosh(wx^3)-1$, and define
\[
u_1(k,\mu,w):=\frac{w^{3/2}}{(2k)!}\int^{\mu}_0e^{-w y^2}w^{2k}y^{6k}\,dy=\frac{\gamma(3k+1/2,\mu^2w)}{2(2k)!\,w^{k-1}},
\]
so that
\begin{align*}
b_1(\mu,w)&=\frac{1}{\erf(\mu\sqrt{w})}\sum_{k=1}^{\infty}u_1(k,\mu,w).
\end{align*}
Now, Lemma~\ref{leIneqGamma} implies that
\[
u_1(k,\mu,w)<\frac{\mu^{2k-2}\Gamma(2k+5/2)}{2(2k)!\,(k-1)\sqrt{(4k+3)\pi}}<\frac{(k+1)}{(k-1)\sqrt{2\pi}}\mu^{2k-2},
\]
where we used \eqref{eq:Stirling} to obtain the last inequality.
On the other hand, we trivially have
\[
u_1(k,\mu,w)<\frac{\Gamma(3k+1/2)}{2(2k)!w^{k-1}}.
\]
Both bounds combined, we find
\begin{align*}
b_1(\mu,w)\leq \frac{1}{\erf(\mu\sqrt{w})}\min\left(\mu w^{3/2},\frac{\Gamma(7/2)}{4}+\frac{\Gamma(13/2)}{48w}+\frac{1}{\sqrt{2\pi}}\sum_{k=3}^{\infty}\frac{k+1}{k-1}\mu^{2k-2}\right).
\end{align*}
This confirms the finiteness of the supremum in \eqref{eqIneqBeta1}
and therefore the existence of the function $\beta_1$. 

\medskip
In order to determine the particular value $\beta_1(20/27)$ (at least
approximately), we first dispose of large $w$ by providing an upper bound for $b_1(20/27,w)$ for $w>w_0:=80$. Indeed, in this regime we have $\mu\sqrt{w}>6$, and therefore
\begin{align*}
b_1(20/27,w)<\frac{1}{\erf 6}\left(\frac{\Gamma(7/2)}{4}+\frac{\Gamma(13/2)}{48w_0}+\frac{1}{\sqrt{2\pi}}\sum_{k=3}^{\infty}\frac{k+1}{k-1}(20/27)^{2k-2}\right)<1.37.
\end{align*}

We then determine the supremum of $b_1(20/27,w)$ over the interval
$[0,w_0]$ by a routine calculation. Namely, to begin with,
we provide a crude upper bound for the derivative $\frac{\partial
  b_1}{\partial w}(\mu,w)$ in this interval.  
To this end, we argue that the inequality $\erf(x)>x/(1+x)$ implies that
\[
\frac{w^{3/2}}{\erf(\mu\sqrt{w})}<\frac{w(1+\mu\sqrt{w})}{\mu}
\]
and
\begin{align*}
\left|\frac{\partial}{\partial
  w}\frac{w^{3/2}}{\erf(\mu\sqrt{w})}\right|
&=
\left|\frac{3\sqrt w}{2\erf(\mu\sqrt{w})}
-\frac{\mu we^{-\mu^2 w}}{\sqrt\pi\erf^2(\mu\sqrt{w})}\right|\\
&\le
\frac{3\sqrt w}{2\erf(\mu\sqrt{w})}
+\frac{\mu we^{-\mu^2 w}}{\sqrt\pi\erf^2(\mu\sqrt{w})}\\
&<
\frac{3(1+\mu\sqrt{w})}{2\mu}
+\frac{(1+\mu\sqrt{w})^2}{\mu
\sqrt\pi}
<\frac{4(1+\mu\sqrt{w})^2}{\mu\sqrt\pi}.
\end{align*}
On the other hand, for all $y\in[0,\mu]$ we have
$$
e^{-w y^2}\left(\cosh(wy^3)-1\right)<e^{-w y^2+w y^3}\leq1,
$$
and
\begin{align*}
\left|\frac{\partial}{\partial w}e^{-w
  y^2}\left(\cosh(wy^3)-1\right)\right|&=\left|y^2e^{-w y^2}(y
\sinh(wy^3)-\cosh(w y^3)+1)\right|\\
&<y^2e^{-w y^2+w y^3}\leq\mu^2.
\end{align*}
Combining these inequalities, we obtain
\begin{align*}
\left|\frac{\partial b_1}{\partial w}(\mu,w)\right|
&\le\left|
\frac{w^{3/2}}{\erf(\mu\sqrt{w})}\int^{\mu}_0
\frac {\partial } {\partial w}\left(
e^{-w  y^2}\left(\cosh(wy^3)-1\right)\right)\,dy
\right|\\
&\kern2cm
+\left|
\left(\frac {\partial } {\partial w}
\frac{w^{3/2}}{\erf(\mu\sqrt{w})}\right)
\int^{\mu}_0e^{-w
  y^2}\left(\cosh(wy^3)-1\right)\,dy\right|\\
&\leq (\mu^2)w(1+\mu\sqrt{w})+\frac{4(1+\mu\sqrt{w})^2}{\sqrt\pi}<\frac{6}{\sqrt\pi}w_0(1+\mu\sqrt{w_0})^2.
\end{align*}

With this upper bound proven, we choose $M=10^6$ uniformly distributed points in the interval $[0,w_0]$, and argue that
\[
\sup_{w\in[0,w_0]}b_1(20/27,w)\leq
\sup_{0\le m\le M}b_1\left(20/27,\frac{m}{M}w_0\right)+\frac{w_0}{M}
\sup_{w\in[0,w_0]}\left|\frac{\partial b_1}{\partial w}(20/27,w)\right|.
\]
The result of this calculation turns out to be $1.3860<1.39$ (accurate to the last significant digit given), which finishes the proof.
\end{proof}

\begin{cor} \label{cor:beta}
For $u,v\in\R^+$ and $x_0\in[0,u/v]$, we have
\begin{align}
\int^{x_0}_0e^{-u x^2}\left(\cosh(v x^3)-1\right)\,dx\leq \beta_1\left(x_0\frac{v}{u}\right)\frac{v^2}{u^{7/2}}\erf(x_0 \sqrt{u}),\label{coIneqBeta1}\\
\int^{x_0}_0xe^{-u x^2}\sinh(v x^3)\,dx\leq \beta_2\left(x_0\frac{v}{u}\right)\frac{v}{u^{5/2}}\erf(x_0 \sqrt{u}),\label{coIneqBeta2}\\
\int^{x_0}_0e^{-u x^2}\sinh(v x^4)\,dx\leq \beta_3\left(x_0\sqrt\frac{v}{u}\right)\frac{v}{u^{5/2}}\erf(x_0 \sqrt{u}),\label{coIneqBeta3}\\
\int^{x_0}_0xe^{-u x^2}\sinh(v x^4)\,dx\leq \beta_4\left(x_0\sqrt\frac{v}{u}\right)\frac{v}{u^3}\erf(x_0 \sqrt{u}).\label{coIneqBeta4}
\end{align}
\end{cor}
\begin{proof}
This follows immediately from Lemma~\ref{leIneqBeta} by, on the one
hand, performing the substitutions $y\to (v/u)x$ and $w\to u^3/v^2$ in
\eqref{eqIneqBeta1} and \eqref{eqIneqBeta2}, and performing the substitutions $y\to (\sqrt{v/u})x$ and $w\to u^2/v$ in \eqref{eqIneqBeta3} and \eqref{eqIneqBeta4}.
\end{proof}

\subsection{A Maclaurin summation estimate}

The following upper bound for an alternating sum is crucial in the
proof of Lemma~\ref{leDerivBounds}, see~\eqref{eqVjPrelimUpperBound}.

\begin{lemma}\label{leAlternatingSum}
For $n\in\Z^+$ and $f\in C^4[0,3n]$, we have
\begin{multline*}
\left|\sum_{k=1}^n \left(f(3k-2)-f(3k-1)\right)\right|\\
\leq \frac13\left|f(3n)-f(0)\right|+\frac23\left|f''(3n)-f''(0)\right|+\frac{11}{96}\sup_{x\in[0,3n]}|f^{(4)}(x)|.
\end{multline*}
\end{lemma}
\begin{proof}
We use the offset Maclaurin summation formula (see, for example, \cite[Theorem~D.2.4]{MR1994507}) to see that
\begin{align*}
\sum_{k=1}^n \left(f(3k-2)-f(3k-1)\right)
&=\sum_{k=1}^{4}\frac{3^{k-1}(B_k(2/3)-B_k(1/3))}{k!}\left(f^{(k-1)}(3n)-f^{(k-1)}(0)\right)\\
&\kern1cm
-\frac98\int^{3n}_{0}f^{(4)}(x)\left(\bar{B}_4\left(\frac{2-x}{3}\right)-\bar{B}_4\left(\frac{1-x}{3}\right)\right)dx\\
&=\frac13\left(f(3n)-f(0)\right)-\frac23\left(f''(3n)-f''(0)\right)\\
&\kern1cm
-\frac98\int^{3n}_{0}f^{(4)}(x)\left(\bar{B}_4\left(\frac{2-x}{3}\right)-\bar{B}_4\left(\frac{1-x}{3}\right)\right)dx,
\end{align*}
where the Bernoulli polynomials $B_k(u)$ are defined by
$$
\sum_{k\ge0}B_k(u)\frac {t^k} {k!}=\frac {te^{ut}} {e^t-1},
$$ 
and $\bar{B}_k(u)=B_k\left(\{u\}\right)$, with $\{u\}$ denoting the
fractional part of~$u$ as usual, is the $k$-th
periodic Bernoulli function.
The lemma follows from the fact that
\[
\int^{3n}_{0}\left|\bar{B}_4\left(\frac{2-x}{3}\right)-\bar{B}_4\left(\frac{1-x}{3}\right)\right|\,dx=\frac{11n}{108}.
\qedhere
\]
\end{proof}

\subsection{Estimates for sums and differences of exponentials}

Here we record two elementary estimates for the difference respectively the sum
of two exponentials that are used in the proof of
Lemma~\ref{leGaussian}. 

\begin{lemma}\label{leIneqComplex}
For $z,w\in\C$, we have the following inequalities:
\begin{align}
  \left|e^z-e^w\right|&\leq 2\sinh\max(|z|,|w|)+2\sinh\frac{|z+w|}{2}, \label{eqIneqComplexO}\\
  \left|e^z+e^w-2\right|&\leq 2\cosh\max(|z|,|w|)-2+2\sinh\frac{|z+w|}{2}. \label{eqIneqComplexE}
\end{align}
\end{lemma}
\begin{proof}
  Without loss of generality, we assume that $\Re (w-z)\leq0$. By the triangle inequality, we have
  \begin{align*}
    \left|e^z-e^w\right|&\leq\left|e^z-e^{-z}\right|+\left|e^{-z}-e^w\right|\\
    &\leq 2\sinh|z|+2\left|e^{(w-z)/2}\right|\sinh\frac{|z+w|}{2}\\
    &\leq 2\sinh\max(|z|,|w|)+2\sinh\frac{|z+w|}{2},
  \end{align*}
  and
  \begin{align*}
    \left|e^z+e^w-2\right|&\leq\left|e^z+e^{-z}-2\right|+\left|e^{-z}-e^w\right|\\
    &\leq 2\cosh|z|-2+2\left|e^{(w-z)/2}\right|\sinh\frac{|z+w|}{2}\\
    &\leq 2\cosh\max(|z|,|w|)-2+2\sinh\frac{|z+w|}{2}.
\qedhere
  \end{align*}
\end{proof}

\subsection{Inequalities for the sums $X_j(n,r)$}

The lemma below provides inequalities for various expressions
involving the sums $X_j(n,r)$  defined in \eqref{eqXmDef}.
These are used in the proof of Lemma~\ref{leDerivBounds} and
for the proof of several particular bounds presented in
Corollary~\ref{coIneqX} below. In their turn, the bounds of the
corollary are used in Lemmas~\ref{leIneqPeak} and~\ref{leIneqTail}.

\begin{lemma}\label{leIneqX}
For $n\in\Z^+$ and $r\in(0,1]$, we have the following inequalities
  concerning the quantities $X_j(n,r)$:
\begin{enumerate}
  \item 
  \begin{equation}\label{eqX1Lower}
    X_1(n,r)\geq\frac{r(1+2r+2r^3+r^4)(1-r^{3n})(1-r^{3n/2})}{(1-r^3)^2}.
  \end{equation}
  \item For $j=0,1,2,3$, we have
  \begin{equation}\label{eqXjBound1}
    \frac{X_{j+1}(n,r)}{X_0(n,r)X_j(n,r)}\leq \frac{X_{j+1}(\infty,r)}{X_0(\infty,r)X_j(\infty,r)}.
  \end{equation}
  \item For $j=0,1,2$, we have
  \begin{equation}\label{eqXjBound2}
    \frac{X_j(n,r)X_{j+2}(n,r)}{X_{j+1}^2(n,r)}\leq \frac{X_{j}(\infty,r)X_{j+2}(\infty,r)}{X_{j+1}^2(\infty,r)}.
  \end{equation}
\end{enumerate}

\end{lemma}
\begin{proof}
(1)
    Inequality \eqref{eqX1Lower} can be proved by observing that 
    \begin{align*}
      X_1(n,r)(1-r^3)^2&=r(1+2r+2r^3+r^4)(1-r^{3n})-3nr^{3n+1}(1+r)(1-r^3)\\
      &\geq r(1+2r+2r^3+r^4)(1-r^{3n})-r^{3n/2+1}(3r^{3/2}+3r^{5/2})(1-r^{3n})\\
      &\geq r(1+2r+2r^3+r^4)(1-r^{3n})-r^{3n/2+1}(1+2r+2r^3+r^4)(1-r^{3n})\\
      &=r(1+2r+2r^3+r^4)(1-r^{3n})(1-r^{3n/2}).
    \end{align*}

\medskip
(2) To prove \eqref{eqXjBound1} and \eqref{eqXjBound2}, we claim that the expressions 
    \[
    \frac{(1-r^3)^{j+2}}{(1+r)r^{3n+3}}\left(X_{j+1}(\infty,r)X_0(n,r)X_j(n,r)-X_{j+1}(n,r)X_0(\infty,r)X_j(\infty,r)\right)
    \]
    and
    \[
    \frac{(1-r^3)^{2j+4}}{3nr^{3n+4}}\left(X_{j}(\infty,r)X_{j+2}(\infty,r)X_{j+1}^2(n,r)-X_j(n,r)X_{j+2}(n,r)X_{j+1}^2(\infty,r)\right)
    \]
    are actually polynomials in $r$ with non-negative coefficients. This claim can be routinely verified by explicitly calculating each coefficient of these expressions as piecewise polynomial function. As an illustrative example, we have
    \[
    \frac{(1-r^3)^{4}}{3nr^{3n+4}}\left(X_{0}(\infty,r)X_{2}(\infty,r)X_1^2(n,r)-X_0(n,r)X_{2}(n,r)X_1^2(\infty,r)\right)=(1+r)\sum_{m=0}^{3n+2}a_mr^m,
    \]
    where the coefficients are given by $a_0=3n$, $a_1=15n-2$, $a_2=24n-4$, $a_3=30n-18$, $a_{3n}=27n-18$, $a_{3n+1}=3n-4$, $a_{3n+2}=3n-2$, and
    \[
    a_m=\begin{cases}
          3(2m-3)(3n-m)+9(m-2), & \mbox{if } m\equiv0\pmod3, \\
          3(3m-4)(3n-m)+9(m-2), & \mbox{if } m\equiv1\pmod3, \\
          3(3m-5)(3n-m)+18(m-2), & \mbox{if } m\equiv2\pmod3.
        \end{cases}
    \]
    for $4\leq m\leq 3n-1$.
\end{proof}

\begin{cor}\label{coIneqX}
For $n\geq1$ and $r\in(0,1]$, we have
  \begin{align}
  \frac{X_0^2(n,r)}{r X_1(n,r)}&\leq \frac43, \label{eqIneqR00_1}\\
  \frac{rX_2(n,r)}{X_0(n,r)X_1(n,r)}&\leq 3, \label{eqIneqR2_01}\\
  \frac{r^2X_2(n,r)}{X_0^3(n,r)}&\leq \frac92, \label{eqIneqR2_000}\\
  \frac{rX_3(n,r)}{X_0(n,r)X_2(n,r)}&\leq \frac92, \label{eqIneqR3_02}\\
  \frac{r^2X_4(n,r)}{X_0^2(n,r)X_2(n,r)}&\leq 27, \label{eqIneqR4_002}\\
  \frac{X_0(n,r)X_3(n,r)}{X_1(n,r)X_2(n,r)}&\leq 3, \label{eqIneqR03_12}\\
  \frac{X_0(n,r)X_3^2(n,r)}{X_2^3(n,r)}&\leq \frac92, \label{eqIneqR033_222}\\
  \frac{X_0(n,r)X_4(n,r)}{X_2^2(n,r)}&\leq 6, \label{eqIneqR04_22}
\end{align}
where the $X_j(n,r)$ are defined in \eqref{eqXmDef}.
\end{cor}
\begin{proof}
To prove \eqref{eqIneqR00_1}, we argue that
    \[
    4rX_1(n,r)-3X_0^2(n,r)=r^2(1-r)(1+3r)\geq0
    \]
    for $n=1$, and make use of \eqref{eqX1Lower} to see that
    \[
    \frac{X_0^2(n,r)}{r X_1(n,r)}\leq\frac{(1+r)^2(1+r^{3n/2})}{(1+2r+2r^3+r^4)}\leq\frac{(1+r)^2(1+r^3)}{(1+2r+2r^3+r^4)}\leq\frac43
    \]
    for $n\geq2$.

\medskip
For the other seven inequalities, we invoke \eqref{eqXjBound1} for \eqref{eqIneqR2_01}--\eqref{eqIneqR4_002} and \eqref{eqXjBound2} for \eqref{eqIneqR03_12}--\eqref{eqIneqR04_22} to see that the left-hand side of all six inequalities does not exceed the corresponding $n\to\infty$ limit. The six limits in question are simple rational functions in $r$ and can be routinely shown to be bounded above by the right-hand side; as an example, for \eqref{eqIneqR2_000} we have
        \[
        \frac{r^2X_2(n,r)}{X_0^3(n,r)}\leq\frac{r^2X_2(\infty,r)}{X_0^3(\infty,r)}=\frac{1+3r-3r^2+16r^3-3r^4+3r^5-r^6}{(1+r)^2},
        \]
        and 
        \[
        9(1+r)^2-2(1+3r-3r^2+16r^3-3r^4+3r^5-r^6)=(1-r)(7+19r+34r^2+2r^3+8r^4+2r^5)\geq0.
\qedhere        \]
\end{proof}

\subsection{Upper bounds for certain trigonometric sums}

This subsection contains two auxiliary results, of different flavour, which provide upper
bounds for the absolute value of certain trigonometric sums, the
second more special than the first. 
Lemma~\ref{leCosSum1} is used in the
proofs of Lemmas~\ref{leTailIneq2} and~\ref{leTailBound1},
while Lemma~\ref{leIneqCosSum1} is used in the proof of Lemma~\ref{leTailBound2}.
An auxiliary result that is needed in the proof of
Lemma~\ref{leIneqCosSum1} is stated separately in Lemma~\ref{leIneqChebyshev2}.

\begin{lemma}\label{leCosSum1}
Suppose that $0<r\leq 1$ and $\theta,\varphi\in\R$. For all positive monotonically increasing sequences $\{u_n\}_{n\geq0}$, and for all non-negative integers $a,b$ such that $a\leq b$, we have 
\[
\left|\sum_{k=a}^{b}u_kr^k\cos(k\theta+\varphi)\right|\leq \frac{1}{|1-re^{i\theta}|}\left((1-r)\sum_{k=a}^{b}u_k r^k+2r^{b+1}u_{b}\right).
\]
\end{lemma}
\begin{proof} We write $z=re^{i\theta}$, and note that the sum $S_{a,b}:=\sum_{k=a}^{b}r^k\cos(k\theta+\varphi)$ can be bounded above by
\begin{align*}
\left|\sum_{k=a}^{b}r^k\cos(k\theta+\varphi)\right|&\leq\left|\sum_{k=a}^{b}z^k\right|=\left|\frac{z^a-z^{b+1}}{1-z}\right|\leq \frac{r^a+r^{b+1}}{|1-z|}.
\end{align*}
Therefore, we can use Abel's lemma (summation by parts) to get
\begin{align*}
&\left|\sum_{k=a}^{b}u_kr^k\cos(k\theta+\varphi)\right|\leq u_a|S_{a,b}|+(u_{a+1}-u_a)|S_{a+1,b}|+\dots+(u_b-u_{b-1})|S_{b,b}|\\
&\kern20pt
\leq \frac{1}{|1-z|}\left(u_a(r^a+r^{b+1})+(u_{a+1}-u_a)(r^{a+1}+r^{b+1})+\dots+(u_b-u_{b-1})(r^{b}+r^{b+1})\right)\\
&\kern20pt
=\frac{1}{|1-z|}\left((1-r)\sum_{k=a}^{b}u_k r^k+2r^{b+1}u_b\right).
\qedhere
\end{align*}
\end{proof}

The following inequality improves Lemma~B.4 from~\cite{WANG2022108028}.

\begin{lemma}\label{leIneqCosSum1}
For $r\in(0,1)$, $n\in\Z^+$, and $\theta\in[-\pi,\pi]$, we have 
\begin{equation}\label{eqIneqCosSum1}
\sum_{k=1}^{n}r^{k-1}\cos k\theta\leq \frac{1-r^{n}}{1-r}\sqrt{\frac{1}{1+4\kappa\tan^2(\theta/2)}},
\end{equation}
where
\[
\kappa=\frac{(1+r)(1-r^n)(1-r^{n/6})}{(1-r)^2}.
\]
\end{lemma}
\begin{proof}
Writing $\cos(k\theta)=\frac {1} {2}\left(e^{ik\theta}+e^{-ik\theta}\right)$,
we see that the sum on the left-hand side can be evaluated 
as it is just the sum of two geometric series. After substitution of
the result, it turns out that
the claimed inequality is equivalent to
\begin{equation}\label{eqIneqCosSum2}
\frac{-r+\cos\theta+r^{n+1}\cos(n\theta)-r^n\cos(n\theta+\theta)}{1-2r\cos\theta+r^2}\leq \frac{1-r^n}{1-r}\sqrt{\frac{1}{1+4\kappa\tan^2(\theta/2)}}.
\end{equation}

Without loss of generality we assume that $\theta\geq0$. We prove
\eqref{eqIneqCosSum2} for all {\it real\/} $n\geq1$ and
$\theta\in[0,\pi]$. We divide the proof into two parts according to
whether $\theta$ is larger than $\frac{\pi}{n+1}$ or not.

\medskip
{\sc Part I.}
$\theta\leq\frac{\pi}{n+1}$. We construct Pad\'e approximants as
bounds for the various non-rational functions involved, with the goal
of reducing the proof of the inequality to the proof of an inequality
for a {\it rational function}. The reason is that inequalities for rational
functions are easier to handle. In particular, they
can be automatically proved by using {\it Cylindrical
Algebraic Decomposition} (CAD),\footnote{\label{foot:1}Cylindrical Algebraic
Decomposition (CAD) is an algorithm that, among others, is able to
{\it prove} that a given polynomial in several variables is
positive (non-negative), respectively provides a description of the subset of the
parameter space for which the polynomial is positive (non-negative).
It also allows one to verify the positivity (non-negativity) of
polynomials in several variables under (polynomial) constraints on the variables.
The reader is referred to the ``user guide'' \cite{MR2773014} and the
references therein. Implementations of CAD are available within any
standard computer algebra programme. The one that we used is the
command {\tt CylindricalDecomposition} within {\sl Mathematica}.} 
and this is what we are going to do in the end
for the most intricate ones.

We let $t=\tan^2(\theta/2)$ so that
$\cos\theta=\frac{1-t}{1+t}$. Using Lemma~\ref{leIneqChebyshev2}
below, Lemma~B.3 from~\cite{WANG2022108028}, and elementary manipulations, we obtain 
\begin{align*}
\cos(n\theta)=\frac {3 + (3 - 5 n^2) t - n^2 t^2} {(1 + t) (3 + n^2 t)} 
&\geq\frac{3-(5n^2-2)t}{3+(n^2+2)t}, \quad \text{for all } \theta\in\R,\\
\cos(n\theta)-\cos(n\theta+\theta)&\leq\frac{6(2n+1)t}{3+(2n^2+2n+3)t}, \quad \text{for all }\theta\in[0,\pi/n],\\
\frac{1}{\sqrt{1+4x}}&\geq\frac{1}{1+2x}, \quad \text{for all } x\geq0.
\end{align*}
With these inequalities in mind, it is sufficient to prove that
  \begin{multline}
  (1-r)\left(1-r^n\frac{3-(5n^2-2)t}{3+(n^2+2)t}\right)-1+\frac{1-t}{1+t}+r^n\frac{6(2n+1)t}{3+(2n^2+2n+3)t}\\
  \leq \frac{1-r^n}{1-r}\left(1-2r\frac{1-t}{1+t}+r^2\right)\frac{1}{1+2\kappa t}. \label{eqIneqCosSumRational}
  \end{multline}
  The difference between the two sides of \eqref{eqIneqCosSumRational} can be written as
\begin{equation} \label{eq:diff} 
  \frac{2t(9a_0+3a_1t+a_2t^2+\kappa (1-r)a_3 t^3)}{(1-r)^2(1+t)(3+(n^2+2)t)(3+(2n^2+2n+3)t)(1+2\kappa t)},
\end{equation}
  where
{\allowdisplaybreaks
  \begin{align*}
    a_0 &= 1+r-r^n\left(\left(1+n(1-r)\right)^2+r\right)-\kappa(1-r)^2(1-r^n), \\
    a_1&=2(n^2+n+3)a_0+3\kappa(1-r)(1+r-3r^n+r^{n+1}) \\
    &\qquad+(n^2-1)(1+r-2r^{n+1})+(n+1)^2(2n+1)r^n(1-r)\\
    &\qquad-\kappa(1-r)\left((n^2-1)(1-r)(1+5r^n)+12nr^n\right), \\
    a_2&=\sum_{j=0}^4a_{2j}n^j, \text{with}\\
    &a_{20}=3(1-r^n)(2+2r+\kappa(1-r)(3+7r)), \\
    &a_{21}=4(1+r-3r^n+r^{n+1})+2\kappa(1-r)(1+5r-25r^n-5r^{n+1})), \\
    &a_{22}=7+7r-12r^n+7r^{n+1}-9r^{n+2}+2\kappa(1-r)(1+8r-13r^n+10r^{n+1}), \\
    &a_{23}=2(1+r-6r^n+7r^{n+1}-3r^{n+2})-2\kappa(1-r)(1-r+11r^n-5r^{n+1}), \\
    &a_{24}=2(1+r-3r^n+4r^{n+1}-6r^{n+2})-2\kappa(1-r)^2(1+5r^n), \\
    a_3&=(1+r)(1-r^n)(n^2+2)(2n^2+2n+3)\\
    &\qquad-6n^2r^n(1-r)(2n^2+2n+3)+4nr^n(n-2)(n^2+2).
  \end{align*}}%
In the following, we are going to prove non-negativity results for these coefficients.

\smallskip
(1) $a_0\geq0$. We substitute the definition of $\kappa$ in \eqref{eq:diff}.
After some simplification, the inequality can be shown to be equivalent to
\begin{equation} \label{eq:diff2} 
  \frac{\left(1+n(1-r)\right)^2+r}{1+r}\leq\frac{1-(1-r^n)^2(1-r^{n/6})}{r^n}.
\end{equation}
  In order to prove this, we first use the classical inequalities $1-r\leq(-\log r)$ and $\frac{1-r}{1+r}\leq(-\log r)/2$ to conclude that
  \[
  \frac{\left(1+n(1-r)\right)^2+r}{1+r}\leq 1+(-\log r)n+(-\log r)^2\frac {n^2}2.
  \]
  Note that the right-hand side is exactly the Taylor polynomial of
  $$r^{-n}(1-(1-r^n)^2(1-r^{n/6}))$$ 
of order $2$ at $n=0$. So, in
  order to prove~\eqref{eq:diff2}, it suffices to show that its third derivative is non-negative. Indeed, this third derivative can be calculated as
  \begin{multline*}
  \left(\frac{d}{dn}\right)^3\frac{(1-(1-r^n)^2(1-r^{n/6}))}{r^n}\\
=\frac{(-\log r)^3r^{n/6}}{216}\left(125r^{-n}+2+216r^{5n/6}-343r^n\right)\geq0.
  \end{multline*}

\smallskip
(2) $a_1\geq0$. We claim that
  \begin{multline}
  (n^2-1)(1+r-2r^{n+1})+(n+1)^2(2n+1)r^n(1-r)\label{eqCoeffa2}\\
\geq\kappa(1-r)^2(n^2-1)(1+5r^n)+12n\kappa r^n(1-r).
  \end{multline}
  By substituting the definition of $\kappa$ and using the inequality $1+5r^n\leq(1-r^n)/(1-r^{n/6})$, we see that \eqref{eqCoeffa2} is implied by
  \[
  2n(n+1)(n+2)\geq\frac{1-r^n}{1-r}(1+r)\left(12n\frac{1-r^{n/6}}{1-r}-n^2+1\right).
  \]
  This can be proved by noting that $n\geq(1-r^n)/(1-r)$, and that
  \begin{align*}
    (1+r)\left(12n\frac{1-r^{n/6}}{1-r}-n^2+1\right)&\leq \begin{cases}
                                                            (1+r)\left(2n^2-n^2+1\right), & \mbox{if } n\geq6, \\
                                                            12n\frac{1-r^{n/6}}{1-\sqrt{r}}-n^2+1, & \mbox{if }n<6,
                                                          \end{cases}\\
    &\leq\begin{cases}
           2n^2+2, & \mbox{if } n\geq6, \\
           12n\max(1,n/3)-n^2+1, & \mbox{if }n<6,
         \end{cases}\\
    &<2(n+1)(n+2).
  \end{align*}

\smallskip
(3) $a_2\geq0$. We prove that $a_{20}$, $a_{22}$, $a_{24}$,
$(1-r^{1/6})a_{21}+(1-r^{n/6})a_{22}$ and\break $\sum_{j=0}^4
(1-r^{1/6})^{4-j}(1-r^{n/6})^ja_{2j}$ are non-negative. All these
expressions are rational functions in $r^{1/6}$ and $r^{n/6}$.
In order to get these expressions ready for application of CAD, we
replace each occurrence of $r^{n/6}$ by $X$, and each occurrence
of $r^{1/6}$ by $Y$, say. 
In this manner, we obtain rational functions
in $X$ and~$Y$. (In order to illustrate this: a term
$r^{n+2/3}$ would be replaced by $X^6Y^2$.)
Now CAD can be applied under the constraints $0<X\leq
Y<1$, and it yields the claimed result.

\smallskip
(4) $(1-r^{n/6})a_2+\kappa (1-r)(1-r^{1/6})a_3\geq0$. The proof is
completely analogous to the proof of $a_2\geq0$ above: we write 
$$(1-r^{n/6})a_2+\kappa (1-r)(1-r^{1/6})a_3=\sum_{j=0}^4n^jb_j,$$ 
and verify by CAD that $b_0, b_2, b_4, (1-r^{1/6})b_1+(1-r^{n/6})b_2$
and 
$$\sum_{j=0}^4 (1-r^{1/6})^{4-j}(1-r^{n/6})^jb_j$$ 
are non-negative.

\smallskip
  With these non-negativity results proven, 
the inequality \eqref{eqIneqCosSumRational} follows from the fact that
  \[
  t\leq\tan^2\left(\frac{\pi}{2n+2}\right)\leq \frac1n\leq\frac{1-r^{1/6}}{1-r^{n/6}}.
  \]

\medskip
{\sc Part II.} 
$\theta>\frac{\pi}{n+1}$. We apply the Cauchy--Schwarz inequality to
the vectors\break $(r-\cos\theta,\sin\theta)$ and $(\cos n\theta,\sin
n\theta)$. This yields
$$
(r-\cos\theta)\cos n\theta+\sin\theta\sin n\theta\le
\sqrt{(r-\cos\theta)^2+\sin^2\theta}\cdot 1,
$$
which is equivalent to
  \begin{equation}\label{eqCauchy}
  r\cos(n\theta)-\cos(n\theta+\theta)\leq \sqrt{1-2r\cos\theta+r^2}.
  \end{equation}
Equality in \eqref{eqCauchy} holds if and only if the two vectors are
proportional to each other, that is, if and only if
$$
\frac {r-\cos\theta} {\sin\theta}=\frac {\cos n\theta} {\sin
  n\theta}=\cot n\theta.
$$
We define the quantity
  \[
  n_0(\theta,r)=\frac{1}{\theta}\left(\frac\pi2-\arctan\frac{r-\cos\theta}{\sin\theta}\right)\in\left[\frac{\pi-\theta}{2\theta},\frac{\pi-\theta}{\theta}\right].
  \]
From the above observation, it follows readily that
we have equality in \eqref{eqCauchy} for $n=n_0(\theta,r)$.

  We now claim that the strengthened inequality
  \begin{equation}\label{eqIneqCosSum3}
    \frac{-r+\cos\theta+s\sqrt{1-2r\cos\theta+r^2}}{1-2r\cos\theta+r^2}\leq \frac{1-s}{1-r}\sqrt{\frac{1}{1+4\kappa^*\tan^2(\theta/2)}},
  \end{equation}
  holds in the region
  \[
  \left\{(r,s,\theta):r,s\in[0,1),\ 0\leq\theta<\pi,\ s\leq r^{\max(1,n_0(\theta,r))}\right\},
  \]
where $\kappa^*$ is defined by
$$
\kappa^*:=\frac{(1+r)(1-s)(1-s^{1/6})}{(1-r)^2}.
$$
If we assume the validity of this inequality, then
 the desired result follows by choosing $s=r^n$ in \eqref{eqIneqCosSum3}, and applying \eqref{eqCauchy}; we point out that, since $n_0(\theta,r)\leq\pi/\theta-1<n$, our desired value of $s=r^n$ indeed belongs to the region.

In order to prove \eqref{eqIneqCosSum3}, first note that the
left-hand side of \eqref{eqIneqCosSum3} is linear with respect
to~$s$. Furthermore, computation of the second derivative of the
right-hand side shows that it is concave with respect
to~$s$. Therefore it suffices to prove \eqref{eqIneqCosSum3} for the
values of $s$ on the boundary --- that is, for $s=0$ and $s=r^{\max(1,n_0(\theta,r))}$. We write $c:=\cos\theta$ for simplicity of notation.

\smallskip
(1) $s=0$. In this case, the inequality \eqref{eqIneqCosSum3} reduces to
    \[
    \frac{c-r}{1-2rc+r^2}\leq \sqrt{\frac{1}{(1-r)^2+4(1+r)\frac{1-c}{1+c}}}.
    \]
    This inequality clearly holds if $c\leq r$. If $r<c\leq 1$, then we have
    \begin{multline*}
    \frac{1}{(1-r)^2+4(1+r)\frac{1-c}{1+c}}-\frac{(c-r)^2}{(1-2cr+r^2)^2}\\
    =\frac{(1-c)^2(1+r)^2(1+3c-2r)}{(1-2cr+r^2)^2((1-c)(1+r)(5-r)+2(c-r)(1-r))}\geq0.
    \end{multline*}

\smallskip
(2) {\sc $s=r$ and $n_0(\theta,r)\leq 1$.} 
Elementary manipulations reveal that the inequality for $n_0$ is
equivalent to $r\geq2c$. Moreover, the equality $s=r$ implies that
        \[
        \kappa=\frac{(1+r)(1-r^{1/6})}{1-r}\leq \frac{1+r}{1+\sqrt{r}}\leq1.
        \]
        So it suffices to prove that
        \begin{equation}\label{eqRC1}
        \frac{c-r+r\sqrt{1-2rc+r^2}}{1-2rc+r^2}\leq\sqrt{\frac{1}{1+4\frac{1-c}{1+c}}}=\sqrt{\frac{1+c}{5-3c}}
        \end{equation}
        holds for $r\in[0,1]$ and $c\in[-1,r/2]$. We argue that the left-hand side of \eqref{eqRC1} is increasing with respect to~$r$ for $r\in[\max(0,2c),1]$ because of
        \begin{equation*}
        \frac{\partial}{\partial r}\frac{c-r+r\sqrt{1-2rc+r^2}}{1-2rc+r^2}
        =\frac{(1-cr)\sqrt{1-2cr+r^2}-(1-c^2-(r-c)^2)}{(1-2cr+r^2)^2},
        \end{equation*}
        and that
        \[
        (1-cr)^2(1-2cr+r^2)-(1-c^2-(r-c)^2)^2=(1-c^2)(r-2c)(3r-2c-r^3)\geq0.
        \]
        Therefore we have
        \begin{align*}
          \frac{c-r+r\sqrt{1-2rc+r^2}}{1-2rc+r^2}\leq \frac{1}{\sqrt{2-2c}}-\frac12\leq\frac{1+c}{3}\leq\sqrt{\frac{1+c}{5-3c}},
        \end{align*}
        as desired.

\smallskip
(3) {\sc $s=r^{n_0(\theta,r)}$ and $n_0(\theta,r)\geq 1$}. 
We recall that \eqref{eqCauchy} holds for $n=n_0(\theta,r)$. This
means that \eqref{eqIneqCosSum3} is equivalent to the special case of \eqref{eqIneqCosSum2} where $n$ is replaced by $n_0(r,\theta)$. Since we have $n_0\leq\pi/\theta-1$ and therefore $\theta\leq \pi/(n_0+1)$, we invoke the result of the first part to conclude the proof.
\end{proof}

The following inequality proves that a Pad\'e approximant of
$\cos(n\theta)-\cos(n\theta+\theta)$ is a lower bound in a small
interval around~0. 

\begin{lemma}\label{leIneqChebyshev2}
For $n\geq1$ and $\theta\in[-\pi/n,\pi/n]$, we have
\begin{equation}\label{eqIneqChebyshev2}
\cos(n\theta)-\cos(n\theta+\theta)\leq \frac{6(2n+1)}{3\cot^2(\theta/2)+2n^2+2n+3}.
\end{equation}
\end{lemma}
\begin{proof}
Without loss of generality assume that $\theta\in[0,\pi/n]$.
If $\theta>2\pi/(2n+1)$ then the left-hand side of
\eqref{eqIneqChebyshev2} is negative and there is nothing to
prove. Otherwise let $\phi:=(2n+1)\theta/2\in[0,\pi]$ and $m:=2n+1$. 
By elementary manipulations, we see that the inequality \eqref{eqIneqChebyshev2} is equivalent to
\[
m\sin\frac{\phi}{m}\geq \left(1+\frac{m^2-1}{6}\sin^2\frac{\phi}{m}\right)\sin\phi.
\]
We use the fact that $\sin^2({\phi}/{m})\leq(\phi/m)^2$ to observe that it suffices to prove
\[
m\sin\frac{\phi}{m}\geq \left(1+\frac{m^2-1}{6m^2}\phi^2\right)\sin\phi.
\]
This is evidently an equality if $m=1$. 
We claim that the difference between the two sides is increasing with respect to~$m$. Indeed, we have
\begin{align*}
\frac{\partial}{\partial
  m}\left(m\sin\frac{\phi}{m}-\left(1+\frac{m^2-1}{6m^2}\phi^2\right)\sin\phi\right)
&=\sin\frac{\phi}{m}-\frac{\phi}{m}\cos\frac{\phi}{m}-\frac{\phi^2}{3m^3}\sin\phi\\
&\geq\sin\frac{\phi}{m}-\frac{\phi}{m}\cos\frac{\phi}{m}-\frac{\phi^2}{3m^2}\sin\frac{\phi}{m}\\
&=\frac13\int^{\phi/m}_0 t(\sin t-t\cos t)\,dt\geq0.
\qedhere
\end{align*}
\end{proof}

\subsection{A decreasing function}

The following technical lemma is of crucial importance in the proof of
the monotonicity property in Lemma~\ref{leIneqTail}.

\begin{lemma}\label{leMonotone1}
For $\lambda>0$ and $n\geq6+36/\lambda$, the function
\[
\frac{1-r^n}{1-r}\exp\left(-\lambda\frac{1-r^{n/6}}{1-r}\right)
\]
is decreasing with respect to~$r$ in the interval $(\exp(-8\lambda/9),1)$.
\end{lemma}
\begin{proof}
By taking logarithmic derivatives with respect to~$r$, we see that it suffices to prove that
\[
\frac{\partial}{\partial r}\log \frac{1-r^n}{1-r}\leq\lambda\frac{\partial}{\partial r}\frac{1-r^{n/6}}{1-r}.
\]
For the left-hand side, we have
$$
\frac{\partial}{\partial r}\log \frac{1-r^n}{1-r}
\le \frac{(1-r^n+r^n\log(r^n))}{(1-r)(1-r^n)}
$$
(which, after simplification, turns out to be equivalent to the obvious
$-\log r^{-1}\ge 1-r^{-1}$), and for the right-hand side (without $\lambda$
and with $n$ replaced by~$6n$)
$$
\frac{\partial}{\partial r}\frac{1-r^{n}}{1-r}
\ge \frac{(1-r^n)(1-r^{(n-1)/2})}{(1-r)^2}
$$
(which, after simplification, turns out to be equivalent to
the easily derived inequality 
$n\le r^{-(n-1)/2}+r^{-(n-3)/2}+\dots+r^{(n-1)/2}$).
Therefore, it suffices to prove that
\[
\frac{(1-r^n+r^n\log(r^n))}{(1-r)(1-r^n)}\leq \lambda\frac{(1-r^{n/6})(1-r^{(n-6)/12})}{(1-r)^2},
\]
or, equivalently,
\[
\frac{(1-r^{n/6})(1-r^{(n-6)/12})(1-r^n)}{(1-r^n+r^n\log(r^n))(1-r)}\geq
\frac 1{\lambda}.
\]
We write $s:=r^{n-6}$.
It is not difficult to show that the function
$x\mapsto\frac{(1-x)(1-x^{1/6})}{1-x+x\log x}$ is decreasing for
$x\in(0,1)$. Since $s=r^{n-6}\geq r^n$, this observation implies that
\begin{align*}
\frac{(1-r^{n/6})(1-r^{(n-6)/12})(1-r^n)}{(1-r^n+r^n\log(r^n))(1-r)}&\geq \frac{(1-s^{1/6})(1-s^{1/12})(1-s)}{(1-s+s\log s)(-\log r)}\\
&=(n-6)\frac{(1-s^{1/6})(1-s^{1/12})(1-s)}{(1-s+s\log s)(-\log s)}.
\end{align*}
Therefore it remains to prove that 
\begin{equation} \label{eq:h(s)} 
\frac{(1-s^{1/6})(1-s^{1/12})(1-s)}{(1-s+s\log s)(-\log s)}\geq\frac{1}{\la(n-6)}
\end{equation}
for $s\in(e^{-8\lambda/9},1)$. 
Let $h(s)$ denote the left-hand side of \eqref{eq:h(s)}.
The function $s\mapsto h(s)$, for $s\in(0,1)$, 
equals $0$ for $s\to0^+$ (due to the term $-\log s$ in the denominator), it equals
$1/36$ for $s\to1^-$, it is increasing at the beginning, has a unique
maximum at (numerically) $s=0.00003158\ldots=e^{-10.3629\dots}$
(with value $h(0.00003158\ldots)=0.0459021\dots$),
and from there on is decreasing. Since, by assumption, we have
$\lambda(n-6)\ge
36$, the inequality \eqref{eq:h(s)} will  be satisfied
on an interval of the form $[y,1]$, with $y$ depending on~$\lambda$ and~$n$. 

We have $h(10^{-12})=0.0322464\ldots>\frac {1} {36}=0.02777\dots$. Since $10^{-12}$
is smaller than the place of the unique maximum of $h(s)$, this implies
\begin{equation} \label{eq:hUngl} 
h(s)\ge \frac {1} {36}\ge\frac {1} {\lambda(n-6)},\quad \quad 
\text{for }s\in(10^{-12},1).
\end{equation}
In order to get an estimate for~$y$, we observe that the function
$s\mapsto h(s)(-\log s)$, that is,
$$
s\mapsto 
\frac{(1-s^{1/6})(1-s^{1/12})(1-s)}{(1-s+s\log s)},$$
is decreasing for $s\in(0,1)$. Its value at $s=10^{-12}$ is
$0.891\dots>\frac {8} {9}$.
Therefore, we have
$$
h(s)\ge \frac {8} {9}\frac {1} {(-\log s)},\quad \quad 
\text{for }s\in(0,10^{-12}).
$$
If we now choose $y=e^{-\frac {8} {9}\lambda(n-6)}$, then we have
$$
h(s)\ge \frac {8} {9}\frac {1} {(-\log s)}\ge \frac {1} {\lambda(n-6)},\quad \quad 
\text{for }s\in(y,10^{-12}).
$$
Together with \eqref{eq:hUngl} and the fact that $n\ge7$ by assumption, 
we have proven \eqref{eq:h(s)} and thus the lemma.
\end{proof}

\subsection{A cosine inequality}
The elementary cosine estimate below is needed in the proofs of
Theorems~\ref{thMain1}, \ref{thMain2}, and~\ref{thMain3}.

\begin{lemma} \label{lem:cos-cos}
For $x\in [-\pi/6,0]$ and all integers $m$, we have
\begin{equation} \label{eq:cos-cos} 
\left|\cos\left(x-2m\pi/3\right)\right|
\ge\begin{cases} 
\frac {1} {2},&\text{for }m\equiv0,1~\text{\em(mod $3$)},\\
\left|\cos\left(\pi/3-x\right)\right|,&\text{for }m\equiv2~\text{\em(mod $3$)}.
\end{cases}
\end{equation}
\end{lemma}

\begin{proof}
We distinguish the congruence classes of~$m$ modulo~$3$.
If $m\equiv0$~(mod~$3$), then we have
\begin{equation} \label{eq:cos-ineq} 
\left|\cos\left(x-2m\pi/3\right)\right|
=
\left|\cos\left(x\right)\right|.
\end{equation}
The claim on the right-hand side of \eqref{eq:cos-ineq} is then
straightforward to verify.
The case where $m\equiv1$~(mod~$3$) can be treated similarly. 
On the other hand, for $m\equiv2$~(mod~$3$) we actually have
\begin{equation*}
\left|\cos\left(x-2m\pi/3\right)\right|
=
\left|\cos\left(\pi/3-x\right)\right|.
\qedhere
\end{equation*}
\end{proof}


\end{document}